\documentclass[11pt,a4paper,leqno]{amsart}
\usepackage{amsmath,amsfonts,amssymb,amsthm, color}
\usepackage[T1]{fontenc}
\usepackage{verbatim}

\usepackage[active]{srcltx}
\usepackage{mathrsfs} 
\advance\textwidth by 2.5cm %%3.6
\advance\oddsidemargin by -1.0cm \advance\evensidemargin by -1.0cm

\newcommand\del[1]{}
\newcommand\think[1]{}
\newcommand\new[1]{}
\newcommand\zus[1]{}

\newcommand\comd[1]{} %% Correction agreed upon and unseen
\newcommand\Redd[1]{} %% Correction agrred upon  and unseen

\def\bdm{\begin{displaymath}}
\def\edm{\end{displaymath}}
\def\bea{\begin{eqnarray}}
\def\eea{\end{eqnarray}}

\newtheorem{theorem}{Theorem}[section]

\newtheorem{lem}[theorem]{Lemma}
\newtheorem{defn}[theorem]{Definition}
\newtheorem{prop}[theorem]{Proposition}
\newtheorem{coro}[theorem]{Corollary}

\newtheorem{remark}{Remark}

\numberwithin{equation}{section}

\begin{document}

%\date{\today}

\title[FSBE in H\"older space]
{Fractional stochastic Burgers-Type Equation in H\"older space\\-\del{Existence, Uniqueness}wellposedness and approximations-}

\author[Zineb Arab  \& Latifa Debbi]{Z. Arab \& L. Debbi.}

\email{ldebbi@yahoo.fr, zinebarab@yahoo.com}

\address{Department of Mathematics, Faculty of mathematics and computer Sciences, University Batna 2, Batna 05000, Algeria. \& LaMa laboratory, Faculty of Sciences, University Ferhat Abbas, El-Maabouda Setif 19000, Algeria.}
\address{National Polythechnic School, BP 182 El-Harrach 16200.  Algiers, Algeria.\\}

\begin{abstract}
In this work, we use the spectral Galerkin method to prove  the existence of a pathwise unique mild solution of a fractional stochastic partial differential equation of Burgers type in a H\"older space. We get the temporal regularity  and using a combination of Galerkin and exponential-Euler methods, we obtain  a fully discretization scheme of the solution. Moreover, we calculate the rates of convergence for both approximations (Galerkin and fully discretization) with respect to time and to space.

\vspace{0.15cm}


\vspace{0.15cm}


\end{abstract}

\maketitle

\section{Introduction} 
Recently, the field of numerical approximations for stochastic differential equations, in finite and infinite dimensions,\del{ordinary and partial,} attracts more and more attension due, not only to its importance to solve problems but also due to the strange phenomena particullarly emerging in this case. For example in \cite{Hairer-AP-N-SDE}, Hairer et al. constructed a stochastic differential equation for which, nevertheless, the rate 1 is well known for  the Euler approximation for the deterministic version, the Euler approximation for the stochastic version converges to the solution in the strong and in the numerically weak sense without any  arbitrarily small polynomial rate. In \cite{Hairer-Voss}, the authors showed that different finite-difference schemes to stochastic Burgers equation driven by space-time white noise converge to different limiting processes. Divergence of schemes can also occur if the stochastic noise is rougher than the space-time white noise. 
One of the explanation for such strange behaviours is  the loss of regularity of the solutions  of the stochastic differential equations, which arises due to the roughness of the random noise. This loss of regularity even yields for some cases to the illposedness of the equations,  see \cite{Hairer-Matetski-16, Hairer-AP-N-SDE, Hairer-Weber-13, Hairer-Voss}. Our present work makes part of this direction of study and for us the loss of regularity is more complicated as it arises not only from the randomness but also from the structure of the equations themselves. In fact, we consider a class of the fractional version of the nonlinear stochastic  Burgers-type  equations studied in \cite{Hairer-Matetski-16,  Hairer-Weber-13, Hairer-Voss}.

\vspace{0.25cm} 

Before describing our class of equations, let us show some of the delicate properties of the fractional  stochastic partial differential equations, such as the wellposedness and the kind of roughness of the solutions  caused by the fractional operators. In \cite{DebbiDozzi1}, the authors proved that the fractional stochastic heat-type equation admits a unique mild solution if the dissipation index $\alpha>1$ (see Eq. \eqref{FSBE-Evol-1} and Notation below). For $\alpha\in (0, 1)$, we can not get a function but only a distribution solution.  Moreover,  for $\alpha\in (1, 3)$, the trajectories live in  the space-time H\"older continuous functions $ C_t^{\min\{\frac{\alpha-1}{2\alpha}, \frac{\alpha-m}{\alpha}\}-}C_x^{\min\{\frac{\alpha-1}{2}, \alpha-[\alpha]\}-}$, where $ [\alpha]$ is the integer part of $\alpha$ and $m$ is the highest order of the entier derivative ($\alpha<m$).  The classical heat equation corresponds to $\alpha=2$ and $m=0$. These results show that the solutions of the fractional  stochastic partial differential equations suffer of more roughness than their counterpart in the classical case. Another, specific diffuclty for the numerical approximation of the fractional stochastic partial differential equations is due to the fact that the fractional operator is  nonlocal. Hence, to the contrarly to the second order differential operators, to apply the finite-difference method we need more than tree points and all the points of the grid have to be used in every step. This fact yields to 
the slow convergence or to the divergence of the schemes. Moreover, it is not easy to find a concrete form to the discretized fractional operator, see for more discussion and results  \cite{DebbiDozzi2}.

\vspace{0.25cm}

For the nonlinear fractional stochastic partial differential  equations, such as  Burgers equation, the fractional dissipation is not even strong enough to controle the steepening of the nonlinear term. For example, the blow up of solutions of the superctical regimes ($\alpha<1$) of Burgers and nonlocal velocity transport equations 
has been proved in \cite{Kiselev-Nazarov-Fractal-Burgers-08} respectively in \cite{cordoba-cordoba-Fontelos-05}, in \cite{Alibaud-no-unquness-burgers}  the authors proved the non uniqueness of the weak solution
of the supercritical fractional Burgers equation,
the H\"older regularity of the solution of
the critical ($\alpha=1$) 2D-quasi-geostrophic equation has been obtained by Caffarelli and Vasseur \cite{Caffarelli-Vasseur2010} but this problem is still open  for the superctical regime.

\vspace{0.25cm}

In the present work, we  are interested in the class of fractional stochastic Burgers-type equations (FSBTE) given by the evolution form  
\begin{equation}\label{FSBE-Evol-1}
\left\{
\begin{array}{lr}
du(t)=[-A^{\alpha/2} u(t) + F(u(t))]dt+ dW(t),\;\; t\in [0, T],\\
u(0)=u_0,
\end{array}
\right.
\end{equation}
where $ A^{\alpha/2}$ is the fractional power of the Laplacian  $A=-\Delta $ with Dirichlet boundary conditions, $F$ is a nonlinear operator given by $ F(u(t, x)):= \partial_xf(u(t, x))$, $x\in (0, 1)$\del{ $f$ being a function satisfying}  see e.g. \cite{Cardon-Weber-Millet, Duan-Lv-conservation-law, Gyongy-98-B-type, Gyongy-Nualart-Burgers-99} and \cite{Hairer-Matetski-16,  Hairer-Weber-13, Hairer-Voss} for the no gradient case,  $W$ is a Wiener process and  $u_0$ is a $ L^2-$value random variable.  These  equations represent typical examples for locally Lipschitz  nonlinear growth nonlinear equations. The fractional respectively the classical stochastic  Burgers-type  equations  are recupared if  $f(u):=u^2$ respectively if  $\alpha=2$. To the contrary to the fractional and Burgers-type equations, the classical stochastic Burgers equation has been analytically and numerically extensively studied, see for short list, \cite{Alabert-Gyongy, Blomker-Jentzen, DapratoDebusscheTemam94, DapratoZabczyk96, Jentzen-Pertub-Arx, PeszatZabczykBook07, Printems} and the references therein. Fractional Burgers equation has been introduced as a relevant model for anomalous diffusions  such as;  diffusion in complex phenomena, relaxations in viscoelastic meduims, propagation of acoustic waves in gaz-filled tube, see e.g. \cite{SugKak, Sug,  Sug-Frac-cal-89}. The analytic study  for the deterministic fractional Burgers equation has been ivestigated e.g. in \cite{Alibaud-no-unquness-burgers, Caffarelli-Vasseur2010, Kiselev-Nazarov-Fractal-Burgers-08}.  The wellposedness of the  $L^2-$solution and the ergodic properties  of the fractional stochastic Burgers equation have been obtained in \cite{BrzezniakDebbiGoldy, BrzezniakDebbi1}, with the fractional dissipation index: $\alpha>\frac32$, see also \cite{TrumanWu-06}. The numerical study for the fractional deterministic equation is still a modest field due to the difficulties mentioned above (e.g.  difficulties to approximate the fractional operator and to control the nonlinearity). \del{One of the works known for the authors is }In \cite{Woyczynski-Num-approx-fract-05} the authors partially circuvent the first difficulty, by using the Monte Carlo method to approximate numerically the solution of some deterministic fractional partial differential equations, among them the Burgers equations. In \cite{Gugg--Niggemann}, the authors considered one dimensional stochastic hyperdissipation  Burgers equation, with $-A^{\alpha/2}$ being replaced by $(-1)^{p+1}\nu\frac{\partial^{2p}}{\partial x^{2p}}, p\geq 2, \nu>0$ and $W$ being a colored noise.\del{ They first discretized the random noise with respect to time and then used the Galerkin approximation.}  Basing on the  discretazation of the random noise, the authors  constructed an approximation for the solution  of the equation and proved that  it converges to the solution  in $L^m(\Omega;
L^\infty(0, T)\times L^2(0, 1)), \; m\in \mathbb{N}$ and in  $L^1(\Omega; L^2(0, T)\times L^\infty(0, 1))$.

\vspace{0.25cm}

To the best knowldge of the authors, G\"oyngy was the first to introduce stochastic Burgers-type equation  \cite{Gyongy-98-B-type} than several works followed, see e.g.  \cite{Cardon-Weber-Millet, Gyongy-Nualart-Burgers-99,  Xie-frac-Burg}. 
Other kind of generalization has been investigated  by the second author in \cite{Debbi-scalar-active}. In this work, the author introduced a new class of multi-dimensional stochastic active scalar equations covering among others the multidimensional quasi-geostrophic, $2D$-Navier-Stokes and the fractional stochastic Burgers equation on the torus with $\alpha\in (1, 2]$.\del{ , for the sub-super and critical regimes. The methods are also valid for the  fractional stochastic Burgers equation on bounded domain. In particular,} The author established  thresholds guaranting, according to the kind of solutions; strong, weak, martingale and to the Sobolev and the integrability regularities requered, the wellposedness of these equations.  A generalization of this class of equations on $D\subseteq \mathbb{R}^d$ and their wellposedness have also been  studied  in \cite{Debbi-scalar-active-R-d}. \del{ and for the fractional stochastic Burgers equation e.g.  in \cite{BrzezniakDebbiGoldy, BrzezniakDebbi1, } and as a special case for the results obtained in \cite{ Debbi-scalar-active, Debbi-scalar-active-R-d}. In \cite{Xie-frac-Burg} the authors studied the FSBE driven by L'evy noise.}

\vspace{0.25cm}

The global picture of the numeical study for  stochastic partial differential equations, is to elaborate schemes
providing approximations with respect to time,  to  space or to both simultanoeusly, to prove convergence of these schemes and to establish the rate of convergence. The classical results state that  space or time discretization schemes convergence in expectation in the case the coefficients are globaly Lipschitz and/or have linear growth property. In the cases the coefficients are only locally Lipschitz or have nolinear growth, then only weak convergence has been proved, see e.g. \cite{Gyongy-Nualart, Printems, Gyongy-Krylov-10, Gyongy-Millet-09, Gyongy-Millet-07,  Kruse-Optimal-14} for time discretization and see e.g. 
\cite{Gyongy-Krylov-10, Gyongy-Millet-09, Gyongy-Millet-07, Gyongy-Nualart, Kruse-Optimal-14} for space and fully discretizations. 
One of the first results about the pathwise convergence for the stochastic Burgers equation known  to the authors is the work  \cite{Alabert-Gyongy}. In this work the authors used  the finite difference method and proved that the  discretized trajectories \del{pathwise }converge almost surely to the solution in $C_tL^2-$topology with rate $\gamma<1/2$.  

\vspace{0.25cm}

Recently,  a new tendency for the numerical study of stochastic differential equations has been developed based on the idea to elaborate numerical approximations for an abstract stochastic differential equation with coefficients satisfying some conditions and than to show that this study covers some specific equations.  For example,  in \cite{Jentzen-Pertub-Arx}, the authors developed a perturbation theory for finite and infinite stochastic differential equations where the main idea is estimate the error between the solution of the stochastic differential equation and an Ito stochastic process. This last process is considered as a perturbation of the solution. As Galerkin approximation can be regarded as such perturbation, the authors applied this method on classical stochastic Burgers equation with colored noise. In particular, they established the rate of the time uniform convergence with respect to the  $L^m(\Omega; L^2(0, 1))-$norm
with a given $m\in \mathbb{R}$. In \cite{Blomker-Jentzen}, the authors used Galerkin approximation to prove the  wellposedness of an abstract evolution stochastic differential equation and calculated the rate of convergence to the solution in abstract spaces. The authors applied this abstract theory on; the multi dimentional heat equation, reaction diffusion equation and on classical Burgers equation deriven by additive space-time white noise.  In particular, they proved that the rate of the pathwise  convergence of the Galerkin approximation is of order $\gamma<1/2$ in the $C_tC_x-$topology. This result gives an improvement to  Alabert-Gy\"ongy's estimates cited above, see \cite{Alabert-Gyongy}.

\vspace{0.25cm}

Our aim in this work is to prove the wellposedness  of space-time H\"older solutions of  fractional stochastic Burgers-type equations;  Eq.\eqref{FSBE-Evol-1} and to establish the rate of convergence of the Galerkin approximation and of the fully discretization schemes with respect to space and to time. To the best knowledge of the authors, the current work is the first proving  these results not only for the fractional stochastic Burgers-type equations but also for the fractional and classical stochastic Burgers equations. In \cite{Blomker-Jentzen} the wellposedness  and the Galerkin approximation have been obtained for the stochastic classical Burgers-type equations in the space of continuous functions $ C^0.$ The exponential Euler scheme method has been applied here for the first time for the fractional stochastic equations. Recall that this method has been introduced in \cite{Jentzen-Kloeden} and used in the approximation of  the stochastic heat and reaction diffusion equations,  see \cite{Jentzen-Kloeden-Winkel, Jentzen-Kloeden}. To  elaborate the fully disretization scheme, we combined the spectral Galerkin method and a version of the  exponential Euler scheme. Furthermore, we have established conditions to prove the existence of the Galerkin approximation and estimates for the products in specific Sobolev and H\"older spaces. In fact, as it is discussed in \cite{Debbi-scalar-active, Debbi-scalar-active-R-d,   Debbi-FNS}, the Sobolev spaces used in the study of  fractional stochastic partial differential equations are  large, e.g. of  index $\alpha/2, 1-\alpha/2 <1$, for $\alpha<2$. Thus, these fractional Sobolev spaces  are not algebras with respect to the product even for dimension 1 and the solutions are not smooth.

\del{\vspace{1cm}
%%%%%%%%%%%%%%%%%%%%%%%%%%%%%%%%%%%%
The main difficulties to apply the techniques in \cite{Blomker-Jentzen} for the fractional stochastic partial differential equations are related to choice of the relevant spaces; fractional Sobolev, H\"older spaces and to the getting the required estimates.

 of them are not classical,  In this work, we chose the spaces $V, U$ in Theorem \ref{Theorem-Blomker-Jentzen} below to be  $V:= C^{\delta}(0,1)$, with a specific $\delta$, $ U:=H^{1-\alpha/2}_2(0, 1)$.

prove the existence and the uniquiness of a mild solution in the space H\"older space.

Let us first of all mention that the works above do not cover the fractional Burgers-type equation \eqref{FSBE-Evol-1}. Some of the difficulties to apply the above works for Eq.\eqref{FSBE-Evol-1},

Let us recall that the technique of using Galerkin method to prove the existence of the solutions of some specific equations is a classical method and it has been widely used either for deterministic or  stochastic partial differential equations, see e.g. \cite{DaPrato-Debussche-Cahn-96, Flandoli-Gatarek-95, Temam-NS-Main-79}.
%%%%%%%%%%%%%%%%%%%%%%%%%%%%%%%%%
}
\vspace{0.25cm}

We present our results in the following plane. In  Section \ref{sec-formulation}, we introduce the ingredients of our main problem and we study some of their properties and assumptions. Our main results are given in Section \ref{sec-Def-Results}. Section \ref{sec-Est-Stoch-Terms} is devoted to the study of the Ornstein-Uhlenbeck stochastic process defined via the fractional semi group. We study the wellposedenss and the properties of the solutions of pathwise deterministic fractional equations of Burgers type in Section \ref{sec-deter-Burgers-Eq}. The proofs of our results are given in Section \ref{sec-Proofs}. In Appendix \ref{sec-Term-Basic-results}, we define the functional spaces and we present the results  we are using in our proofs.\del{, for more general definitions and properties, see e.g.  \cite{Adams-Hedberg-99, Runst-Sickel, Triebel-83}. In  Section \ref{sec-}} Regarding the importance of the results in \cite{Blomker-Jentzen}, we end this section by recalling below \cite[Theorem 3.1.]{Blomker-Jentzen}.
\begin{theorem}\label{Theorem-Blomker-Jentzen}\cite[Theorem 3.1.]{Blomker-Jentzen}
 Let $T$ be fixed, $ V, U$ be two $ \mathbb{R}-$Banach spaces and let $ P_N: V\rightarrow V$ be a sequence of linear bounded operators.

Assume that the following assumptions are fulfilled:\del{{\color{red} Please write the assumptions}, }
\begin{itemize}
\item \textbf{Assumption 1.} Let $ S: (0, T] \rightarrow \mathcal{L}(U,V)$ be a continuous map satisfying 
\begin{equation}\label{Eq-Assum-1-1}
\sup_{t\in (0, T]}\big( t^\alpha|S(t)|_{\mathcal{L}(U, V)}\big)<\infty, 
\end{equation}
\begin{equation}\label{Eq-Assum-1-2}
\sup_{N\in \mathbb{N}}\sup_{t\in (0, T]}\big(N^\gamma t^\alpha|S(t)-P_NS(t)|_{\mathcal{L}(U, V)}\big)<\infty, 
\end{equation}

where $ \alpha \in [0, 1)$ and $ \gamma \in (0, \infty)$ are given constants.
\item \textbf{Assumption 2.} Let $ F: V \rightarrow U$ be a maping which  satisfies 
\begin{equation}\label{Eq-Assum-2}
\sup_{|u|_{V}, |v|_{V}\leq r\\ u\neq v}\frac{|F(u)-F(v)|_{U}}{|u-v|_{V}}<\infty. 
\end{equation}

\item \textbf{Assumption 3.} Let $ O: [0, T]\times \Omega \rightarrow V$ be a stochastic process with continuous simple paths and  
\begin{equation}\label{Eq-Assum-3}
\sup_{N\in \mathbb{N}}\sup_{t\in (0, T]}\big(N^\gamma \del{t^\alpha}|(1-P_N)O_t(\omega)|_{V}\big)<\infty, \;\; \text{for every}\; \omega, 
\end{equation}
where  $ \gamma \in (0, \infty)$ are given in Assumption 1.

\item \textbf{Assumption 4.} Let $ X^N: [0, T]\times \Omega \rightarrow V, N\in \mathbb{N}$ be a sequence of stochastic processes with continuous simple paths and  with 
\begin{equation}\label{Eq-Assum-4-1}
\sup_{N\in \mathbb{N}}\sup_{t\in (0, T]}\big(|X^N_t(\omega)|_{V}\big)<\infty, 
\end{equation}

\begin{equation}\label{Eq-Assum-4}
X_t^N(\omega)= \int_0^t P_NS(t-s)F(X_s^N(\omega))ds + P_N(O_t(\omega)),
\end{equation}
for every  $ \omega \in \Omega, \;\; t \in [0, T]$ and every $N\in \mathbb{N}$.
\end{itemize} 
Then there exists a unique stochastic process $ X: [0, T]\times \Omega \rightarrow V,$ with continuous simple paths such that  

\begin{equation}
X_t(\omega)= \int_0^t S(t-s)F(X_s(\omega))ds + O_t(\omega),
\end{equation}
for every  $ \omega \in \Omega, \;\; t \in [0, T]$.  Moreover, there exists a $\mathcal{F}/\mathcal{B}(0, \infty)-$measurable mapping 
$ C: \Omega \rightarrow [0, \infty),$ such that 
\begin{equation}
\sup_{t\in (0, T]}|X_t(\omega)-X^N_t(\omega)|_{V}\leq C(\omega). N^{-\gamma}, 
\end{equation}
holds for every $N\in \mathbb{N}$ and every  $ \omega \in \Omega$, where $ \gamma$  is given in Assumption 1.

\end{theorem}   

{\bf Notations.} Let $ \mathbb{N}_0:=\mathbb{N}-\lbrace 0 \rbrace $. For $ 1 < p < \infty $, we say that $ q $ is the conjugate of $ p $ if $ \frac{1}{p}+\frac{1}{q}=1 $ and for $ p=1 $ (resp. $ p=\infty $) $ q=\infty $ (resp. $ q=1 $). 
By a domain $ D $ we mean a non empty open set. Here, we briefly give the notations of the functional spaces defined on $ D $ used in this paper, the complete definitions and some of the properties will be presented in Appendix \ref{sec-Term-Basic-results}. For $ 1 \leq p \leq \infty $ we denote the Lebesgue space by $L^p(D)$, the Sobolev space respectively the fractional Sobolev space (also called Aronszajn, Gagliardo or Slobodockij space) are denoted by $ W^{m}_{p}(D) $ for $ m \in \mathbb{N} $ and by $ W^{s}_{p}(D) $ for $ 0 < s \neq \text{integer} $. In this work we are interested in the special case $ p=2 $ where $  W^{s}_{2}(D) $ is Hilbert space, for this the shorter notation $  H^{s}_{2} $ will be used. Finally, the notation $ B_{pq}^{s}(D) $, $ s \in \mathbb{R} $ and $ 0 < p,q \leq \infty $ is reserved for Besov space.  $C(D)$ is the space of bounded continous functions, the space of H\"older continous functions of exponent $ \delta \in (0,1) $ is denoted by $ C^\delta(D)$ and the notation $ \mathcal{C}^\delta(D)$ is reserved for Zygmund space. Let $ X $ and $ Y $ be two Banach spaces, we use the notation $ \vert . \vert_{X} $ to indicate the norm in $ X $ and we denote by $ \mathcal{L}(X,Y)$ the space of linear bounded operators defined on $ X $ into $ Y $ endowed with the norm $ \Vert  .\Vert_{\mathcal{L}(X,Y)} $.  The abbreviations (FSBE) and (FSBTE) are used respectively for fractional stochastic Burgers equation and the fractional stochastic Burgers type equation. In the end, let us mention that, the values of the constants may change from line to line and we sometimes delete their dependence on parameters.

\del{{\bf Notation.}  Here, we briefly, give  the notation used in this paper. The complete definitions and some of the properties will be presented in Appendix \ref{sec-Term-Basic-results}.  $ \mathbb{N}_0:=\mathbb{N}-\lbrace 0 \rbrace $. 
By a domain $ D \subset \mathbb{R}^d, \;\; d\in \mathbb{N}_0$,  we mean a non empty connected open set. We assume that the boundary of $D$ are at least Lipschitz.  For $ 1 \leq p < \infty$,  we denote by  $L^p(D)$ the Lebesgue space and by $ W^{s}_{p}(D)$, with $ s \in \mathbb{R}$, the Sobolev space. In particular, for $ s \neq \text{integer}$,  the fractional Sobolev space $ W^{s}_{p}(D)$ is also called Aronszajn, Gagliardo or Slobodockij space. $L^{\infty}(D)$ is the space of bounded functions endowed by the essential supremum norm. 
 The notation $ B_{pq}^{s}(D)$,\del{, respectively $ B_{pq}^{s}(D) $,}  with $ s \in \mathbb{R}$ and $ 0 < p,q \leq \infty $ is reserved for Besov space. In this work we are interested in the special case $ p=2 $ where $  W^{s}_{2}(D) $ is a Hilbert space. For this the shorter notation $  H^{s}_{2} $ will be used.  We say that  $1 < p , q < \infty $ are conjugate if $ \frac{1}{p}+\frac{1}{q}=1$ and for $ p=1 $ (resp. $ p=\infty $) $ q=\infty $ (resp. $ q=1 $).  $C(\overline{D})$ is the space of continous functions in $\overline{D}$ endwed with the norm . The space of H\"older continous functions of exponent $ \delta \in (0,1) $ is denoted by $ C^\delta(D)$ and the notation $ \mathcal{C}^\delta(D)$ is reserved for Zygmund space. 
 
Let $ X $ and $ Y $ be two Banach spaces, we use the notation $ \vert . \vert_{X} $ to indicate the norm in $ X $ and we denote by $ \mathcal{L}(X,Y) $ the space of linear bounded operators defined on $ X $ into $ Y $ endowed with the norm $ \Vert  .\Vert_{\mathcal{L}(X,Y)} $.  The abbreviations (FSBE) and (FSBTE) are used respectively for fractional stochastic Burgers equation and the fractional stochastic Burgers type equation. In the end of this introduction, let us mention that, the values of the constants may change from line to line and we sometimes delete their dependence on parameters.
  We denote by $L^p(0,1)$ the Lebesgue space on $(0, 1)$, $ C^{\delta}(0,1)$ the space of H\"older continuos functions of index $\delta\in [0, 1)$ defined on $[0, 1]$, in particular,  $C^0$ is the space of continuous and bounded functions,} 
%%%%%%%%%%%%%%%%%%%%%%%%%%%%%%%%

\section{Formulation of the problem}\label{sec-formulation}

\subsection{Definition and properties of the linear drift term.} As mentioned above, we denote $-\Delta$\del{the Laplacian} with Dirichlet conditions boundary by $A$. We denote the part of $A$ on $L^q(0, 1)$ by $A_q$, but exceptionally for $q=2$ and later on, we will write only $A$  and we will omit the subscript $q$.  We start by recalling the following classical results:
\begin{theorem}\label{Prop-1-Laplace} \cite{Taylor-PDE-III}\del{ \cite[p 28-L-p Spectral theory]{Taylor-PDE-III}}
\noindent The operator $A_q$ is densely defined, has bounded inverse ($0$ is in the resolvent)  and the corresponding semi group $ (e^{-A_q t})_{t\geq 0}$ is analytic on $L^q(0,1)$, $q\geq 2$.
\end{theorem}
\noindent Consequently, as $ A_q$ is  the infinitesimal generator of analytic semigroup, then we can define the fractional power of $ A_q^\beta, \beta \in \mathbb{R}$,   see e.g. \cite[Definition 6.7]{Pazy-83} and \cite[Chap. IX]{Yosida} 
\begin{defn}
For all $ \beta >0$, we define $ A_q^{\beta}$, the fractional power of the operator $ A_q$, as the inverse of
\begin{equation}\label{Eq-def-A-alpha}
 A_q^{-\beta}:= \frac{1}{\Gamma(\beta)} \int_0^\infty z^{\beta-1} e^{-zA_q}dz,
\end{equation}
where the Dunford integral in RHS of \eqref{Eq-def-A-alpha} converges in the uniform operator topology.
\end{defn}

\noindent Furthermore, we  recall, see  e.g. \cite{Debbi-scalar-active} and \cite[ps. 283, 303]{Taylor-PDE-I} that $A: D(A)\rightarrow L^2(0,1)$ is an isomorphism, the inverse $A^{-1}$ is self adjoint and
thanks to the compact embedding of $ D(A)$ in
$L^2(0,1)$, we conclude that  $A^{-1}$ is compact on $L^2(0,1)$.
Hence, there exists an orthonormal basis $( e_j)_{j\in \mathbb{N}}\subset D(A)$ consisting of eigenfunctions of
 $ A^{-1}$ and such that the sequence of eigenvalues
$ (\lambda_j^{-1})_{j\in \mathbb{N}}$ with $ \lambda_j>0 $,  converges to zero. Consequently, $( e_j)_{j\in \mathbb{N}}$ is also a sequence of eigenfunctions of $ A$  corresponding to the eigenvalues $ (\lambda_j)_{j\in \mathbb{N}}$.

\begin{lem}\label{Lem-semigroup}
The operator $A$ is  positive, self adjoint on $L^2(0,1)$ and densely defined. Using the spectral decomposition, we construct  positive and negative fractional powers  $A^{\frac\beta2}, \; \beta \in \mathbb{R} $. In particular, as the spectrum of $ A$ is reduced to the discrete one, we get an elegant representation for  $(A^{\frac\beta2}, D(A^{\frac\beta2}))$. In fact, let  $\beta \geq 0$, then, see e.g. \cite{Flandoli-Schmalfub-99},
\begin{eqnarray}\label{construction-of-fract-bounded}
 D(A^\frac\beta2)&=& \{v\in L^2(0,1), \; s.t.\; |v|^2_{D(A^\frac\beta2)}:=
\sum_{j\in \mathbb{N}} \lambda_j^{\beta} \langle v, e_j \rangle^2<\infty\}, \nonumber\\
A^\frac\beta2 v &=& \sum_{j\in \mathbb{N}} \lambda_j^{\frac\beta2}\langle v, e_j \rangle e_j,\;  
\forall v\in D(A^\frac\beta2),
\end{eqnarray}
\del{\begin{equation}\label{basis}
A^{\frac\alpha2} e_k := \lambda_k^{\frac\alpha2} e_k, \; k\in \mathbb{N}
\end{equation}}
The operator $ A_q^\frac\alpha2$ is the infinitesimal generator of an analytic semi group
$(e^{-A_q^\frac\alpha2 t} )_{t\geq 0}$ on
$ L^q(0,1)$. Moreover, we have for $ \beta \geq 0$,
\begin{equation}\label{eq-semigp-property}
 | A_q^\frac\beta2 e^{-A_q^{\alpha/2} t}|_{\mathcal{L}(L^q)}\leq c t^{-\frac\beta\alpha}
\end{equation}
and 
\begin{equation}\label{semigrp-def}
(e^{- A^{\alpha/2} t}v)(x):= \sum_{k=1}^{\infty} e^{-\lambda_{k}^{\frac{\alpha}{2}}t} \langle  v , e_{k} \rangle e_{k}(x), \; for \; all \;  v \in L^2(0,1) .
\end{equation}
\end{lem}
We add also to the list of the properties of the semigroup the  following no classical results:

\begin{lem}\label{Prop-sg-1} 
Let $0<T<\infty$ and $1<\alpha\leq 2$ be fixed and let  $ \delta \in [0, 1) $ and  $ \beta \in \mathbb{R},$ such that $  \delta-\beta < \alpha-\frac{1}{2}$. Then for all  $\eta\in (\frac{1+2\delta-2\beta}{2\alpha}, 1)$, there exists a positive constant  $ C_{\alpha, \delta, \eta, \beta} > 0$ s.t. for all   $ t \in (0,T]$,
\begin{equation}\label{semi-gr-1}
\Vert e^{-A^{\alpha/2} t} \Vert _{\mathcal{L}(H^{\beta}_2, C^{\delta})} \leq C_{\alpha,\delta ,\beta,\eta}t^{-\eta}.
\end{equation}
In particular, for $ \beta>\frac12$ and $ \delta <\beta-\frac12$, there exists a positive constant $ C_{\delta, \beta} > 0$ s.t.
\begin{equation}\label{aide-2-1-Bis}
\Vert e^{-A^{\alpha/2} t} \Vert _{\mathcal{L}(H^{\beta}_2, C^{\delta})} \leq C_{\alpha, \delta, \beta}.
\end{equation}
\end{lem}

\begin{proof}

Let $ \delta \in [0, 1) $ and  $ \beta \in \mathbb{R},$ satisfying $ \delta - \beta < \alpha-\frac{1}{2}$ and let $ v \in H^{\beta}$. Then
it is easy to see that 
\begin{eqnarray}\label{Proof-Est-1-sg-C-delta-Sobolev}
|e^{-A^{\alpha/2} t}v|_{C^\delta}&=& |\sum_{k=1}^\infty e^{-\lambda_k^{\frac{\alpha}{2}}t}\langle v, e_k\rangle e_k|_{C^\delta}\leq \sum_{k=1}^\infty e^{-\lambda_k^{\frac{\alpha}{2}}t}|\langle v, e_k\rangle| |e_k|_{C^\delta}.
\end{eqnarray}
Using lemmas \ref{lem-e-k-holder} and \ref{lem-elementary-1} and  H\"older inequality, we get
\begin{eqnarray}\label{Proof-Est-2-sg-C-delta-Sobolev}
|e^{-A^{\alpha/2} t}v|_{C^\delta}&\leq & c_{\alpha, \delta, \eta} \sum_{k=1}^\infty  t^{-\eta} (k\pi)^{-\eta\alpha}|\langle v, e_k\rangle| k^\delta \leq  C_{\alpha, \delta, \eta} t^{-\eta}\sum_{k=1}^\infty   k^{-\eta\alpha+\delta-\beta}|\langle v, e_k\rangle| (k\pi)^{\beta}\nonumber\\
&\leq &  c_{\alpha, \delta, \eta, \beta} t^{-\eta}\big(\sum_{k=1}^\infty   k^{2(-\eta\alpha+\delta-\beta)}\big)^\frac12  \big(\sum_{k=1}^\infty  (\langle v, e_k\rangle)^2 (k\pi)^{2\beta}\big)^\frac12 \nonumber\\
&\leq &  c_{\alpha, \delta, \eta, \beta} t^{-\eta} |v|_{H^{\beta}_2},
\end{eqnarray}
provided $ 2(-\eta\alpha+\delta-\beta) <-1.$  This last is equivalent to  $ \eta> \frac{1+2\delta-2\beta}{2\alpha}$. A sufficient condition for  $ \eta $ to be in $ [0, 1)$ is that $ \delta - \beta < \alpha-\frac{1}{2}$. The proof of the Est. \eqref{semi-gr-1} is then completed.

Now if $ \beta>\frac12$,  we use the estimate; $ e^{-\lambda_k^\frac\alpha2t}\leq 1$ in Est. \eqref{Proof-Est-1-sg-C-delta-Sobolev} and than we follow the same steps as to get \eqref{Proof-Est-2-sg-C-delta-Sobolev}. The condition $ \delta < \beta-\frac12$ emerges as a consequence for the above calculus. 
\end{proof}
\begin{coro}\label{Coro-1} Let $0<T<\infty$ and $1<\alpha\leq 2$ be fixed and let  $ \delta \in [0, \frac{\alpha-1}{2})$.  Then for all  $\eta\in (\frac{1+2\delta+\alpha}{2\alpha}, 1)$, there exists a positive constant  $ C_{\alpha, \delta, \eta} > 0$ s.t. for all   $ t \in (0,T]$,
\begin{equation}\label{semi-gr-1-partic-beta=alpha-2}
\Vert e^{-A^{\alpha/2} t} \Vert _{\mathcal{L}(H^{-\frac\alpha2}_2, C^{\delta})} \leq C_{\alpha,\delta, \eta}t^{-\eta}.
\end{equation}

\end{coro}
\begin{proof}
Appllying Lemma \ref{Prop-sg-1} for $1<\alpha\leq 2$, $\beta =-\frac\alpha2$ and $  \delta < \frac{\alpha-1}{2}$,
we conclude that Est.\eqref{semi-gr-1-partic-beta=alpha-2}   is fulfilled

\end{proof}

\begin{lem}\label{Prop-sg-2} 
Let $0<T<\infty$ and $1<\alpha\leq 2$ be fixed and let  $ \beta , \gamma  \in \mathbb{R}$ s.t. $ 0< \gamma-\beta <\alpha $. Then for all $ \eta\in  ( \frac{\gamma-\beta}{\alpha}, 1)$, there exists a positive constant $ C_{\alpha, \gamma, \eta, \beta} > 0$ s.t. for all $ t \in (0,T]$,
\begin{equation}\label{semi-gr-2-*}
\Vert e^{-A^{\alpha/2} t} \Vert _{\mathcal{L}(H^{\beta}_2, H^{\gamma}_2)} \leq C_{\alpha,  \gamma, \eta, \beta} t^{-\eta}.
\end{equation}
In particular, for $ \beta>\gamma$, we have 
\begin{equation}\label{semi-gr-H-beta-gamma}
\Vert e^{-A^{\alpha/2} t} \Vert _{\mathcal{L}(H^{\beta}_2, H^{\gamma}_2)} \leq 1.
\end{equation}
\end{lem}
 
\begin{proof}
Let $ \beta , \gamma  \in \mathbb{R}$ s.t. $ 0< \gamma-\beta <\alpha $ and let $ v \in H^{\beta}_{2} $. A simple calculus yields to
 \begin{eqnarray} 
\vert e^{-  A^{\alpha/2} t} v \vert_{H^{\gamma}_{2}}^{2}&=& \vert A^{\frac{\gamma}{2}} e^{-  A^{\alpha/2}t}v \vert_{L^{2}}^{2} \del{= \sum_{k=1}^{\infty} \langle A^{\frac{\gamma}{2}} e^{- A^{\alpha/2}t} v,e_{k} \rangle^{2}}
\leq 
C_{\alpha,\gamma}\sum_{k=1}^{\infty} k^{2\gamma} e^{-2 \lambda_{k}^{\frac{\alpha}{2}} t} \langle  v,e_{k} \rangle^{2}.
\end{eqnarray}
Using Lemma \ref{lem-elementary-1}, we get for a given $\eta>0$, 
\begin{eqnarray} 
\vert e^{-  A^{\alpha/2} t} v \vert_{H^{\gamma}_{2}}^{2} & \leq & C_{\alpha, \gamma, \eta}t^{-2\eta} \sum_{k=1}^{\infty} k^{2(\gamma- \alpha \eta)} \langle  v,e_{k} \rangle^{2} \leq C_{\alpha, \gamma, \eta, \beta} t^{-2\eta} \sum_{k=1}^{\infty} k^{2(\gamma-\alpha \eta -\beta)} \langle A^{\frac{\beta}{2}}v,e_{k} \rangle^{2}\nonumber \\
&\leq &  C_{\alpha,\gamma,\eta,\beta} t^{-2\eta}\vert v \vert_{H^{\beta}_{2}}^{2},   
\end{eqnarray}
provided $ 2(\gamma-\alpha \eta -\beta) < 0 $ which is equivalent to $ \eta > \frac{\gamma -\beta}{\alpha} $.\\
If $ \beta > \gamma $, we get Est. \eqref{semi-gr-H-beta-gamma} by following the same steps of the proof above and using of  the estimate $ e^{-2 \lambda_{k}^{\frac{\alpha}{2}} t} < 1$ in stead of Lemma \ref{lem-elementary-1}. 
\end{proof}
\begin{lem}\label{lem-est-sg-l-2-l-4}
Let $0<T<\infty$ and $1<\alpha\leq 2$ be fixed. Then there exists a positive constant $   C_{\alpha} > 0$ s.t. for all   $ t \in (0,T]$,
\begin{equation}\label{semi-gr-1-L-2-L-4}
\Vert e^{-A^{\alpha/2} t} \Vert _{\mathcal{L}(L^2, L^4)} \leq C_{\alpha} t^{-\frac1{4\alpha}-}.
\end{equation}
\end{lem}

\begin{proof} Let $ v\in L^2$. Using the boundness of the semi group $ e^{-A^{\alpha/2} t}$ on $ L^2(0, 1)$ and Lemma \ref{Prop-sg-1} (with $ \beta=\delta=0 $), we can easily deduce that 
\begin{eqnarray}
\vert e^{-A^{\alpha/2} t}v \vert^4_{ L^4} &\leq &  \int_0^1|(e^{-A^{\alpha/2} t}v)(x)|^4 dx \leq  \sup_{x\in [0, 1]}|(e^{-A^{\alpha/2} t}v)(x)|^2 \int_0^1|(e^{-A^{\alpha/2} t}v)(x)|^2 dx \nonumber\\
&\leq & |e^{-A^{\alpha/2} t}|^2_{\mathcal{L}(L^2, C^0)} |v|_{L^2}^2|e^{-A^{\alpha/2} t}v|_{L^2}^2 \leq |e^{-A^{\alpha/2} t}|^2_{\mathcal{L}(L^2, C^0)} |v|^4_{L^2}\leq C_{\alpha} t^{-\frac1{\alpha}-}|v|^4_{L^2}.
\end{eqnarray}
\end{proof}

\begin{lem}\label{lem-sg-regul}
Let $0<T<\infty$ and $1<\alpha\leq 2$ be fixed and let $ \delta \in [0, \frac{\alpha-1}{2})$. Then for all $\gamma \in (\frac{1+\alpha+2\delta}{2\alpha}, 1)$ and all $ \eta \in (0, \gamma - (\frac{1+\alpha+2\delta}{2\alpha}))$, there exists a a positive constant $ C_{\alpha, \delta, \gamma, \eta} > 0$ s.t. for all   $ t, s \in (0,T]$,
\begin{equation}\label{est-sg-t-s}
\Vert e^{-A^{\alpha/2} t} - e^{-A^{\alpha/2} s} \Vert _{\mathcal{L}(H^{-\frac\alpha2}_2, C^{\delta})} \leq C_{\alpha, \delta, \gamma, \eta}s^{-\gamma}|t-s|^{\eta},
\end{equation} 
\end{lem}

\begin{proof}
Let $ v\in H^{-\frac\alpha2}_2(0, 1)$. Using the semigroup property, lemmas \ref{lem-e-k-holder}, \ref{lem-elementary-1} and \ref{lem-elementary-2}, we infer, for $\gamma, \eta \in (0, 1)$, that
\begin{eqnarray}
| (e^{-A^{\alpha/2} t} - e^{-A^{\alpha/2} s})v|_{C^\delta}&=&
| e^{-A^{\alpha/2} s}(e^{-A^{\alpha/2} (t-s)} - I)v|_{C^\delta}\nonumber\\
&\leq & \sum_{k=1}^\infty
e^{-\lambda_k^{\alpha/2}s}(1-e^{-\lambda_k^{\alpha/2}(t-s)})|\langle v, e_k\rangle||e_k|_{C^\delta}\nonumber\\
&\leq & C_{\alpha, \delta, \gamma}\sum_{k=1}^\infty
k^{-\alpha\gamma}s^{-\gamma}(1-e^{-\lambda_k^{\alpha/2}(t-s)})|\langle v, e_k\rangle|k^\delta\nonumber\\
&\leq & C_{\alpha, \delta, \gamma, \eta} \sum_{k=1}^\infty k^{-\alpha\gamma} s^{-\gamma} k^{\alpha\eta} (t-s)^\eta|\langle v, e_k\rangle|k^\delta \nonumber\\
&\leq & C_{\alpha, \delta, \gamma, \eta} s^{-\gamma}(t-s)^\eta\sum_{k=1}^\infty k^{-\alpha\gamma}  k^{\alpha\eta} |\langle v, e_k\rangle|k^{\delta+\alpha/2} k^{-\alpha/2}
\end{eqnarray}
Now, we apply H\"older inequality, we get
\begin{eqnarray}
| (e^{-A^{\alpha/2} t} - e^{-A^{\alpha/2} s})v|_{C^\delta}
&\leq & C_{\alpha, \delta, \gamma, \eta} s^{-\gamma}(t-s)^\eta(\sum_{k=1}^\infty k^{2(\alpha(1/2-\gamma+\eta)+\delta)})^{1/2}(\sum_{k=1}^\infty k^{-\alpha} \langle v, e_k\rangle^2)^{1/2} \nonumber\\
&\leq & C_{\alpha, \delta, \gamma, \eta} s^{-\gamma}(t-s)^\eta|v|_{H^{-\frac\alpha2}_2},
\end{eqnarray}
provided that $ \delta \in [0,\frac{\alpha -1}{2}) $, $\gamma \in (\frac{1+\alpha+2\delta}{2\alpha}, 1)$ and $ \eta \in (0, \gamma - (\frac{1+\alpha+2\delta}{2\alpha}))$.  
\end{proof}

\begin{coro}\label{coro-sg-regul}
Assume that the  conditions of Lemma \ref{lem-sg-regul} are satisfied,  then there exists a positive function $ C_{\alpha, \delta}(\cdot)$ on $(0, T]$,  s.t. for all  $ t\in (0,T]$,
\begin{equation}\label{est-sg-t-0}
\Vert e^{-A^{\alpha/2} t} - e^{-A^{\alpha/2} t_{0}} \Vert _{\mathcal{L}(H^{-\frac\alpha2}_2, C^{\delta})} \leq C_{\alpha, \delta}(t_0)|t-t_0|^{(\frac{\alpha-1-2\delta}{2\alpha})-}.
\end{equation}
\end{coro}
\begin{proof}
Est.\eqref{est-sg-t-0} can be easily obtained from Est.\eqref{est-sg-t-s} with $\gamma =1-\epsilon,$ $\epsilon>0$ and $\eta= \frac{\alpha-1-2\delta}{2\alpha}-\epsilon$.
\end{proof}

\del{\begin{lem}\label{lem-sg-regul}
Let $0<T<\infty$ and $1<\alpha\leq 2$ be fixed and let  $ \delta \in [0, 1)$. Then there exists a positive constant $ C_{\alpha, \delta} > 0$ s.t. for all   $ t, s \in (0,T]$,
\begin{equation}\label{est-sg-t-s}
\Vert e^{-tA^{\alpha/2}} - e^{-sA^{\alpha/2}} \Vert _{\mathcal{L}(H^{-\frac\alpha2}_2, C^{\delta})} \leq C_{\alpha, \delta}s^{-\gamma}|t-s|^{\eta},
\end{equation}
with $\gamma >\frac{1+\alpha+2\delta}{2\alpha}$ and $ \eta \in (0, \min\{1, \gamma - (\frac{1+\alpha+2\delta}{2\alpha})\})$. 

In particular, for $\delta<\frac{\alpha-1}{2}$, we can choose  $\gamma \in (\frac{1+\alpha+2\delta}{2\alpha}, 1)$ and $ \eta \in (0, \gamma - (\frac{1+\alpha+2\delta}{2\alpha}))$. 

\end{lem}
\begin{proof}
Let $ v\in H^{-\frac\alpha2}_2(0, 1)$, using the semigroup property, lemmas \ref{lem-e-k-holder}, \ref{lem-elementary-2} and 
\ref{lem-elementary-1},

\begin{eqnarray}
| (e^{-tA^{\alpha/2}} - e^{-sA^{\alpha/2}})v|_{C^\delta}&=&
| e^{-sA^{\alpha/2}}(e^{-(t-s)A^{\alpha/2}} - I)v|_{C^\delta}\nonumber\\
&\leq & \sum_{k=1}^\infty
e^{-\lambda_k^{\alpha/2}s}(e^{-\lambda_k^{\alpha/2}s} - 1)|\langle v, e_k\rangle||e_k|_{C^\delta}\nonumber\\
&\leq & C \sum_{k=1}^\infty k^{-\alpha\gamma} s^{-\gamma} k^{\alpha\eta} (t-s)^\eta|\langle v, e_k\rangle|k^\delta \nonumber\\
&\leq & C s^{-\gamma}(t-s)^\eta\sum_{k=1}^\infty k^{-\alpha\gamma}  k^{\alpha\eta} |\langle v, e_k\rangle|k^{\delta+\alpha/2} k^{-\alpha/2}
\end{eqnarray}
with $\gamma, \eta \in (0, 1)$. We apply H\"older inequality, we get
\begin{eqnarray}
| (e^{-tA^{\alpha/2}} - e^{-sA^{\alpha/2}})v|_{C^\delta}
&\leq & C s^{-\gamma}(t-s)^\eta(\sum_{k=1}^\infty k^{2(\alpha(1/2-\gamma+\eta)+\delta)})^{1/2}(\sum_{k=1}^\infty k^{-\alpha} \langle v, e_k\rangle^2)^{1/2} \nonumber\\
&\leq & C_{\alpha, \delta} s^{-\gamma}(t-s)^\eta|v|_{H^{-\frac\alpha2}_2},
\end{eqnarray}
provided that $\gamma >\frac{1+\alpha+2\delta}{2\alpha}$ and $ \eta \in (0, \gamma - (\frac{1+\alpha+2\delta}{2\alpha}))$. Under the condition $ \delta<\frac{\alpha-1}{2}$, we can get $\frac{1+\alpha+2\delta}{2\alpha}< \gamma <1$.\del{It is easy to deduce the particular case.}  
\end{proof}

\begin{coro}\label{coro-sg-regul}
Under the conditions of Lemma \ref{lem-sg-regul}. Then for all $ t_0$ there exists a positive constant $ C_{\alpha, \delta, t_0} > 0$ s.t. for all  $ t\in (0,T]$,
\begin{equation}\label{est-sg-t-0}
\Vert e^{-tA^{\alpha/2}} - e^{-t_0A^{\alpha/2}} \Vert _{\mathcal{L}(H^{-\frac\alpha2}_2, C^{\delta})} \leq C_{\alpha, \delta, t_0}|t-t_0|^{(\frac{\alpha-1+2\delta}{2\alpha})-}.
\end{equation}
\end{coro}
\begin{proof}
Est. \eqref{est-sg-t-0} can be easily obtained from Est. \eqref{est-sg-t-s} with $\gamma =1-\epsilon,$ $\epsilon>0$ and 
$\eta= \frac{\alpha-1+2\delta}{2\alpha}-\epsilon$.
\end{proof}
The useful case for us is Estimate \eqref{semi-gr-1} for $ \beta =-\frac\alpha2$ and  $ \delta <\frac{\alpha-1}{2}$. In fact, we have the following corollary:

\begin{coro}\label{Coro-1} For $1<\alpha\leq 2$, $  \delta < \frac{\alpha-1}{2}$ and thanks to Lemma \ref{Prop-sg-1} and Corollary \ref{coro-sg-regul},  
we conclude that the first part of  $\textit{Assumption 1.}$  is satisfied for $ U=H^{-\frac\alpha2}$ and $ V= C^\delta$.
\end{coro}}

\subsection{Definition and properties of the nonlinear drift term.}\label{sub-sec-nonlin}

The general form of the nonlinear part of the drift term is given by a function $F$ satisfying \textit{Assumption 2.},  with $ V:= C^{\delta}(0,1)$,  for a relevant $\delta \in (0, 1)$  and $ U:= H^{-\alpha/2}_2(0, 1)$. Our typical example is,
\begin{equation}\label{Def-F}
F(u)(x)=\frac{d f(u(x))}{dx} , \;\; \text{with} \;\; f: H^{1-\frac\alpha2}_2(0, 1) \rightarrow C^{\delta}(0,1)\;\; \text{ being locally Lipschitz}
\end{equation}
Specially,  the case when $f$  is a polynomial; $f(x):= a_0+a_1x+a_2x^2+ ... a_nx^n,$ with $ a_0, a_1, ..., a_n\neq 0\in \mathbb{R}, \; n\in \mathbb{N}$. It is easily seen that for $ n=2$ (i.e. $ a_2\neq 0$) and $ a_1=0$, we recuperate the fractional stochastic Burgers equation. This fact justifies the name of Burgers type. For $ n=1$, the drift is then linear and the equation is nothing than the fractional stochastic heat equation with globally Lipschitz coefficients. This study covers, with slight modifications, the case  when the coefficients $(a_j(\cdot))_{j=0}^n$ are differentiable real functions with  bounded derivatives.  Moreover, the proofs can be easily extended to the case of no autonomuous function $f(t, x, u)$, under switable conditions.

\begin{lem}\label{main-nonlinear}
Let $ 1 < \alpha \leq 2$,  $ \delta \in ( 1-\frac{\alpha}{2}, 1) $\del{$ \delta > 1-\frac{\alpha}{2} $ } and let $F$ be given by \eqref{Def-F}, with $ f$ being a polynomial of order $n$, then the following mapping 
\begin{eqnarray}
F &:& C^{\delta}(0,1) \rightarrow H^{-\frac{\alpha}{2}}_2(0, 1)\nonumber \\
& &
u \mapsto F(u)= \frac{d f(u(\cdot))}{dx}, \del{\frac{\partial f(u)}{\partial x},}
\end{eqnarray}
 is well defined. Moreover, for all  $R>0$, there exists a positive constant $ C_{R} $ such that for every $  u, \; v \in C^{\delta}(0,1) $ with $ | u | _{C^{\delta}}  \; , |v |_{C^{\delta}} \leq R$, the following inequality holds
\begin{equation}\label{main-inq-nolin}
| F(u)-F(v) |_{H^{-\frac{\alpha}{2}}_2} \leq C_{R} |u-v |_{C^{\delta}}.
\end{equation}  
\end{lem}

\begin{proof}
Let $\; u , \; v \in C^{\delta}(0,1) $ such that $ |u|_{C^{\delta}}  \; , |v|_{C^{\delta}} \leq R $ for a given  $ R > 0 $, thanks to  the Imbedding \eqref{Imbd-req-3} in Lemma \ref{lem-ess-general}, we get
\begin{equation*}
|F(u)-F(v) |_{H^{-\frac{\alpha}{2}}_2} \leq C|f(u)-f(v))|_{H^{1-\frac{\alpha}{2}}_2} \leq C |f(u)-f(v)|_{C^{\delta}}.
\end{equation*}
First, we consider the case $ n=1$.  Then 
\begin{eqnarray}
| F(u)-F(v) |_{H^{-\frac{\alpha}{2}}_2} \leq  C | a_1| | u-v|_{C^{\delta}}.
\end{eqnarray}
In this case $F$ is  globaly Lipschitz. Now For $n\geq 2$, we use the definition of $f$ and the fact that $ C^\delta$ is a multiplication algebra, see Lemma \ref{lem-ess-general}, we get  
\begin{eqnarray}
| F(u)-F(v) |_{H^{-\frac{\alpha}{2}}_2} \leq  C  \sum_{k=1}^{n}  \sum_{j=0}^{k-1}|a_k ||u^{k-1-j}v^{j}|_{C^{\delta}} | u-v|_{C^{\delta}} \leq  
 n^2CR^n | u-v |_ {C^{\delta}}.
\end{eqnarray}
\end{proof}

\begin{coro}\label{Coro-nolinearterm}
For $ 1 < \alpha \leq 2$,  $ \delta \in ( 1-\frac{\alpha}{2}, 1)$,  $F$ given by \eqref{Def-F}, with $ f$ being a polynomial, $ U=H^{-\frac\alpha2}_2(0, 1)$ and $ V=C^{\delta}(0,1)$, then  
thanks to Lemma \ref{main-nonlinear}, \textit{Assumption 2.} is fulfilled. 
\end{coro}

\begin{coro}\label{Coro-nolinearterm+1}
For $ \alpha \in (\frac32,  2]$,  $\delta \in ( 1-\frac{\alpha}{2}, \frac{\alpha-1}{2})$,  $F$ given by \eqref{Def-F}, with $ f$ being a polynomial, $ U=H^{-\frac\alpha2}_2(0, 1)$ and $ V=C^\delta(0, 1)$ then  thanks to lemmas \ref{Prop-sg-1} and \ref{main-nonlinear}, the first Part of \textit{Assumption 1.} and \textit{Assumption 2.} are simultenuously fulfilled. 
\end{coro}

\begin{coro}\label{Coro-linear-growth}
\del{Thanks to \eqref{main-inq-nolin}, we deduce the local linear growth of the nonlinear term, i.e. }For  $R>0$, there exists  a positive constant $ C_{R} $ such that for every $  u \in C^{\delta}(0,1) $ with $ |u |_{C^{\delta}} \leq R $, 
\begin{equation}
| F(u)|_{H^{-\frac{\alpha}{2}}_2} \leq C_{R}(1+ |u|_{C^{\delta}}).
\end{equation}  
\end{coro}

\begin{proof} It is suffisant to take $v=0$, in the Est.\eqref{main-inq-nolin}.\del{, we get 
\begin{equation}
| F(u)|_{H_2^{-\frac{\alpha}{2}}} \leq | F(u)-F(0) |_{H_2^{-\frac{\alpha}{2}}}+ |F(0)|) \leq C_{R} |u|_{C^{\delta}} + |F(0)|.
\end{equation} } 

\end{proof}
In \cite{Blomker-Jentzen}, the authors assumed the existence of a family $(X_N)_N$ satisfying suitabel conditions, see Assumption 4, Cond.\eqref{Eq-Assum-4-1} \& Cond.\eqref{Eq-Assum-4}). In our work, we give sufficient conditions for the existence of such family.  We assume that: 
\begin{itemize}
\item There exists a function $g:\mathbb{R}^2 \rightarrow \mathbb{R}$, such that 
\begin{equation}\label{cond-f-g}
f(x)- f(y)= (x-y)g(x, y) \;\; \& \;\; \forall R>0, \; \exists \; C_R, \; s.t. \forall  x, y,  |x|, |y|\leq R,\;\; |g(x, y)|\leq C_R.
\end{equation}
\item There exist $m\in \mathbb{N}_0$, $(c_j)_{j=1}^{m}, c_j>0$\del{, $(\nu_j)_{j=1}^{m}$} and $(\mu_j)_{j=1}^{m}$, with $ 0<\mu_j<2$, such that for all $v\in H^\frac\alpha2_2(0, 1), \; \xi \in C^\delta(0, 1)$, 
\begin{equation}\label{con-F-supp}
|\langle F(v+\xi), v \rangle| \leq \sum_{j=1}^m c_j|v|^{\mu_j}_{H^\frac\alpha2_2}|\xi|^j_{C^\delta}+ |v|_{L^2} |v|_{H^\frac\alpha2_2} (\sum_{j=1}^m c_j|\xi|^j_{C^\delta}).
\end{equation}
\end{itemize}

%%%%%%%%%%%%%%%%%%%%%%%%%%%%%%%%%%%%%%%%%%%%%%%%%%%%%%%%%%%%%%%%%%%%%%%%%%%%%%%%%%%%%%%
\subsection{Definition of the stochastic term.}
We fix a stochastic basis $ (\Omega, \mathcal{F}, \mathbb{P}, \mathbb{F}, W)$, where
$ (\Omega, \mathcal{F}, \mathbb{P}) $  is a complete probability space, $\mathbb{F} := (\mathcal{F}_t)_{t\geq 0}$ is a filtration satisfying the usual conditions, i.e. $(\mathcal{F}_t)_{t\geq 0}$ is an increasing right continuous filtration. The  Wiener process
$ W:= (W(t), t\in [0, T])$ is a mean zero
 Gaussian process defined (on the filtered probability space $ (\Omega,
\mathbb{F}, \mathbb{P}, \mathbb{F} )$, such that the covariance function is given by:
\begin{equation}\label{Eq-Cov-W}
\mathbb{E}[W(t)W(s)]= (t\wedge s)I,  \;\;\;  \forall\;\;  t,s \geq 0,
\end{equation}
where $ I$ is the identity. Formally,  we rewrite $ W$ as the sum of an infinite series;
\begin{equation}\label{W-series}
W(t)= \sum_{k=1}^{\infty}  \beta_{k}(t)e_{k}, \; \; \;  \mathbb{P}-a.s ,
\end{equation}
where $ (\beta_{k})_{k\geq1}$ is a family of independent standard Brownian motions and $ (e_{k}(.)= \sqrt{2}\sin (k \pi .)) $ is an orthonormal basis in the space $ L^{2}(0,1)$.
We introduce the following Ornstein-Uhlenbeck stochastic process (OU) 
\del{\begin{equation}\label{OU-process}
\mathcal{W}(t):= \int_0^te^{-A^{\alpha/2} (t-s)}GW(ds),
\end{equation}
with $ G\in \mathcal{L}(H)$ is a bounded operator satisfying 
\begin{equation}
Ge_k=q_ke_k, 
\end{equation}
with $ (q_k)_{k=1}^\infty$ is a real sequence \del{complex-valued measurable functions} satisfying
\begin{equation}
\sum_{k=1}^{\infty} k^{-\alpha}q_k^2 <\infty.
\end{equation}}
\begin{equation}\label{OU-process}
\mathcal{W}(t):= \int_0^te^{-A^{\alpha/2} (t-s)}W(ds).
\end{equation}
It is easy to see that  $\mathcal{W}$ is well defined for all $\alpha>1$, see e.g. \cite{DebbiDozzi1}.

%%%%%%%%%%%%%%%%%%%%%%%%%%%%%%%%%%%%%%%%%%%%%%%%%%%%%%%%%%%%%%%%%%%%%%%%%%%%%%%%% 

\del{\section{Discretization} Let us fix $ N, M \geq 1 $ and consider the uniform step subdivision of the time interval $[0, T]$, with time step  $\Delta t = \frac{T}{M}$. We define $t_{m}= m \Delta t , \; for \; m=1,..., M$.
 }
 %%%%%%%%%%%%%%%%%%%%%%%%%%%%%%%%%%%%%%%%%%%%%%%%%%%%%%%%%%%%%%%%%%%%%%%%%%%%%%%%%%%%
\subsection{Definition of the Galerkin approximation.} We fix $ N\geq 1$.
\del{We recall in $ L^2(0, 1)$ the projection $P_N$ on the finite dimension spaces generated by the $N$th first vectors of the basis.} We denote by $P_N$, the Galerkin projection on the finite  space $H_N$ generated by the $N$ first eigenvectors $(e_k)_{k=1}^N$, i.e. for $ \delta>0$,  $ v\in C^{\delta}(0,1)\subset L^2(0, 1)$ and for all $ x\in [0, 1]$,
\begin{equation}\label{P-N-Brut}
P_N v (x)= \sum_{k=1}^N\langle v, e_k\rangle e_k(x).
\end{equation}

\begin{lem}\label{Lem-Est-1-PN}
\begin{itemize}
\item $P_N$ and $e^{-tA^{\alpha/2}}$ commute.
\item Let $ \delta \in [0,1) $ and $\eta >\delta+\frac12$, then there exists $ C_{\delta,\eta}>0,$ s.t.
\begin{equation}\label{est-P_n-H-eta-C-delta}
\Vert P_N \Vert_{\mathcal{L}(H^{\eta}_2, C^\delta)}\leq  C_{\delta, \eta,}.
\end{equation}
\item Let $\beta\leq \gamma \in \mathbb{R}$, then there exists $ C_{\gamma, \beta}>0,$ s.t.
\begin{equation}\label{Est-1-P-N}
\Vert 1-P_N \Vert_{\mathcal{L}(H^{\gamma}_2, H^{\beta}_2)}\leq  C_{\gamma, \beta} N^{-(\gamma-\beta)}.
\end{equation}
\item Let $ \alpha \in (1, 2]$ and  $\delta\in [0, \frac{\alpha-1}{2})$, then there exists $ C_{\gamma, \beta}>0,$ s.t.
\begin{equation}
N^{(\frac{\alpha-1}{2}-\delta)-}\Vert(1-P_N)e^{-A^{\alpha/2} t}\Vert_{\mathcal{L}(\del{H^{(\delta+\frac12)+}_2}H^{\alpha/2}_2, C^{\delta})}\leq  C_{\alpha, \delta}.
\end{equation}
\item Let $\beta \in \mathbb{R}$, then there exists $ C_{\beta}>0,$ s.t.
\begin{equation}\label{est-P_n-H-beta}
\Vert P_N \Vert_{\mathcal{L}(H^{\beta}_2)}\leq  C_{\beta},
\end{equation}
\end{itemize}
\end{lem}
\begin{proof} We omit the proofs of the first and last statements as they are easy. Now, we prove  the second one. Let $v\in C^{\delta}(0,1)$, thanks to  Identity \eqref{P-N-Brut} and Lemma \ref{lem-e-k-holder}, it is easy to see that 
\begin{eqnarray}
|P_Nv|_{C^\delta}&\leq & \sum_{k=1}^N|\langle v, e_k\rangle||e_k|_{C^\delta}\leq C_\delta \sum_{k=1}^N|\langle v, e_k\rangle|k^\delta \nonumber\\
&\leq &  C_\delta \sum_{k=1}^N|\langle v, k^\eta e_k\rangle|k^{\delta-\eta} \leq  C^\delta \sum_{k=1}^N|\langle A^{\eta/2}v, e_k\rangle|k^{\delta-\eta}.
\end{eqnarray}
Using H\"older inequality and the condition $\eta >\delta+\frac12$, we deduce that 
\begin{eqnarray}
|P_Nv|_{C^\delta}&\leq &  C_\delta (\sum_{k=1}^\infty\langle A^{\eta/2}v,  e_k\rangle^2)^\frac12 (\sum_{k=1}^\infty k^{2(\delta-\eta)})^\frac12 \leq C_{\delta, \eta}|v|_{H^\eta_2}.
\end{eqnarray}

For the third estimate, we consider $v\in H^{\gamma}_2$, then
\begin{eqnarray}
|(1-P_N)v|^2_{H^{\beta}_2} &= & |A^{\beta/2}(1-P_N)v|^2_{L^2}= 
\sum_{k=1}^\infty \langle  A^{\beta/2}(1-P_N)v, e_k \rangle^2 \leq  \sum_{k=N+1}^\infty   \lambda_k^{\beta} \langle v, e_k \rangle^2 \nonumber\\   
&\leq & \sum_{k=N+1}^\infty   \lambda_k^{\beta-\gamma} \langle A^{\gamma/2}v, e_k \rangle^2   \leq  \pi^{2(\beta-\gamma)} \sum_{k=N+1}^\infty   k^{2(\beta-\gamma)} \langle A^{\gamma/2}v, e_k \rangle^2   
 \nonumber\\
& \leq & \pi^{2(\beta-\gamma)} N^{ 2(\beta-\gamma)}\sum_{k=N+1}^\infty   \langle A^{\gamma/2}v, e_k \rangle^2 \leq  \pi^{2(\beta-\gamma)} N^{ 2(\beta-\gamma)} |v|_{H^{\gamma}_2}^2. 
\end{eqnarray}
For the fourth estimate,  we assume that $ v\in  H_2^{\alpha/2}$ and we prove that there exists $ C_{\alpha, \delta}>0$, such that  
\begin{equation}\label{Eq-Assum-3-proo-u-0}
\sup_{N\in \mathbb{N}}\sup_{t\in (0, T]}\big(N^{(\frac{\alpha-1}{2} -\delta)-} \del{t^\alpha}|(1-P_N)e^{-tA^{\alpha/2}}v|_{C^\delta}\big)<\infty. 
\end{equation}
In fact, by application of Lemma  \ref{Lem-Est-1-PN}, \del{Lemma \ref{Prop-sg-1} } Est.\eqref{aide-2-1-Bis} and Est.\eqref{Est-1-P-N}, we infer the existence of $ C_{\alpha, \delta}>0$, such that
\begin{eqnarray}
|(1-P_N)e^{-A^{\alpha/2} t}v|_{C^\delta} &=& |e^{-A^{\alpha/2} t}(1-P_N)v|_{C^\delta} \nonumber\\
&\leq &  |e^{-A^{\alpha/2} t}|_{\mathcal{L}( H^{(\delta+\frac12)+}_2, C^\delta)} |1-P_N|_{\mathcal{L}(H^{\alpha/2}_2, H^{(\delta+\frac12)+}_2)} |v|_{ H^{\alpha/2}_2}\nonumber\\
&\leq & C_{\alpha, \delta} N^{-(\frac{\alpha-1}{2}-\delta)+}|v|_{H^{\alpha/2}_2}.
\end{eqnarray}
\end{proof}

\begin{coro}  Let  $0<T<\infty$, $1<\alpha\leq 2$, $ \delta \in [0, 1)$.
\begin{itemize}
\item Let  $ \beta \in \mathbb{R},$ such that $  \delta-\beta < \alpha-\frac{1}{2}$. Then, for all $\eta\in (\frac{1+2\delta-2\beta}{2\alpha}, 1)$, there exists a positive constant $  C_{\alpha, \delta, \beta, \eta} > 0$ s.t. for all   $ t \in (0,T]$,
\begin{equation}\label{semi-gr-1-P-N}
\Vert e^{-A^{\alpha/2} t}P_N \Vert _{\mathcal{L}(H^{\beta}, C^{\delta})} \leq C_{\alpha, \delta ,\beta, \eta}t^{-\eta}
\end{equation}
and for $ \beta>\frac12$ and $ \delta <\beta-\frac12$, there exists a positive constant $ C_{\delta, \beta} > 0$ s.t.
\begin{equation}\label{aide-2-1-Bis-P-N}
\Vert e^{-A^{\alpha/2} t} P_N\Vert _{\mathcal{L}(H^{\beta}, C^{\delta})} \leq C_{\delta ,\beta}.
\end{equation}
\item Let $ \beta \leq \gamma  \in \mathbb{R}$. Then for all  $ \eta'\in  ( \frac{\gamma-\beta}{\alpha}, 1)$, there exists a positive constant $C_{\alpha, \beta, \gamma} > 0$ s.t. for all   $ t \in (0,T]$,
\begin{equation}\label{semi-gr-2-*_P_N}
\Vert e^{-tA^{\alpha/2}} P_N\Vert _{\mathcal{L}(H^{\beta}_2, H^{\gamma}_2)} \leq C_{\alpha, \beta, \gamma}t^{-\eta'}.
\end{equation}
In particular, for $ \beta>\gamma$, we have 
\begin{equation}\label{semi-gr-H-beta-gamma}
\Vert e^{-tA^{\alpha/2}} P_N\Vert _{\mathcal{L}(H^{\beta}_2, H^{\gamma}_2)} \leq 1.
\end{equation}

\end{itemize}
\end{coro}
\begin{proof}
Combinning Lemma \ref{Prop-sg-1} and Lemma \ref{Lem-Est-1-PN},  
we conclude the first statement and  Similarly, combinning Lemma \ref{Prop-sg-2} and Lemma \ref{Lem-Est-1-PN} we get the second one.
\end{proof}
We introduce the following discretized version of Eq.\eqref{FSBE-Evol-1}, using the spectral Galerking method:

\begin{equation}\label{Discr-Galerkin-FSBE-Evol-1}
\left\{
\begin{array}{lr}
du_N(t)=[-A^{\alpha/2} u_N(t) + P_NF(u_N(t))]dt+ dW_N(t),\;\; t\in [0, T],\\
u(0)=P_Nu_0,
\end{array}
\right.
\end{equation}
where  \begin{equation}\label{W-N-series}
W_N(t):= P_NW(t)= \sum_{k=1}^{N}  \beta_{k}(t)e_{k}\del{ \; \; \;  \mathbb{P}-a.s,}
\end{equation}

\subsection{Fully Discretization}\label{subsec-full-discrt}
Let us fix $ M \geq 1 $ and consider the uniform step subdivision of the time interval $[0, T]$, with time step  $\Delta t = \frac{T}{M}$. We define $t_{m}= m \Delta t , \; for \; m=1,..., M$. We construct the sequence of random variables $(u_{N, M}^{m})_{m=0}^{M}$ as:
\begin{equation}
\left\{
\begin{array}{rl}
u_{N, M}^{0}& := P_N u_0,\\
u_{N, M}^{m+1} & := e^{-A^{\alpha/2}T/M}\Big( u_{N, M}^{m} + \frac{T}{M}(P_NF)(u_{N, M}^{m})\Big) + P_N\Big( \mathcal{W}((m+1)\frac{T}{M} )-  e^{-A^{\alpha/2}T/M}\mathcal{W}((m)\frac{T}{M})  \Big).    
\end{array}
\right.
\end{equation}
 Let us mention here that we can also rewrite $u_{N, M}^m$ as 
\begin{equation}\label{eq-u-N-M-m}
\left\{
\begin{array}{rl}
u_{N, M}^{0}& := P_N u_0,\\ 
u_{N,M}^{m}& := e^{-A^{\alpha/2}t_{m}} u_{N,M}^0 + \Delta t \sum_{k=0}^{m-1}e^{-A^{\alpha/2}(t_{m}-t_{k})} P_N F(u_{N,M}^{k}) + \mathcal{W}_N(t_m)    
\end{array}
\right.
\end{equation}
and that the sequence $ (u^m_{N, M})_m $ has the following property, see the proof in Subsection \ref{subsec-proof-lem}.

\begin{lem}\label{u-N-M-bounded},
Let $ \frac74<\alpha< 2$, $\delta\in (1-\frac{\alpha}{2}, \frac{2\alpha-3}{2})$, $ F$ defined as in \del{cby $f(x)=x^2$} Subsection \ref{sub-sec-nonlin} and  $ u_0$ satiesfies Assumption $ \mathcal{A}$. \del{$ u_0 \in C^\delta \cap H^{\eta}(0,1) $ s.t. $ \eta > \delta + \frac{1}{2} $} 
Then there exists a finite $\mathcal{F}_0$-random variable $ C_{\alpha, \delta}$, s.t. for all $ \omega \in \Omega$, 
\begin{equation}\label{eq-u-m-N-M-bounded}
\sup_{m, N, M}|u_{N, M}^m(\omega)|_{C^\delta} < C_{\alpha, \delta}(\omega),
\end{equation}
 
\end{lem}
In the end of this section we give our assumption of the initial condition:

\noindent \textbf{Assumption $ \mathcal{A} $.} For $ \delta \in [0,1) $ and $ \beta > \delta +\frac{1}{2} $, we have $ u_0: \Omega \rightarrow H^{\beta}_{2}(0, 1) $ is a $ \mathcal{F}_0 $ random variable.

%%%%%%%%%%%%%%%%%%%%%%%%%%%%%%%%%%%%%%%%%%%%%%%%%%%%%%%%%%%%%%%%%%%%%%%%%%

\section{Definitions and results.}\label{sec-Def-Results}
In this section we present the main defintions and results. We define the mild solution  in a  general framework as in \cite{Neerven-Evolution-Eq-08},
\begin{defn}\label{def-mild-solution}
Let $ X$ be an UMD-Banach space of type 2 and $H$ be a Hilbert space. Assume that $ u_0:\Omega\rightarrow X$ is strongly $ \mathcal{F}_0-$measurable\del{$ \in L^p(\Omega, \mathcal{F}_0, P; X)$}.  A strongly measurable $ \mathcal{F}_t-$adapted $ X$-valued
stochastic process, $(u(t), t\in [0, T])$, is called a mild solution of Eq.\eqref{FSBE-Evol-1} if
\begin{itemize}
\item (i) for all $ t\in [0, T]$, $ s \mapsto e^{-(t-s)A^{\alpha/2}}F(u(s))$ is in $ L^0(\Omega, L^1(0, t: X))$,
\item  (ii) for all $ t\in [0, T]$, $ s \mapsto e^{-(t-s)A^{\alpha/2}} G$ is $ H-$strongly measurable $ \mathcal{F}_t-$adapted and a.s. in the $\gamma-$Radonifying space; $ R(H, X)$,   
\item (iii) $ \forall t \in [0, T]$, the following equality holds in $ X$,  $ P-a.s.$ 
\begin{equation}\label{Eq-Mild-Solution}
u(t)= e^{-A^{\alpha/2} t}u_0 +
\int_0^t e^{-A^{\alpha/2} (t-s)}F(u (s))ds + \int_0^te^{-A^{\alpha/2} (t-s)}W(ds).
\end{equation}
\end{itemize}
\end{defn}

\begin{defn}\label{def-uniqueness}
We say that a pathwise uniqueness holds for Eq.\eqref{FSBE-Evol-1}  if for any two solutions $ (u^1(t), t\in [0, T])$ and $ (u^2(t), t\in [0, T])$ starting from the same initial data $ u_0$, we have
\begin{equation}
P[u^1(t)= u^2(t), \;\; \forall \; t\in[0, T]]=1.
\end{equation}
\end{defn}
First of all, we give the following auxilliary result,

\begin{theorem}\label{Theorem-disc}
Let $ T > 0$, $\alpha \in (\frac{3}{2}, 2)$, $ \delta \in (1-\frac{\alpha}{2}, \frac{\alpha-1}{2})$ and let $u_0$ satisfies Assumption $\mathcal{A}$. Then Eq.\eqref{Discr-Galerkin-FSBE-Evol-1} admits a unique $L^2-$mild solution $u_N:= (u_N(t), t\in [0, T])$ satisfying  $ \sup_{N} \sup_{t \in [0,T]} \vert u_N(t) \vert_{L^{2}} < \infty $. Moreover, for $ \frac74<\alpha< 2$ and $\delta\in (1-\frac{\alpha}{2}, \frac{2\alpha-3}{2})$, for almost all $\omega$, the map $ u_N: [0, T]\del{\times \Omega} \rightarrow C^{\delta}(0,1)$ is H\"older continuous of index $(\frac{\alpha-1-2\delta}{2\alpha})-$ and satisfies Est.\eqref{Eq-Assum-4-1}, with $   V:=C^{\delta}(0,1)$, i.e. for all most all $\omega\in \Omega,$
\begin{equation}\label{Eq-Assum-4-1-u-N}
\sup_N\sup_{t\in[0, T]}|u_N(t, \omega)|_{C^\delta}<\infty.
\end{equation}
\end{theorem}	

\del{\begin{theorem}\label{Theorem-disc}
Let $\alpha \in (1, 2]$ and $ \max\{2, \frac{1}{\alpha-1}\}\leq_1 q\leq \infty$. \del{ and let $u_0$ satisfies Assumption $\mathcal{A}$.} Then Eq. \eqref{Discr-Galerkin-FSBE-Evol-1} admits a unique mild solution $u_N:= (u_N(t), t\in [0, T])$ satisfying the Est. \eqref{Eq-reg-theta-mild} with $ V:= C^{\delta}(0,1)$. Moreover, for $ \frac74<\alpha< 2$ and $\delta\in [0, \frac{2\alpha-3}{2}$, for almost all $\omega$, the map $ u_N: [0, T]\del{\times \Omega} \rightarrow C^{\delta}(0,1)$ is H\"older continuous of index $(\frac{\alpha-1-2\delta}{2\alpha})-$ and satisfies   Est. \eqref{Eq-Assum-4-1}.
\end{theorem}}
Our main results are obtained  under the conditions $0<T<\infty$, $ \alpha \in (\frac74, 2)$,  $\delta\in (1-\frac{\alpha}{2}, \frac{2\alpha-3}{2})$ and  that $u_0$ satisfies Assumption $\mathcal{A}$. We have

\begin{theorem}\label{Main-Theorem-Exit-Uniq}
 The fractional stochastic Burgers type Equation \eqref{FSBE-Evol-1} with initial condition $u_0$, admits a unique mild solution $ u: [0, T]\times\Omega \rightarrow C^\delta(0,1)$.
Moreover, almost surely, the paths of $u$ are H\"older continuous of order; $(\frac{\alpha-1-2\delta}{2\alpha})-$.
\end{theorem}

\begin{theorem}\label{main-result-Galerkin-approx}
There exists a $ \mathcal{F}/\mathcal{B}([0,\infty))$-measurable mapping $C: \Omega \rightarrow \mathbb{R}_+^*$, such that almost surely,
\begin{equation}\label{main-est-Galerkin-approx}
\sup_{t \in [0,T]} |u(t) - u_{N}(t) |_{C^{\delta}} \leq C(\omega) \; N^{-(\frac{\alpha-1}{2}-\delta)-},
\end{equation}
where $u$ is the unique solution of Eq.\eqref{FSBE-Evol-1} with initial condition $u_0$ and $ u_N$ is the Galerkin approximation solution of Eq.\eqref{Discr-Galerkin-FSBE-Evol-1}.
\end{theorem}

\begin{theorem}\label{main-result-fully}
There exists a $ \mathcal{F}/\mathcal{B}([0,\infty))$-measurable mapping $C: \Omega \rightarrow \mathbb{R}_+^*$, such that almost surely,
\begin{equation}
\sup_{t_{m} \in [0,T]} | u(t_{m}) - u_{N}^{m} |_{C^\delta} \leq C(w) \; \left( (\Delta t)^{(\frac{\alpha -1-2\delta}{2\alpha})-}+  N^{-(\frac{\alpha -1}{2}-\delta)+}\right),
\end{equation}
for every $ N, M \geq 1.$\del{ where $ \xi = (\frac{\alpha -1}{2} - \delta )-$.}  
\end{theorem}

%%%%%%%%%%%%%%%%%%%%%%%%%%%%%%%%%%%%%%%%%%%%%%%%%%%%%%%%%%%%%%%
\section{Some Estimates  for the stochastic terms}\label{sec-Est-Stoch-Terms}
This section is mainly devoted to study the rate, the different kinds of  convergence, in particular the pathwise convergence, and the regularities of the Galerkin approximation of the stochastic terms;  
$W_N$ given by \eqref{W-N-series} and
\begin{equation}\label{OU-process-W-N}
\mathcal{W}_N(t):= P_N\mathcal{W}(t)= \int_0^te^{-A^{\alpha/2} (t-s)}W_N(ds)= \sum_{k=1}^{N}\int_0^te^{-\lambda^{\alpha/2}_k (t-s)}d\beta_{k}(t)e_{k},
\end{equation}
where  $\mathcal{W}$ is the Ornstein-Uhlenbeck stochastic process given by \eqref{OU-process} and  $W$ is the Wiener process given by \eqref{W-series}. The Ornstein-Uhlenbeck stochastic process $\mathcal{W}$ has been studied, for example, in \cite{BrzezniakDebbi1, DebbiL2Solution, DebbiDozzi1}, where the authors proved that for $ \alpha \in (1, 2]$ and $ G=I$, the Ortein Uhlenbeck process \eqref{OU-process} is well defined as an $L^2-$valued  stochastic process with $C^{\frac{\alpha-1}{2\alpha}}_tC^{\frac{\alpha-1}{2}}_x-$H\"older continuous trajectories.  In \cite{Debbi-scalar-active}, the author proved that for $ \alpha \in (1, 2]$, $ \mathcal{W}\in L^p(\Omega; C_t\times H_{q}^\beta)(C(0, 1))$, $ 2\leq q<\infty$, $ \beta\geq 0$, $ p\geq 2$ and $ C(0, 1)$ is the unit circle. The following Lemma is a generalization of \cite[Proposition 4.2.]{Blomker-Jentzen},

\begin{lem}\label{lem-est-stoch-main}
Let $ \alpha \in (1, 2]$, $0<\beta<\frac{\alpha-1}2$, $q\geq 2$ and let $p_0>\frac{2\alpha}{\alpha-1-2\beta}$ be fixed. Then for \del{ $\tau \in (0, \frac{\alpha-1}{2}-(\beta+\frac \alpha p!!)$ and} $ p\geq1$, there exist $C_{\alpha,\beta, q, p}>0$, s.t.
\begin{equation}\label{Eq-est-E-W-N}
\sup_{N \in \mathbb{N}}\mathbb{E} \Big[ |\mathcal{W}_{N} |^p _{C_tH^{\beta}_q}  +  N^{{p(\frac{\alpha-1}{2}-(\beta+\frac \alpha {p_0}))-}}|(1-P_N) \mathcal{W} |^p _{C_tH^{\beta}_q}+  | \mathcal{W} |^p _{C_tH^{\beta}_q}\Big]< C_{\alpha,\beta, q, p}.
\end{equation}
\end{lem}

\begin{proof}
To prove Lemma \ref{lem-est-stoch-main}, it is sufficient to prove the result for the second term in the LHS of  \eqref{Eq-est-E-W-N}. The remaining estimates can be easily obtained by following a similar, but simple culculus without considering any power of $N$.

The  proof is given in two steps. In the first one, we prove, that for $ \alpha \in (1, 2]$, $0<\beta<\frac{\alpha-1}2$, $q\geq 2$ and $ p>\frac{2\alpha}{\alpha-1-2\beta}$,  there exist $C_{\alpha,\beta, q, p}>0$ and $\xi_p \in (0, \frac{\alpha-1}{2}-(\beta+\frac \alpha p))$, such that the following estimate holds:
\begin{equation}\label{Eq-est-E-W-N-Inter}
\sup_{N \in \mathbb{N}}\mathbb{E} \Big[ N^{\xi_pp\del{p(\frac{\alpha-1}{2}-(\beta+\frac \alpha p))-}}|(1-P_N) \mathcal{W} |^p _{C_tH^{\beta}_q}\Big]< C_{\alpha,\beta, q, p},
\end{equation}
In the second step, we show that Est.\eqref{Eq-est-E-W-N-Inter} is true for all $p\geq 1$ and for a universal $\xi$. 

\vspace{0.25cm}
 
{\bf Step1.}  Recall that 
\begin{equation}
(1-P_N)\mathcal{W}(t)= \int_{0}^{t} e^{A^{\alpha/2}(t-s)}(1-P_N)W(ds).
\end{equation}
Using the factorization method, see e.g. \cite{BrzezniakDebbi1, Daprato}, we represent $ (1-P_N)\mathcal{W}$ as,
\begin{equation}
(1-P_N)\mathcal{W}(t)= \int_{0}^{t} (t-s)^{\nu -1} e^{A^{\alpha/2}(t-s)} Y^{N}(s) ds,
\end{equation}
\begin{equation}\label{Factoz-W-N}
Y^{N}(t)= \int_{0}^{t} (t-s)^{-\nu} e^{A^{\alpha/2}(t-s)}(1-P_N)W(ds),  
\end{equation}
with $\nu \in (0, 1)$.  Thanks to Est.\eqref{eq-semigp-property}\del{Lemma \ref{Lem-semigroup-N}} and by application of H\"older inequality and the fact that we can choose $ \frac1p+\frac\beta\alpha<\nu<1$, we get
\begin{eqnarray}\label{Eq-fact-N-1}
\mathbb{E}| (1-P_N)\mathcal{W} |_{C_tH^{\beta}_q}^p&=& \mathbb{E}\big(\sup_{t\in[0, T]}| A^{\frac\beta2}\int_0^t(t-s)^{\nu -1} e^{A^{\alpha/2}(t-s)} Y^{N}(s) ds|^p_{L^q}\big)\nonumber\\
&\leq & \mathbb{E}\sup_{t\in[0, T]}\big(\int_0^t(t-s)^{\nu -1} |A^{\frac\beta2}e^{A^{\alpha/2}(t-s)} Y^{N}(s)|_{L^q} ds\big)^p\nonumber\\
&\leq & C \mathbb{E}\sup_{t\in[0, T]}\big(\int_0^t(t-s)^{\nu -1-\frac\beta\alpha} |Y^{N}(s)|_{L^q} ds\big)^p\nonumber\\
&\leq & C\sup_{t\in[0, T]}\big(\int_0^t(t-s)^{(\nu -1-\frac\beta\alpha)\frac{p}{p-1}}ds\big)^{(p-1)} \mathbb{E} \int_0^T |Y^{N}(s)|^p_{L^q} ds\nonumber\\
&\leq & C T^{(1+(\nu -1-\frac\beta\alpha)\frac{p}{p-1})(p-1)} \mathbb{E} \int_0^T |Y^{N}(s)|^p_{L^q} ds.
\end{eqnarray}
Moreover, using the stochastic isometry, the estimate $|e_k|_{L^q}\leq 1$ and Lemma \ref{lem-elementary-1},\del{the inequality $ e^{-x}\leq c_\gamma x^{-\gamma}$,} with $ \frac1\alpha<\gamma <1-2\nu$ and $\nu<\frac 12-\frac{1}{2\alpha}$, we obtain
\begin{eqnarray}\label{Eq-fact-N-2}
\mathbb{E} \int_0^T |Y^{N}(s)|^p_{L^q} ds &\leq & \mathbb{E} \int_0^T |\int_{0}^{t} (t-s)^{-\nu} e^{A^{\alpha/2}(t-s)} (1-P_N)W(ds)|^p_{L^q} ds \nonumber\\
&\leq & C\int_0^T (\int_{0}^{t} (t-s)^{-2\nu} \sum_{k=N+1}^\infty e^{-2\lambda_k^\frac\alpha2(t-s)}|e_k|^2_{L^q}ds)^{\frac p2} dt\nonumber\\
&\leq & C\int_0^T (\int_{0}^{t} (t-s)^{-2\nu} \sum_{k=N+1}^\infty e^{-2\lambda_k^\frac\alpha2(t-s)}ds)^{\frac p2} dt \nonumber\\
&\leq & C\int_0^T (\int_{0}^{t} (t-s)^{-2\nu} \sum_{k=N+1}^\infty (2\lambda_k)^\frac{-\gamma \alpha}2(t-s)^{-\gamma}ds)^{\frac p2} dt \nonumber\\
&\leq & C_{\alpha, \gamma}\int_0^T (\int_{0}^{t} (t-s)^{-2\nu-\gamma}ds \sum_{k=N+1}^\infty k^{-\gamma \alpha})^{\frac p2} dt
\end{eqnarray}
Remark that $ \gamma$ exists thanks to the condition $ \nu<\frac 12-\frac{1}{2\alpha}$ and $ \alpha>1$. It is also easy to see that, thanks to the choice of  $ \nu$ and $ \gamma$, the integral in the RHS of the last inequality of Est.\eqref{Eq-fact-N-2} converges. Now, let $\xi\in (0, \alpha\gamma -1)$, then $ \sum_{k=N+1}^\infty k^{-\gamma \alpha} \leq N^{-\xi}\sum_{k=N+1}^\infty k^{-\gamma \alpha+\xi}\leq C_{\alpha, \gamma, \xi} N^{-\xi}$. Hence

\begin{eqnarray}\label{Eq-fact-N-2-aux}
\mathbb{E} \int_0^T |Y^{N}(s)|^p_{L^q} ds &\leq & 
C_{\alpha, \gamma, \beta}N^{-\xi\frac p2}(\int_0^T t^{(1-2\nu-\gamma)\frac p2}dt)( \sum_{k=N+1}^\infty k^{-\gamma \alpha+\xi})^{\frac p2} \nonumber\\
&\leq & 
C_{\alpha, \gamma, \beta}N^{-\xi\frac p2} T^{(1-2\nu-\gamma)\frac p2+1} ( \sum_{k=1}^\infty k^{-\gamma \alpha+\xi})^{\frac p2} 
\nonumber\\
&\leq & C_{\alpha, \gamma, \beta, p, \nu, \xi, T}N^{-\xi\frac p2}.
\end{eqnarray}
Now to combine \eqref{Eq-fact-N-1} and \eqref{Eq-fact-N-2-aux}, we have to assume that $ \frac1p + \frac\beta\alpha< \nu<\frac 12-\frac{1}{2\alpha}$. The existence of $\nu$ is guaranted thanks to the conditions  $0<\beta<\frac{\alpha-1}2$ and $ p>\frac{2\alpha}{\alpha-1-2\beta}$. \del{Now, by combining \eqref{Eq-fact-N-1} and \eqref{Eq-fact-N-2}, we deduce that the first term in the LHS of \eqref{Eq-est-E-W-N} is finite.} Therefore,
\begin{equation}
\mathbb{E} |(1-P_N)\mathcal{W} |^p_{C_tH^{\beta}_q}
\leq C_{\alpha, \gamma, \beta, p, \nu, \xi, T}N^{-p\frac \xi2}.
\end{equation}
According to the values of the parameters $\xi$, $ \gamma$  and $\nu$, we deduce that $\frac \xi2\in (0, \frac{\alpha-1}{2}-(\beta+\frac \alpha p))$, hence $\xi$ depends, in particular, on $p$, so  let us denote $\frac \xi2$ by $ \xi_p$.
\del{the maximum value possible for $\xi$ is $\frac{\alpha-1}{2}-(\beta+\frac \alpha p)$.}

\vspace{0.25cm}

\noindent {\bf Step2.} 
Let us fix $ p_0> \frac{2\alpha}{\alpha-1-2\beta}$, using H\"older inequality, we deduce Est.\eqref{Eq-est-E-W-N-Inter}, for all $p\leq p_0$, with $\xi_{p}$ being replaced by $\xi_{p_0}$ and the constant depends also on $p_0$ and $p$. Now, for $ p\geq p_0$, then Est.\eqref{Eq-est-E-W-N-Inter} holds and thanks to the inclusion  $ (0, \frac{\alpha-1}{2}-(\beta+\frac \alpha p_0))\subset (0, \frac{\alpha-1}{2}-(\beta+\frac \alpha p))$, it is sufficient to take $\xi_p$  equals to    $ \xi_{p_0}$. 
\end{proof}

\begin{coro}\label{Coro-lem-1}
Let $ \alpha \in (1, 2]$, $0<\delta<\frac{\alpha-1}{2}$ and  $p_0>\frac{2(\alpha+1)}{\alpha-1-2\delta}$. Then for $p\geq 1$, there exists $C_{\alpha,\delta, p}>0$, s.t.
\begin{equation}\label{Eq-est-E-W-N-C-delta-1}
\sup_{N \in \mathbb{N}}\mathbb{E} \Big[ | \mathcal{W}_{N} |^p _{C_tC^\delta}  +  N^{{p(\frac{\alpha-1}{2}-(\delta+\frac{ \alpha +1}{p_0}))-}}|(1-P_N) \mathcal{W} |^p _{C_tC^\delta}+  | \mathcal{W} |^p _{C_tC^\delta}\Big]< C_{\alpha,\delta, p},
\end{equation}
\end{coro}

\begin{proof}
It is easy to see that Est.\eqref{Eq-est-E-W-N} is valid for the special case $ \beta = \delta+\frac 1{p_0}$, with  $\delta < \frac{\alpha-1}{2}$, $ p_0>\frac{2(\alpha+1)}{\alpha-1-2\delta}$  and $q=p_0$. Hence, we get
\begin{equation}\label{Eq-est-E-W-N-p}
\sup_{N \in \mathbb{N}}\mathbb{E} \Big[ | \mathcal{W}_{N} |^p _{C_tH^{\delta+\frac1{p}}_{p}}  +  N^{{p(\frac{\alpha-1}{2}-(\delta+\frac{ \alpha+1}{p_0}))-}\del{\tau p}\del{p(\frac{\alpha-1}{2}-(\beta+\frac \alpha p))-}}|(1-P_N) \mathcal{W} |^p _{C_tH^{\delta+\frac1{p}}_{p}}+  | \mathcal{W} |^p _{C_tH^{\delta+\frac1{p}}_{p}}\Big]< C_{\alpha,\delta, p}.
\end{equation}
Thanks to the embedding $ H^{\delta+\frac1{p_0}}_{p_0} \hookrightarrow C^\delta$, see Theorem \ref{main-sob-embedding},       we get Est.\eqref{Eq-est-E-W-N-C-delta-1}. 
\end{proof}

\begin{coro}\label{Coro-lem-2}
Let $ \alpha \in (1, 2]$ and $\delta \in (0, \frac{\alpha-1}{2})$. Then  there exists  a finite  positive random variable $C_{\alpha,\delta}$, s.t. for almost surely,
\begin{equation}\label{Eq-est-E-1-P-N-C-delta}
\sup_{N \in \mathbb{N}}\big[N^{(\frac{\alpha-1}{2}-\delta)-}|(1-P_N) \mathcal{W}(\omega)|_{C_tC^\delta}\big]\leq C_{\alpha,\delta}(\omega).
\end{equation}
\end{coro}
\begin{proof}
Using Lemma \ref{lem-est-stoch-main} and Lemma  \ref{appendix-A-4}, we deduce for a given  $p_0>\frac{2\alpha}{\alpha-1-2\beta}$, that almost surely,
\begin{equation}\label{Eq-Assum-3-proo}
\sup_{N\in \mathbb{N}}\sup_{t\in (0, T]}\big(N^{[\frac{\alpha-1}{2}-(\delta +\frac{\alpha +1}{p_0})]-} \del{t^\alpha}|(1-P_N) \mathcal{W}(\omega, t)|_{C^\delta}\big)<\infty. 
\end{equation}
\del{that almost surely,
\begin{equation}\label{Eq-Assum-3-proo-beta}
\sup_{N\in \mathbb{N}}\sup_{t\in (0, T]}\big(N^{(\frac{\alpha-1}{2}-(\beta+\frac \alpha {p_0}))-} \del{t^\alpha}|(1-P_N)\mathcal{W}(\omega, t)|_{H^\beta_q}\big)<\infty. 
\end{equation}
In addition, as $ \delta>1-\frac\alpha2$, 
we have to choose $p_0$ such that $ 1-\frac\alpha2<\frac{\alpha-1}{2}-\frac1{p_0}$, which means  that $ \alpha>3/2$ and $ p_0>\max\{\frac{2(\alpha+1)}{\alpha-1-2\delta}, \frac1{\alpha-3/2}\}$.}
For $p_0$ large ($ \frac{\alpha+1}{p_0}=\epsilon$) then there exists a random variable $C_{\alpha, \delta,\epsilon}$, such that  \eqref{Eq-est-E-1-P-N-C-delta} is fulfilled.\del{replaced by $ (\frac{\alpha-1}{2}-\delta)-$.}

\end{proof}

\begin{lem}\label{lem-Soundous}
Let $ \alpha \in (1, 2]$ and $0<\delta<\frac{\alpha-1}{2}$. Then  there exists a finite positive random variable $C_{\alpha,\delta}$, s.t. for almost surely,
\begin{equation}\label{Eq-est-E-W-N-C-delta}
\sup_{N \in \mathbb{N}}\sup_{t\in [0, T]} | \mathcal{W}_{N}(t, \omega)| _{C^\delta} < C_{\alpha,\delta}(\omega).
\end{equation}
\end{lem}
\begin{proof} Thanks to Lemma \ref{lem-e-k-holder}, we have 
\begin{eqnarray}
| \mathcal{W}_{N}(t, \omega)| _{C^\delta}&=& |\sum_{k=1}^N (\int_0^t e^{-(t-s)\lambda_k^{\alpha/2}}d\beta_k(s))(\omega)e_k| _{C^\delta} \leq \sum_{k=1}^N |(\int_0^t e^{-(t-s)\lambda_k^{\alpha/2}}d\beta_k(s))(\omega)||e_k| _{C^\delta}\nonumber\\
&\leq &  \sum_{k=1}^\infty |(\int_0^t (k\pi)^\delta e^{-(t-s)(k\pi)^{\alpha}}d\beta_k(s))(\omega)|.
\end{eqnarray}
We define 
\begin{equation}
 C(t, \omega):= \sum_{k=1}^\infty |(\int_0^t (k\pi)^\delta e^{-(t-s)(k\pi)^{\alpha}}d\beta_k(s))(\omega)|.
\end{equation}
It is well known that the process $C(t, \omega)$ is well defined  provided that $ \sum_{k=1}^\infty \int_0^t (k\pi)^{2\delta} e^{-2(t-s)(k\pi)^{\alpha}}ds$.  This last condition is satisfied  by using Lemma 
\ref{lem-elementary-1}, with $ \frac{2\delta +1}{\alpha}< \gamma <1$. Moreover,  $ C(\cdot, \omega)$ has continuous trajectories on $[0, T]$.  Therefore, the random variable:
\begin{equation}
C_{\alpha,\delta}(\omega):= \del{c_{\alpha,\delta}}\sup_{t\in [0, T]}C(t, \omega)
\end{equation}
 \del{, with $c_{\alpha,\delta}>0$,} exists, is positive and finite and we have,
\begin{eqnarray}
\sup_{t\in [0, T]}| \mathcal{W}_{N}(t, \omega)|_{C^\delta}
&\leq &  C_{\alpha,\delta}(\omega).
\end{eqnarray}
\end{proof}

\begin{lem}\label{auxiliary-3}
Let $\alpha\in (1, 2]$, $ 0\leq \delta<\frac{\alpha-1}{2}$ and fix $ N \in \mathbb{N}$.  The stochastic process: $ \mathcal{W}_N:[0, T]\times \Omega \rightarrow C^{\delta}(0,1)$  has H\"older continuous sample paths of degree $\frac12-$. \del{Moreover, we have for $p$, there exists $C>0$, such that 
\begin{equation}\label{Eq-holder-expectation-new}
\left( \mathbb{E} \vert  \mathcal{W}_{N}(t_{2}) - \mathcal{W}_{N}(t_{1}) \vert_{C^{\delta}}^{p} \right) ^{\frac{1}{p}} \leq C_{N, p} \vert t_{1} -t_{2} \vert^{\frac{\tau'}{2}},
\end{equation} }

 \del{$ \tau \in (0, \frac{\alpha-1-2\delta}{4\alpha})$.} 
\del{\begin{itemize}
\item $ \tau\in (0, 1)$, for any $\tau$ satisfying that the series $\sum_{k=1}^{N}q_k^2 k^{-\alpha(1-\frac{2}{\alpha}(\delta))+ 2\alpha \tau'}$ converges, ($\delta\in (0, 1)$),
\item $ \tau \in (0, \frac{\alpha-1-2\delta}{2\alpha})$ with $ 0\leq \delta<\frac{\alpha-1}{2}$, when $ G=I$.
\end{itemize}}
\end{lem}

\begin{proof}
Our main tool here is Kolomogorov-Centsov Theorem. First, we prove that  for large $p\in [1, \infty)$ and  for $ t_{1} , t_{2} \in (0,T]$, there exist  positive constants $ C_{\delta}$ and $\tau'\in (0, 1)$ s.t.
\begin{equation}\label{Eq-holder-expectation}
\left( \mathbb{E} \vert  \mathcal{W}_{N}(t_{2}) - \mathcal{W}_{N}(t_{1}) \vert_{C^{\delta}}^{p} \right) ^{\frac{1}{p}} \leq C_{N, p} \vert t_{1} -t_{2} \vert^{\frac{\tau'}{2}},
\end{equation} 
Let $ x, y \in [0,1] $ and $ t_{1}, t_{2} \in (0,T]$. \del{ and we consider in first step $p=2$. }Then
\begin{eqnarray}
( (\mathcal{W}_{N}(t_{2})(x)&-& \mathcal{W}_{N}(t_{1})(x)) - (\mathcal{W}_{N}(t_{2})(y)-\mathcal{W}_{N}(t_{1})(y)) )\nonumber \\
&=&
\sum_{k=1}^{N} \left( \int_{0}^{t_{2}} e^{-\lambda_{k}^{\frac{\alpha}{2}}(t_{2}-s)} d\beta_{k}(s) - \int_{0}^{t_{1}} e^{-\lambda_{k}^{\frac{\alpha}{2}}(t_{1}-s)} d\beta_{k}(s) \right)\left(e_{k}(x)-e_{k}(y) \right).  
\end{eqnarray}
Thanks to the fact that the elements of the sequence $ (\beta_k)_k $ are independent and to  Lemma \ref{appendix-A-2}, we get for every $\tau'\in (0, 1)$,  
\begin{eqnarray}\label{label-1-1}
\mathbb{E}( (\mathcal{W}_{N}(t_{2})(x) &-& \mathcal{W}_{N}(t_{1})(x)) - (\mathcal{W}_{N}(t_{2})(y)-\mathcal{W}_{N}(t_{1})(y)) )^{2}\nonumber \\
&= &
\sum_{k=1}^{N} \vert e_{k}(x)-e_{k}(y) \vert^{2} \mathbb{E}\left( \int_{0}^{t_{2}} e^{-\lambda_{k}^{\frac{\alpha}{2}}(t_{2}-s)} dB_{k}(s) - \int_{0}^{t_{1}} e^{-\lambda_{k}^{\frac{\alpha}{2}}(t_{1}-s)} dB_{k}(s) \right) ^{2} \nonumber \\
&\leq &
\sum_{k=1}^{N} \vert e_{k}(x)-e_{k}(y) \vert^{2} \lambda_{k} ^{-\frac{\alpha}{2}(1-\tau')}|t_{2}-t_{1}|^{\tau'}.  
\end{eqnarray}
Let $ 0<\epsilon<\frac1p.$ Using the properties of the trigonometric function $sin(k\pi x)$, a simple calculus yields to 
\begin{eqnarray}\label{label-1-2}
\mathbb{E}( (\mathcal{W}_{N}(t_{2})(x) &-& \mathcal{W}_{N}(t_{1})(x)) - (\mathcal{W}_{N}(t_{2})(y)-\mathcal{W}_{N}(t_{1})(y)) )^{2}\nonumber \\
&\leq &
\sum_{k=1}^{N} \vert e_{k}(x)-e_{k}(y) \vert^{2(\delta+\frac1p+\epsilon)} (\vert e_{k}(x)|+ |e_{k}(y)) \vert^{2-2(\delta+\frac1p+\epsilon)} \lambda_{k} ^{-\frac{\alpha}{2}(1-\tau')}|t_{2}-t_{1}|^{\tau'} \nonumber \\
&\leq &
4 \sum_{k=1}^{N} k^{2(\delta+\frac1p+\epsilon)} \vert x-y \vert^{2(\delta+\frac1p+\epsilon)}\lambda_{k} ^{-\frac{\alpha}{2}(1-\tau')}|t_{2}-t_{1}|^{\tau'} \nonumber \\
& \leq &  C \big( \sum_{k=1}^{N} k^{-\alpha(1-\tau' -\frac{2}{\alpha}(\delta+\frac1p+\epsilon))}\big)|t_{2}-t_{1}|^{\tau'} \vert x-y \vert^{2(\delta+\frac1p+\epsilon)}
\nonumber \\
& \leq &  C \big( \sum_{k=1}^{N} k^{-\alpha(1-\tau' -\frac{2}{\alpha}(\delta+\frac1p+\epsilon))}\big)|t_{2}-t_{1}|^{\tau'} \vert x-y \vert^{2(\delta+\frac1p+\epsilon)}
\nonumber \\
& \leq &  C_{N} |t_{2}-t_{1}|^{\tau'} \vert x-y \vert^{2(\delta+\frac1p+\epsilon)}.  
\end{eqnarray}
Furthermore,  using  Lemma \ref{appendix-A-2} and the properties of the trigonometric functions,   we get 
\begin{eqnarray}\label{auxiliarly-3-2-1}
\mathbb{E}( \mathcal{W}_{N}(t_{2})(x)& - & \mathcal{W}_{N}(t_{1})(x) )^{2} \nonumber \\
& =  &
 \sum_{k=1}^{N}\vert e_{k}(x) \vert^{2} \mathbb{E} \left( \int_{0}^{t_{2}} e^{-\lambda_{k}^{\frac{\alpha}{2}}(t_{2}-s)} dB_{k}(s) - \int_{0}^{t_{1}} e^{-\lambda_{k}^{\frac{\alpha}{2}}(t_{1}-s)} dB_{k}(s) \right) ^{2} \nonumber \\
&\leq &
C \sum_{k=1}^{N} k ^{-\alpha(1-\tau')}|t_{2}-t_{1}|^{\tau'}\leq C_N |t_{2}-t_{1}|^{\tau'}. 
\end{eqnarray}
\del{provided that  $ 0 < \tau' < 1-\frac{2}{\alpha}(\delta+\frac1p+\epsilon+\frac12)$, $ \epsilon+\frac1p<\frac{\alpha-1}{2}-\delta$ and $\delta < \frac{\alpha-1}{2}$.}
Thus, using  Theorem \ref{append-lem-gauss}\del{Lemma \ref{auxiliary-3}}, we infer that  
\begin{eqnarray}\label{label-1}
\mathbb{E}( (\mathcal{W}_{N}(t_{2})(x) &-& \mathcal{W}_{N}(t_{1})(x)) - (\mathcal{W}_{N}(t_{2})(y)-\mathcal{W}_{N}(t_{1})(y)) )^{p}\nonumber \\
&\leq & 
p!\big(\mathbb{E}( (\mathcal{W}_{N}(t_{2})(x) - \mathcal{W}_{N}(t_{1})(x)) - (\mathcal{W}_{N}(t_{2})(y)-\mathcal{W}_{N}(t_{1})(y)) )^{2}\big)^{\frac{p}{2}}\nonumber \\
&\leq & C_{N, p} |t_{2}-t_{1}|^{\tau'\frac{p}{2}} \vert x-y \vert^{p(\delta+\frac1p+\epsilon)}.  
\end{eqnarray}
And 
\begin{eqnarray}\label{auxiliarly-3-2}
\mathbb{E}( \mathcal{W}_{N}(t_{2})(x)- \mathcal{W}_{N}(t_{1})(x) )^{p} & \leq & p!\big(\mathbb{E}( \mathcal{W}_{N}(t_{2})(x) -  \mathcal{W}_{N}(t_{1})(x) )^{2}\big)^{\frac p2} \leq   C_{N, p} |t_{2}-t_{1}|^{\tau'\frac p2}.\nonumber\\ 
\end{eqnarray}
Thanks to the Sobolev embedding  $ H_p^{\delta+\frac1p} \hookrightarrow C^{\delta}$, see e.g. Theorem \ref{main-sob-embedding}\del{ (with $ r$ in Formula is replaced by $ \delta +\frac{1}{p}\geq 2$)}, Est. \eqref{label-1}, Est.\eqref{auxiliarly-3-2} and Lemma \ref{appendix-A-1}, we obtain
\begin{eqnarray}
 \mathbb{E} \vert  \mathcal{W}_{N}(t_{2}) &-& \mathcal{W}_{N}(t_{1}) \vert_{C^{\delta}}^{p} ) 
\leq 
C_N ( \int_{0}^{1} \mathbb{E}( \mathcal{W}_{N}(t_{2})(x) -  \mathcal{W}_{N}(t_{1})(x) )^{p} dx \nonumber \\
& + &
 \int_{0}^{1} \int_{0}^{1} \frac{\mathbb{E}( (\mathcal{W}_{N}(t_{2})(x) - \mathcal{W}_{N}(t_{1})(x)) - (\mathcal{W}_{N}(t_{2})(y)-\mathcal{W}_{N}(t_{1})(y)) )^{p}}{\vert x-y \vert^{2+\delta p}} dx \; dy )\del{\nonumber \\
&\leq &
C (t_{2}-t_{1})^{\frac{\tau' p}{2}} \left( 1 + \int_{0}^{1} \int_{0}^{1} \vert x-y \vert^{-(2+\delta p - \beta p)} dx \; dy \right) \leq C_{\delta ,p} (t_{2}-t_{1})^{\frac{\tau' p}{2}}.}  \nonumber\\
&\leq &
C_{N, p} |t_{2}-t_{1}|^{\frac{\tau' p}{2}} \left( 1 + \int_{0}^{1} \int_{0}^{1} \vert x-y \vert^{-(1- p\epsilon)} dx \; dy \right) \nonumber \\
& \leq & C_{N, p} (t_{2}-t_{1})^{\frac{\tau' p}{2}}.
\end{eqnarray}
\end{proof}
 
\begin{coro}\label{Coro-Reg-W} 
Let $\alpha\in (1, 2]$ and $ 0\leq \delta<\frac{\alpha-1}{2}$. The Ornstein-Uhlenbeck stochastic  process $ \mathcal{W}$ has a continuous version, we still denote by 
 $ \mathcal{W}:[0, T]\times \Omega \rightarrow C^{\delta}(0,1)$ with  H\"older continuous sample paths of degree $\big(\frac{ \alpha-1-2\delta}{2\alpha}\big)-$.
\end{coro}
\begin{proof}
First, we find $\tau'$ such that Est.\eqref{Eq-holder-expectation} holds with a constant in the RHS which is independent of $N$, i.e. we prove that for large $p\in [1, \infty)$ and  for $ t_{1} , t_{2} \in (0,T]$, there exist  positive constants $ C_{\alpha, \delta, p}$ and $\tau'\in (0, 1)$ s.t.
\begin{equation}\label{Eq-holder-expectation-univ}
\left( \mathbb{E} \vert  \mathcal{W}_{N}(t_{2}) - \mathcal{W}_{N}(t_{1}) \vert_{C^{\delta}}^{p} \right) ^{\frac{1}{p}} \leq C_{\alpha, \delta, p} \vert t_{1} -t_{2} \vert^{\frac{\tau'}{2}},
\end{equation}  
To this aim, it is suffisant to  follow the same calculus as in the proof of Lemma \ref{auxiliary-3} and to  choose $ \tau'$ such that  $\sum_{k=1}^{\infty} k^{-\alpha(1-\tau' -\frac{2}{\alpha}(\delta+\frac1p+\epsilon))}<\infty$. We consider $p$ large and $\epsilon$ small such that $ 0<\frac1\alpha(\frac 2p+2\epsilon)<\epsilon'$ for a given $\epsilon'$. We have 
\begin{equation}
\sum_{k=1}^{\infty} k^{-\alpha(1-\tau' -\frac{2}{\alpha}(\delta+\frac1p+\epsilon))}\leq \sum_{k=1}^{\infty} k^{-\alpha(1-\tau' -\frac{2\delta}{\alpha}+\epsilon'))}<\infty,
\end{equation} 
provided $\alpha(1-\tau' -\frac{2\delta}{\alpha}-\epsilon')>1$. So it is suffisant to take $0<\tau'<\frac{\alpha-1-2\delta}{\alpha}$.

\noindent Now, we take $p_0>\frac{2(\alpha+1)}{\alpha-1-2\delta}$ and we  use  Corollary \ref{Coro-lem-1}, and Est. \eqref{Eq-holder-expectation-univ}, we infer the existence of $ C_{\delta, p}>0 $ s.t for all $ N\in \mathbb{N}$,
\begin{eqnarray}\label{Eq-Reg-W-1}
\mathbb{E} | \mathcal{W}(t_{2}) - \mathcal{W}(t_{1}) |_{C^{\delta}}^{p}  &\leq &  \mathbb{E} | (1-P_N)(\mathcal{W}(t_{2}) - \mathcal{W}(t_{1})) |_{C^{\delta}}^{p} + \mathbb{E} | \mathcal{W}_{N}(t_{2}) - \mathcal{W}_{N}(t_{1}) |_{C^{\delta}}^{p} \nonumber\\
&\leq & C_{\alpha,\delta,p} N^{-{p(\frac{\alpha-1}{2}-(\delta+\frac{ \alpha +1}{p_0}))+}}+ C_{\delta ,p} \vert t_{1} -t_{2} \vert^{\frac{\tau'p}{2}}.
\end{eqnarray} 
The result is easily deduced making $N\rightarrow \infty$ and  applying  Kolomogorov-Centsov Theorem.
\end{proof}

%%%%%%%%%%%%%%%%%%%%%%%%%%%%%%%%%%%%%%%%%%%%%%%%%%%%%%%%%%%%%%%%%%%%%%%%%%

\section{Some auxilliary results}\label{sec-deter-Burgers-Eq}
In this section we provide no classical results to estimate the nonlinear term. We mainly focus on nonlinear term of Burgers Equation, i.e. for $F$ given by $f(x)=x^2$.\del{ We end this section by extending the results} Let $v^N: (0, T) \rightarrow L^2(0, 1)$ be a sequence of  continuous functions. We define  
\begin{equation}
y^N(t):=  \int_0^te^{-A^{\alpha/2}(t-s)}P_NF(v^N(s))ds.
\end{equation}
  
\begin{lem}
Assume $ \alpha\in (\frac74, 2)$ and  that the sequence  $(v^N)_N$ satisfies  
\begin{equation}\label{eq-bound-v-n-l-2}
\sup_{N}\sup_{t\in [0, T]}|v^N(t)|_{L^2}<\infty.
\end{equation}
Then 
\begin{equation}
\sup_{N}\sup_{t\in[0,T]}|y^N(t)|_{L^4}<\infty.
\end{equation} 
\end{lem}

\begin{proof}
Using Lemma \ref{lem-est-sg-l-2-l-4}, Lemma \ref{Lem-Est-1-PN} and  \cite[Lemma 2.11]{BrzezniakDebbi1}(see e.g. Lemma \ref{Lemma-DeBre}), we get 
\begin{eqnarray}
|y^N(t)|_{L^4} & \leq & \int_0^t|e^{-A^{\alpha/2}(t-s)}P_NF(v^N(s))|_{L^4}ds  \nonumber\\
&\leq &  \int_0^t\Vert e^{-A^{\alpha/2}\frac{(t-s)}{2}}\Vert_{\mathcal{L}(L^2, L^4)} \Vert P_N \Vert_{\mathcal{L}(L^{2})}|e^{-A^{\alpha/2}\frac{(t-s)}{2}}F(v^N(s))|_{L^2}ds\nonumber\\ 
 &\leq & C\int_0^t (t-s)^{-\frac1{4\alpha}-} (t-s)^{-\frac3{2\alpha}-}|v^N(s)^2|_{L^1}ds \leq C \int_0^t (t-s)^{-\frac7{4\alpha}-} |v^N(s)|_{L^2}^2ds.
\end{eqnarray}
Finally, using Assumption \eqref{eq-bound-v-n-l-2}, we end up with 
\begin{eqnarray}
|y^N(t)|_{L^4}  \leq C\sup_{N}\sup_{t\in[0,T]}|v^N(s)|_{L^2}^2 T^{1-\frac7{4\alpha}}<\infty.
\end{eqnarray}
\end{proof}

\begin{lem}
Assume $ \alpha\in (\frac32, 2)$,  $ \delta\in [0, \frac{2\alpha-3}{2})$.  Let $ v^N: (0, T) \rightarrow L^4(0, 1)$ satisfying
\begin{equation}\label{eq-est-v-n-l-4}
 \sup_{N}\sup_{t\in [0, T]}|v^N(t)|_{L^4}<\infty.
\end{equation}
Then
\begin{equation}
\sup_{N}\sup_{t\in[0,T]}|y^N(t)|_{C^\delta}<\infty.
\end{equation}    
\end{lem}

\begin{proof}
\del{Let $ \beta <\alpha-\frac12-\delta$ be ....} Using Lemma \ref{Prop-sg-1}, in particular Est.\eqref{semi-gr-1}  and  Lemma \ref{Lem-Est-1-PN}, we infer that
\begin{eqnarray}
|y^N(t)|_{C^\delta} & \leq & \int_0^t|e^{-A^{\alpha/2}(t-s)}P_NF(v^N(s))|_{C^\delta}ds \nonumber\\
&\leq & \int_0^t \Vert e^{-A^{\alpha/2}(t-s)}\Vert_{\mathcal{L}(H^{-1}_2, C^\delta)}\Vert P_N \Vert_{\mathcal{L}(H^{-1}_2)}| F(v^N(s))|_{H^{-1}_2}ds\nonumber\\
 & \leq & C_{\alpha, \beta, \delta}\int_0^t (t-s)^{-\frac{3+2\delta}{2\alpha}-}| v^N(s)|^2_{L^4}ds \leq  C_{\alpha, \beta, \delta}T^{1-\frac{3+2\delta}{2\alpha}-} \sup_{N}\sup_{t\in [0, T]}| v^N(s)|^2_{L^4}<\infty.
 \end{eqnarray}

\del{\begin{eqnarray}
|y(t)|_{C^\delta} & \leq & \int_0^t|e^{-A^{\alpha/2}(t-s)}F(v(s))|_{C^\delta}ds \nonumber\\ &\leq & \int_0^t|e^{-A^{\alpha/2}\frac{(t-s)}{2}}|_{\mathcal{L}(H^{-\beta}_2, C^\delta)} |e^{-A^{\alpha/2}\frac{(t-s)}{2}}|_{\mathcal{L}(H^{-\beta}_2, H^{-1}_2)}    | F(v(s))|_{H^{-1}_2}ds\nonumber\\
 & \leq & C_{\alpha, \beta, \delta}\int_0^t (t-s)^{-\frac{1+2\delta+2\beta}{2\alpha}-}
 \end{eqnarray}}
\end{proof}

\begin{coro}\label{coro-y-N-c-delta}
Assume $ \alpha\in (\frac74, 2)$, $ \delta\in (0, \frac{2\alpha-3}{2})$ and that  $(v^N)N$ satisfyies Cond.\eqref{eq-bound-v-n-l-2}. Then 
\begin{equation}
\sup_N\sup_{t\in [0, T]}|y^N(t)|_{C^\delta} <\infty.
\end{equation} 
\end{coro}

\begin{lem}\label{lem-deter-diff-eq}
Let $ \alpha\in (\frac32, 2)$ and $\delta\in (1-\frac\alpha2, \frac{\alpha-1}{2})$. We introduce the following  initial value problems
\begin{equation}\label{Pbm-v-N}
\left\{
\begin{array}{rl}
\frac{\partial}{\partial t}v^N(t) & = -A^{\alpha/2} v^N(t)+ P_NF(v^N(t)+\xi_N(t)),\\
v^N(0)& =P_Nv_0,
\end{array}
\right.
\end{equation}
where $ v_0\in L^2(0, 1)$ and  $ \xi_N :(0, T) \rightarrow C^{\delta}(0,1)$ is continuous with 
\begin{equation}\label{cond-xi-N}
\sup_{N}\sup_{[0, T]}|\xi_N(t)|_{C^{\delta}}<\infty.
\end{equation}
Assume that $F$ is given by $f$ satisfying Cond.\eqref{cond-f-g}. Then  for all $ N\in \mathbb{N}^*$, IVP.\eqref{Pbm-v-N} admits a local solution. Moreover, if $F$ satisfies Cond.\eqref{con-F-supp}, then the local solution $ v^N$ becomes global, unique  and it \del{  unique mild solution $v^N: (0, T) \rightarrow L^2(0, 1)$,} satisfies 
\begin{equation}\label{Mohamed-1}
v^N\in C(0, T; L^2(0, 1))\cap L^2(0, T; H^\frac\alpha2_2(0, 1))
\end{equation}
 and  
\begin{equation}\label{Mohamed}
\sup_{N}\sup_{t\in [0, T]} |v^N(t)|_{L^2}<\infty.
\end{equation}
In particular, this result is true for Burgers equation.
\end{lem}
\begin{proof}
\del{The solution of Pbm.\eqref{Pbm-v-N}, is given in the integral form by
\begin{equation}
v^N(t)  = e^{-A^{\alpha/2} t}P_Nv_0+\int_0^t e^{-A^{\alpha/2} (t-s)} P_NF(v^N(s)+\xi_N(s))ds.
\end{equation}}
To prove the existence of the local solution it is sufficient to prove that there exists $T_0\leq T$, such that the application $ \varphi^N: C(0, T_0; C^{\delta}(0,1)) \rightarrow  C(0, T_0; C^{\delta}(0,1))$ is welldefined and it is a contraction, where $\varphi^N$ is given by 
\begin{equation}
(\varphi^N v)(t)  = e^{-A^{\alpha/2} t}P_Nv_0+\int_0^t e^{-A^{\alpha/2} (t-s)} P_NF(v(s)+\xi_N(s))ds.
\end{equation}
In fact, let $ u, v \in C(0, T; C^{\delta}(0,1))$  such that $ |u|_{C^\delta}, |v|_{C^\delta}<R$.  Using Corollary \ref{Coro-1}, the embedding $ C^{\delta}(0,1) \hookrightarrow H^{1-\frac\alpha2}_2(0, 1)$, see Lemma   \ref{lem-ess-general}, Cond.\eqref{cond-f-g} and the fact that $C^{\delta}(0,1)$ is a multiplication algebra, we obtain 
\begin{eqnarray}
|(\varphi^N v \!\!\!\!\!&-&\!\!\!\!\!\! \varphi^N u)(t)|_{C^\delta} \leq  
\int_0^t |e^{-A^{\alpha/2} (t-s)} P_N(F(v(s)+\xi_N(s))-F(u(s)+\xi_N(s)))|_{C^\delta}ds\nonumber\\
&\leq & 
\int_0^t \Vert e^{-A^{\alpha/2} (t-s)}\Vert_{\mathcal{L}(H^{-\frac\alpha2}_2, C^\delta)} \Vert P_N \Vert_{\mathcal{L}(H^{-\frac\alpha2}_2)}|f(v(s)+\xi_N(s)) - f(u(s)+\xi_N(s))|_{H^{1-\frac\alpha2}_2}ds\nonumber\\
&\leq & 
\int_0^t (t-s)^{-\frac{1+2\delta+\alpha}{2\alpha}-}\vert(v(s)-u(s))\left( g(v(s)+\xi_N(s), u(s)+\xi_N(s))\right) \vert_{C^\delta}ds\nonumber\\
&\leq & 
\int_0^t (t-s)^{-\frac{1+2\delta+\alpha}{2\alpha}-}\vert(v(s)-u(s))|_{C^\delta}|\left( g(v(s)+\xi_N(s), u(s)+\xi_N(s))\right) \vert_{C^\delta}ds\nonumber\\
\del{&\leq & 
\int_0^t (t-s)^{-\frac{1+2\delta+\alpha}{2\alpha}-}|f(v(s)+\xi_N(s)) - f(u(s)+\xi_N(s))|_{C^\delta}ds\nonumber\\
&\leq & 
\int_0^t (t-s)^{-\frac{1+2\delta+\alpha}{2\alpha}-}   |v(s)-u(s)|_{C^\delta}|g|(|v(s)|_{C^\delta}+|\xi_N(s))|_{C^\delta}, |u(s)|_{C^\delta}+|\xi_N(s))|_{C^\delta})ds\nonumber\\}
&\leq &  
C_R\sup_{s\in [0,1]}|v(s)-u(s)|_{C^\delta} \int_0^t (t-s)^{-\frac{1+2\delta+\alpha}{2\alpha}-}ds
 \leq  T^{1-\frac{1+2\delta+\alpha}{2\alpha}} C_R\sup_{s\in [0,1]}|v(s)-u(s)|_{C^\delta}.\nonumber\\
\end{eqnarray}
Than, we choose $T_0$, such that $C_RT_0^{1-\frac{1+2\delta+\alpha}{2\alpha}} < 1.
$
\del{\begin{equation}
C_RT_0^{1-\frac{1+2\delta+\alpha}{2\alpha}} < 1.
\end{equation} }
\noindent Now, we prove that there exists $C_T>0$, such that for any  solution of  IVP\eqref{Pbm-v-N} on $[0, T_0]$, we have 
\begin{equation}
\sup_{N}\sup_{t\in [0, T_0]} |v^N(t)|_{L^2}<\infty.
\end{equation}
This last condition is sufficient to guaranty the global existence of the solution.
In fact, we multiply the two sides of the first equation in IVP\eqref{Pbm-v-N} by $ v^N$ and we use \cite{Temam-NSE} and we integrate, we get 
\del{\begin{eqnarray}\label{Pbm-v-N-Temam}
\frac{\partial}{\partial t}|v^N(t)|_{L^2}^2  &=& -2\langle A^{\alpha/2} v^N(t), v^N(t)\rangle + 2\langle  P_NF(v^N(t)+\xi_N(t)), v^N(t) \rangle.
\end{eqnarray}
Thus, }
 \begin{eqnarray}\label{Pbm-v-N-Temam}
|v^N(t)|_{L^2}^2  &+ &  2\int_0^t |v^N(s)|^2_{H^{\frac\alpha2}_2} ds 
 =  |v^N(0)|_{L^2}^2 + 2 \int_0^t \langle  F(v^N(s)+\xi_N(s)), P_Nv^N(s)\rangle ds.\nonumber\\
 \end{eqnarray}
We use Cond.\eqref{con-F-supp}, Young inequality with $\epsilon_1, \epsilon_2>0$ and Lemma \ref{Lem-Est-1-PN}, we obtain
\begin{eqnarray}
|v^N(t)|_{L^2}^2  &+ &  2\int_0^t |v^N(s)|^2_{H^{\frac\alpha2}_2} ds \leq |P_Nv_0|_{L^2}^2 + \sum_{j=1}^m c_j \int_0^t |v^N(s)|^{\mu_j}_{H^\frac\alpha2_2}|\xi_N(s)|^j_{C^\delta}ds
   \nonumber\\
 & +& \int_0^t|v^N(s)|_{L^2} |v^N(s)|_{H^\frac\alpha2_2} (\sum_{j=1}^m c_j|\xi_N(s)|^j_{C^\delta})ds\nonumber\\
&\leq & |v_0|_{L^2}^2 + m \epsilon_1 \int_0^t |v^N(s)|^2_{H^\frac\alpha2_2}ds + \frac1{\epsilon_1}(\sum_{j=1}^m c_j\int_0^t |\xi_N(s)|^{\frac2{2-\mu_j}}_{C^\delta}ds)   \nonumber\\
 & +& \epsilon_2\int_0^t |v^N(s)|^2_{H^\frac\alpha2_2}ds + \frac1{\epsilon_2} \int_0^t |v^N(s)|_{L^2}^2 (\sum_{j=1}^m c_j|\xi_N(s)|^j_{C^\delta})^2ds.
\end{eqnarray}
We choose $\epsilon_1,  \epsilon_2>0$ such that $m\epsilon_1+ \epsilon_2<2$ and we use Cond. we end up with
\begin{eqnarray}\label{Pbm-v-N-Temam-gen}
|v^N(t)|_{L^2}^2  +  (2- m \epsilon_1- \epsilon_2)\int_0^t |v^N(s)|^2_{H^{\frac\alpha2}_2} ds &\leq &  |v_0|_{L^2}^2 +\frac1{\epsilon_1}T(\sum_{j=1}^m c_j\sup_{N}\sup_{s} |\xi_N(s)|^{\frac2{2-\mu_j}}_{C^\delta})\nonumber\\
&  +&  \frac1{\epsilon_2} sup_{N}\sup_{s}(\sum_{j=1}^m c_j|\xi_N(s)|^j_{C^\delta})^2)\int_0^t |v^N(s)|_{L^2}^2 ds
\nonumber\\
 &\leq & C_1+C_2\int_0^t |v^N(s)|_{L^2}^2 ds.
\end{eqnarray}
In particular, as the first term in the LHS of Est.\eqref{Pbm-v-N-Temam-gen} is bounded by the RHS of Est.\eqref{Pbm-v-N-Temam-gen} and \del{\begin{eqnarray}\label{Pbm-v-N-Temam-gen}
|v^N(t)|_{L^2}^2  
 &\leq & C_1+C_2\int_0^t |v^N(s)|_{L^2}^2 ds.
\end{eqnarray}}and by application of Gronwall lemma, we deduce that $|v^N(t)|_{L^2}^2  \leq  C_1e^{ C_2T}$ and consequently that $ \sup_N\int_0^t |v^N(s)|^2_{H^{\frac\alpha2}_2} ds<\infty$. Thus, conditions \eqref{Mohamed-1} \& \eqref{Mohamed}  are fulfilled.
\del{\begin{eqnarray}
|v^N(t)|_{L^2}^2  & \leq & C_1e^{ C_2T}. 
\end{eqnarray}}

\vspace{0.15cm}

To more clarify that our study covers the fractional stochastic Burgers equation, we independently develop bellow this later case. In fact, using the fact that  for all $y$, we have $\langle  \partial_x y^2, y\rangle =0$,  \cite[Lemma 11]{BrzezniakDebbi1}, the embedding $C^{\delta}(0,1)\hookrightarrow  H^{1-\frac\alpha2}_2(0, 1)$, see  Lemma \ref{lem-ess-general}, and the fact that $C^{\delta}(0,1)$ is a multiplication algebra, we get
 \begin{eqnarray}\label{Pbm-v-N-Temam}
|v^N(t)|_{L^2}^2  &+ &  2\int_0^t |v^N(s)|^2_{H^{\frac\alpha2}_2} ds 
 =  |v^N(0)|_{L^2}^2 + 2 \int_0^t \langle  F(v^N(s)+\xi_N(s)), P_Nv^N(s)\rangle ds \nonumber\\
&\leq &
|v^N(0)|_{L^2}^2 + 2 \int_0^t \left( |\langle  \partial_x (v^N(s))^2, v^N(s)\rangle| +|\langle  \partial_x (\xi_N(s))^2, v^N(s)\rangle|\right)ds \nonumber \\ 
& + & 2 \int_0^t|\langle  \partial_x (v^N(s)\xi_N(s)), v^N(s)\rangle|  ds \nonumber\\
&\leq &
|v^N(0)|_{L^2}^2 + \int_0^t |\xi_N(s))^2|_{H^{1-\frac\alpha2}_2} |v^N(s)|_{H^\frac\alpha2_2} ds+ 4 \int_0^t |v^N(s)|_{H^\frac\alpha2_2}|v^N(s)\xi_N(s))|_{H^{1-\frac\alpha2}_2}ds \nonumber\\
&\leq &
|P_Nv(0)|_{L^2}^2 + \int_0^t |\xi_N(s))|^2_{C^\delta} |v^N(s)|_{H^\frac\alpha2_2} ds+ 4 \int_0^t |v^N(s)|_{H^\frac\alpha2_2}|v^N(s)|_{H^{1-\frac\alpha2}_2} |\xi_N(s))|_{C^\delta}ds.\nonumber \\
& &
\end{eqnarray}
Using the following interpolation $ |v^N(s)|_{H^{1-\frac\alpha2}_2}\leq c |v^N(s)|^{2\frac{\alpha-1}{\alpha}}_{L^2}|v^N(s)|^{\frac{2-\alpha}{\alpha}}_{H^{\frac\alpha2}_2}$, Young inequality, Lemma \ref{Lem-Est-1-PN}  and Cond.\eqref{cond-xi-N}, \del{that $ \sup_{N}\sup{s\in [0, T]}|\xi_N(s))|^2_{C^\delta}<\infty$,} we deduce that 
 \begin{eqnarray}
|v^N(t)|_{L^2}^2  &+ &  2\int_0^t |v^N(s)|^2_{H^{\frac\alpha2}_2} ds 
\leq  |v_0|_{L^2}^2 + \int_0^t (\frac{1}{\epsilon_1}|\xi_N(s))|^4_{C^\delta} + \epsilon_1^2|v^N(s)|^2_{H^\frac\alpha2_2}) ds\nonumber\\
&+& 4c \int_0^t |v^N(s)|^\frac2\alpha_{H^\frac\alpha2_2}|v^N(s)|^{2\frac{\alpha-1}{\alpha}}_{L^2} |\xi_N(s))|_{C^\delta}ds\nonumber\\
& \leq &  |v_0|_{L^2}^2 + \int_0^t (\frac{1}{\epsilon_1}|\xi_N(s))|^4_{C^\delta} + \epsilon_1^2|v^N(s)|^2_{H^\frac\alpha2_2}) ds\nonumber\\
&+& 4c \int_0^t ({\epsilon_2}|v^N(s)|^2_{H^\frac\alpha2_2} + \frac{1}{\epsilon_2}(|v^N(s)|^{2}_{L^2} |\xi_N(s))|_{C^\delta}^{2} ds\nonumber\\
& \leq &  |v_0|_{L^2}^2  + \frac{Tc}{\epsilon_1}+(\epsilon_1^2+4c\epsilon_2)\int_0^t |v^N(s) ^2_{H^\frac\alpha2_2}) ds + \frac{4c}{\epsilon_2} \int_0^t |v^N(s)|^{2}_{L^2} ds. 
\end{eqnarray}
the choice of $ \epsilon_1 $ and $ \epsilon_2 $ such that $ \epsilon_1^2+4c\epsilon_2 \leq 2 $ gives us
\begin{eqnarray}
|v^N(t)|_{L^2}^2  +  (2- \epsilon_1^2-4c\epsilon_2) \int_0^t |v^N(s)|^2_{H^{\frac\alpha2}_2} ds & \leq & (|v_0|_{L^2}^2 + \frac{Tc}{\epsilon_1})+ \frac{4c}{\epsilon_2} \int_0^t |v^N(s)|^{2}_{L^2} ds.\nonumber \\
& &
\end{eqnarray}
Now, arguing as above and we apply Gronwall lemma, we infer that 
\begin{eqnarray}
|v^N(t)|_{L^2}^2  & \leq & (|v_0|_{L^2}^2 + \frac{Tc}{\epsilon_1})e^{ \frac{4c}{\epsilon_2}T}. 
\end{eqnarray}
Thus, Cond.\eqref{Mohamed} is fulfilled and consequently, $ v^N\in L^2(0, T; H^\frac\alpha2_2(0, 1))$. Thus the proof is achieved.
\end{proof}

%%%%%%%%%%%%%%%%%%%%%%%%%%%%%%%%%%%%%%%%%%%%%%%%%%%%%%%%%%%%%%%%%%%%%%%%%%%%%%%%%%

\section{Proof of Theorems}\label{sec-Proofs}

\subsection{Proof of Theorem \ref{Theorem-disc}}
\vspace{0.35cm}

{ \bf Existence.} We understand Equation \eqref{Discr-Galerkin-FSBE-Evol-1} in the integral form as 
\begin{equation}\label{Discr-Galerkin-FSBE-Evol-1-Integ}
u_N(t)= e^{- tA^{\alpha/2}}P_Nu_0 + \int_0^te^{-A^{\alpha/2} (t-s)}P_NF(u_N(s))ds+ \mathcal{W}_N(t), \;\; t\in [0, T],\\
\end{equation}
where $\mathcal{W}_N(t)$ is given by \eqref{OU-process-W-N}.\del{$ \mathcal{W}_N(t):=\int_0^te^{-A^{\alpha/2} (t-s)}W_N(dt)$.}
Remark that if $ u_N$ is solution of Eq.\eqref{Discr-Galerkin-FSBE-Evol-1}, then $ v_N:= u_N-\mathcal{W}_N$ is a mild solution of the pathwise IVP.\eqref{Pbm-v-N}, with $\xi_N$ by $\mathcal{W}_N$ and vice versa.  To prove the existence of the solution $ u_N$ satisfying  Eq.\eqref{Discr-Galerkin-FSBE-Evol-1-Integ} and Cond.\eqref{Mohamed}, we apply Lemma \ref{lem-Soundous} and 
Lemma \ref{lem-deter-diff-eq}. 
\del{We solve Eq.  \eqref{Discr-Galerkin-FSBE-Evol-1} in pathwise form. It is quiet standard to prove that Eq.  \eqref{Discr-Galerkin-FSBE-Evol-1} admits a local solution, see e.g. \cite{BrzrezniakDebbi1, Daprato,  Debbi-scalar-active1, }. To prove that the local solution is global, we prove that it is globally bounded, i.e.

It is easy to prove that Eq. \eqref{Discr-Galerkin-FSBE-Evol-1} admits a local solution $ u_N^n:= (u_N^n(t), t\in [0, T])$, see e.g.  To prove that this solution is global we have to prove that $ \sup_n$

\vspace{0.35cm}

 The existence and the uniqueness of the solution is guaranted thanks to Theorem \cite{Debbi-scalar-active1}.} 

\vspace{0.35cm}

{ \bf Uniform boundedness of $(u_N)_N$.} Now, we assume that $ \frac74<\alpha< 2$, $\delta\in (1-\frac{\alpha}{2}, \frac{2\alpha-3}{2})$ and we prove that the solutions $(u_N)_N$ satisfy Est.\eqref{Eq-Assum-4-1-u-N}. In fact, uisng Identity \eqref{Discr-Galerkin-FSBE-Evol-1-Integ}, it is obvious that  
\begin{equation}\label{eq-est-saparate-u-N}
|u_{N}(t)|_{C^\delta} \leq |e^{-A^{\alpha/2}t}P_{N}u_{0}|_{C^\delta} + |\int_{0}^{t} e^{-A^{\alpha/2}(t-s)}P_{N}F(u_{N}(s))ds|_{C^\delta}+|\mathcal{W}_{N}(t)|_{C^\delta}.
\end{equation} 
Let us remark that the second term respectively the third one in Est.\eqref{eq-est-saparate-u-N}, are bounded thanks to Corollary \ref{coro-y-N-c-delta} and Lemma \ref{lem-deter-diff-eq} respectively to Lemma \ref{lem-Soundous}. To estimate the first term in Est.\eqref{eq-est-saparate-u-N}, we use Lemma \ref{Prop-sg-1}, Lemma \ref{Lem-Est-1-PN} and Assumption $\mathcal{A}$, we get
\begin{equation}
|e^{-A^{\alpha/2}t}P_{N}u_{0}|_{C^\delta} \leq \Vert e^{-A^{\alpha/2}t} \Vert_{\mathcal{L}(H^{\beta}_{2},C^\delta) }\Vert P_{N} \Vert_{\mathcal{L}(H^{\beta}_{2})}|u_{0}|_{H^{\beta}_{2}} \leq C_{\alpha, \delta, \beta}|u_{0}|_{H^{\beta}_{2}}<\infty.
\end{equation}

The proof is then achieved.
\vspace{0.35cm}

{\bf H\"older Regularity of $u_N$.} We prove that each term of 
Identity \eqref{Discr-Galerkin-FSBE-Evol-1-Integ} is H\"older continuous of index $ \frac{\alpha-1-2\delta}{2\alpha}$. In fact, the regularity of $\mathcal{W}_N$ follows from Lemma \ref{auxiliary-3}. To get the regularity of the first term, we use Corollary \ref{coro-sg-regul}, Lemma \ref{Lem-Est-1-PN}, the embedding $ H^{\beta}_2(0, 1) \hookrightarrow H^{-\frac\alpha2}_2(0, 1)$ and the Assumption $\mathcal{A}$. Then,  for $ \tau< t\in (0, T)$, we have
 \begin{eqnarray}\label{eq-Reg-1-term}
 | (e^{-A^{\alpha/2} t}- e^{- A^{\alpha/2} \tau})P_N u_0|_{C^\delta}&\leq & C\Vert e^{-A^{\alpha/2} t}- e^{-A^{\alpha/2} \tau} \Vert_{\mathcal{L}(H^{-\frac\alpha2}_2, C^\delta)} \Vert  P_N \Vert_{\mathcal{L}(H^{-\frac{\alpha}{2}}_{2})} |u_0|_{H^{-\frac\alpha2}_2} \nonumber\\
 &\leq & C_{\alpha, \delta, \tau} |t-\tau|^{(\frac{\alpha-1-2\delta}{2\alpha})-}| u_0|_{H^{\beta}_2}.
 \end{eqnarray}
For the second term, we have,
\begin{eqnarray}
|\int_0^t &{}&\!\!\!\!\! \!\!\!\! \!\!\!\!\!e^{-A^{\alpha/2} (t-s)}P_NF(u_N(s))ds - \int_0^\tau e^{-A^{\alpha/2} (\tau -s)}P_NF(u_N(s))ds|_{C^\delta}\nonumber\\
&\leq & |\int_0^\tau [e^{-A^{\alpha/2} (t-s)}- e^{-A^{\alpha/2} (\tau-s)}]P_NF(u_N(s))ds|_{C^\delta}+  |\int_\tau^t e^{-A^{\alpha/2} (t-s)}P_NF(u_N(s))ds|_{C^\delta}\nonumber\\
&\leq & \int_0^\tau \Vert e^{-A^{\alpha/2} (t-s)}- e^{-A^{\alpha/2} (\tau-s)}\Vert_{\mathcal{L}(H^{-\frac\alpha2}_2, C^\delta)}\Vert P_N \Vert_{\mathcal{L}(H^{-\frac\alpha2}_2)}|F(u_N(s))|_{H^{-\frac\alpha2}_2}ds\nonumber\\
&+&  \int_\tau^t \Vert e^{-A^{\alpha/2} (t-s)}\Vert_{\mathcal{L}(H^{-\frac\alpha2}_2, C^\delta)} \Vert P_N \Vert_{\mathcal{L}(H^{-\frac\alpha2}_2)}|F(u_N(s))|_{H^{-\frac\alpha2}_2}ds.\nonumber\\
\end{eqnarray}
Using Lemma \ref{lem-sg-regul} and Lemma \ref{Lem-Est-1-PN}, we infer that, for $ \epsilon_1, \epsilon_2 \in (0, \frac{\alpha-1-2\delta}{2\alpha})$, 
\begin{eqnarray}
|\int_0^t &{}&\!\!\!\!\! \!\!\!\! \!\!\!\!\!e^{-A^{\alpha/2} (t-s)}P_NF(u_N(s))ds - \int_0^\tau e^{-A^{\alpha/2} (\tau -s)}P_NF(u_N(s))ds|_{C^\delta}\nonumber\\
&\leq & C_{\alpha}\int_0^\tau (t-\tau)^{\frac{\alpha-1-2\delta}{2\alpha}-\epsilon_1} (\tau-s)^{-1+\epsilon_2}|F(u_N(s))|_{H^{-\frac\alpha2}_2} ds\nonumber\\
&+&  C_{\alpha}\int_\tau^t (t-s)^{-1+\epsilon_2}|F(u_N(s))|_{H^{-\frac\alpha2}_2} ds.
\end{eqnarray}
Let us now, remark that thanks to the uniform boundedness of $(u_N)_N$ with respect to $t$ and $N$, we can choose $R(\omega)= \sup_N\sup_{t\in[0,T]}|u_N(t, \omega)|_{C^\delta}$, for $a.s. \omega \in \Omega$ and by Application of Corollary \ref{Coro-linear-growth}, we infer the existence of a random variable $C_{F, \alpha, \delta}(\omega)$, such that 
\begin{equation}\label{cond-lin-growth-univ}
|F(u_N(s, \omega))|_{H^{-\frac\alpha2}_2}\leq C_{F, \alpha, \delta}(\omega)(1+|u_N(s, \omega)|_{C^\delta})\leq C_{F, \alpha, \delta}(\omega).
\end{equation}
Thus
\begin{eqnarray}
|\int_0^t &{}&\!\!\!\!\! \!\!\!\! \!\!\!\!\!e^{-A^{\alpha/2} (t-s)}P_NF(u_N(s))ds - \int_0^\tau e^{-A^{\alpha/2} (\tau -s)}P_NF(u_N(s))ds|_{C^\delta}\nonumber\\
&\leq & C_{F, \alpha, \delta}(\cdot)\big[\int_0^\tau (t-\tau)^{\frac{\alpha-1-2\delta}{2\alpha}-\epsilon_1} (\tau-s)^{-1+\epsilon_2} ds +\int_\tau^t (t-s)^{-1+\epsilon_2}ds\big]\nonumber\\
&\leq & C_{F, \alpha, \delta, T} \big[ (t-\tau)^{\frac{\alpha-1-2\delta}{2\alpha}-\epsilon_1} + (t-\tau)^{\epsilon_2} \big].
\end{eqnarray}
Now, it is easy to get the H\"older index, by taking $\epsilon_1 \rightarrow 0$ and  $\epsilon_2 \rightarrow \frac{\alpha-1-2\delta}{2\alpha}$.

\vspace{0.35cm}

{\bf Uniqueness.} Assume that there exist two solution $u_N^1$  and $u_N^2$ two solutions of Eq.\eqref{Discr-Galerkin-FSBE-Evol-1-Integ} starting from the same initial condition $u_0$ and satisfying the boundedness, the regularity peroperties above, then using Corollary \ref{Coro-1}\del{Lemma \ref{Prop-sg-1}}, the boundedness property of $u^1_N$ and $u^2_N$ and Lemma \ref{main-nonlinear}, we obtain $P-a.s.$, for all $t\in (0, T)$
\begin{eqnarray}
|u^1_N(t, \omega)-u^2_N(t, \omega)|_{C^\delta}&\leq &  \int_0^t|e^{-A^{\alpha/2} (t-s)}P_N(F(u^1_N(s, \omega))- F(u^2_N(s, \omega)))|_{C^\delta}ds\nonumber\\
&\leq & \int_0^t\Vert  e^{-A^{\alpha/2} (t-s)}\Vert_{\mathcal{L}(H^{-\frac\alpha2}_2, C^\delta)} \Vert P_N \Vert_{\mathcal{L}(H^{-\frac\alpha2}_2)}|F(u^1_N(s, \omega))- F(u^2_N(s, \omega))|_{H^{-\frac\alpha2}_2}ds\nonumber\\
&\leq &  C_{F, \alpha, \delta}(\omega)\int_0^t(t-s)^{-(\frac{1+2\delta+\alpha}{2\alpha})-}|u^1_N(s, \omega)- u^2_N(s, \omega)|_{C^\delta}ds.
\end{eqnarray}
By application of Gronwall lemma we get, $P-a.s.$,  $\forall t\in (0, T)$, $|u^1_N(t, \omega)-u^2_N(t, \omega)|_{C^\delta}=0$. Thus the uniqueness is proved.

%%%%%%%%%%%%%%%%%%%%%%%%%%%%%%%%%%%%%%%%%%%%%%%%%%%%%%%%%%%

\subsection{Proof of Theorems \ref{Main-Theorem-Exit-Uniq} \& \ref{main-result-Galerkin-approx}}
To prove theorems \ref{Main-Theorem-Exit-Uniq} \& \ref{main-result-Galerkin-approx}, we will mainly check that  assumptions 1-4 of Theorem \ref{Theorem-Blomker-Jentzen} hold, with $ V:=C^{\delta}(0,1)$, $ U:= H^{-\alpha/2}_2(0, 1)$ and $ S(t)=e^{-A^{\alpha/2} t}$. 

\vspace{0.25cm}
{\bf Assumption 1.}  Lemma \ref{lem-sg-regul} states that for $\alpha\in (1, 2]$ and $\delta\in [0, 1)$ the semigroup $e^{-A^{\alpha/2} \cdot}: [0, T]\rightarrow \mathcal{L}(H^{-\frac\alpha2}_2(0, 1), C^{\delta}(0,1))$ is H\"olderian so it is continuous. Moreover, for $\delta<\frac{\alpha-1}{2}$ and thanks to Corollary  \ref{Coro-1}, we have for all $ \eta' \in ( \frac{1+2\delta+\alpha}{2\alpha}, 1)$, 
\begin{equation}\label{Eq-Assum-1-1-app}
\sup_{t\in (0, T]}\big( t^{\eta'} \Vert e^{-A^{\alpha/2} t} \Vert_{\mathcal{L}(H^{-\alpha/2}_2, C^\delta)}\big)<\infty. 
\end{equation}
\del{To prove Cond.\eqref{Eq-Assum-1-2},}Now, we introduce the auxaliary parameter $ \beta \in (\frac{\alpha}{2}, \alpha-\delta-\frac{1}{2})$. Thanks to Lemma \ref{Prop-sg-1} and Lemma \ref{Lem-Est-1-PN}, we infer that for all $\eta''\in ( \frac{1+2\delta+2\beta}{2\alpha}, 1)$, there exists
 $C_{\delta, \beta, \eta''}>0$, s.t. 
\begin{eqnarray}\label{Eq-Assum-1-2-app-1}
\Vert (1-P_N)e^{-A^{\alpha/2} t}\Vert_{\mathcal{L}(H^{-\frac{\alpha}{2}}_2, C^\delta)} &=& \Vert e^{-A^{\alpha/2} t}(1-P_N)\Vert_{\mathcal{L}(H^{-\frac{\alpha}{2}}_2, C^\delta)}\nonumber\\
&\leq &
 \Vert e^{-A^{\alpha/2} t}\Vert_{\mathcal{L}(H^{-\beta}_2, C^\delta)} \Vert 1-P_N \Vert_{\mathcal{L}(H^{-\frac{\alpha}{2}}_2, H^{-\beta}_2)}\nonumber\\
& \leq & C_{\delta, \beta, \eta''} t^{-\eta''}N^{-(\beta-\frac{\alpha}{2})}. 
\end{eqnarray}
We consider  $ \beta=(\alpha-\delta-\frac12)-$ and we take $ \eta'=\eta'':= 1-\epsilon$, with $ \epsilon \in (0, \frac{\alpha-1-2\delta}{2\alpha})$, then we get the estimate
\begin{eqnarray}\label{Eq-Assum-1-2-app}
\Vert(1-P_N)e^{-A^{\alpha/2} t}\Vert_{\mathcal{L}(H^{-\frac{\alpha}{2}}_2, C^\delta)} & \leq & C_{\delta, \beta, \eta} t^{-1+\epsilon}N^{-(\frac{\alpha-1}{2}-\delta)+}. 
\end{eqnarray}
Consequently \del{for $ \eta\in (\max\{\eta', \eta''\}, 1)$, or in other word,   $\eta\in ( \max\{\frac{1+2\delta+2\beta}{2\alpha}, \frac{1+2\delta+\alpha}{2\alpha}\}, 1)$}
Assumption 1 is satisfied.

\vspace{0.25cm}
{\bf Assumption 2.} This assumption is satisfied,  for the fractional stochastic Burgers type equations,  thanks to Corollary \ref{Coro-nolinearterm}, provided that   $1< \alpha\leq 2$ and $\delta>1-\frac{\alpha}{2}$.

Consequently, assumptions 1 \& 2 are simultaneously satisfied for  $3/2< \alpha\leq 2$ and $\delta \in ( 1-\frac{\alpha}{2}, \frac{\alpha-1}{2})$, see also Corollary \ref{Coro-nolinearterm+1}.

\vspace{0.25cm}

{\bf Assumption 3.}  Corollary \ref{Coro-Reg-W} shows the continuity of the process $\mathcal{W}$. Moreover, it is easy to see tha Est.\eqref{eq-Reg-1-term} is still valid without $P_N$. Thus, the process  $ e^{-A^{\alpha/2}t}u_{0}: \Omega \rightarrow C^{\delta}(0,1)$  and consequently    the process $O(t):= e^{-tA^{\alpha/2}}u_0+ \mathcal{W}(t): \Omega \rightarrow C^{\delta}(0,1)$ are continuous. In addition, thanks to  Lemma \ref{Lem-Est-1-PN}, the fact that $\alpha<2+\delta$, Assumption $\mathcal{A}$  and the embeddings $ H^{\beta}_{2}(0, 1) \hookrightarrow H^{\delta +\frac{1}{2}}(0, 1)  \hookrightarrow  H^{\alpha/2}_{2}(0, 1)$, we have $P-a.s.$
\begin{equation}\label{u-oassump-3}
|(1-P_N)e^{-A^{\alpha/2} t}u_0(\omega)|_{C^{\delta})}\leq 
\Vert(1-P_N)e^{-A^{\alpha/2} t}\Vert_{\mathcal{L}(H^{\alpha/2}_2, C^{\delta})}|u_0(\omega)|_{H^{\alpha/2}_2}\leq 
C_{\alpha, \delta}(\omega)N^{-(\frac{\alpha-1}{2}-\delta)+}|u_0(\omega)|_{H^{\beta}_2}.
\end{equation}
Now,  Corollary \ref{Coro-lem-2} and Est.\eqref{u-oassump-3} together  show that $O(t)$ satisfies $P-a.s.$ 
 \begin{equation}\label{Eq-Assum-3-O-final}
\sup_{N\in \mathbb{N}}\sup_{t\in (0, T]}\big(N^{(\frac{\alpha-1}{2}-\delta)-}|(1-P_N)O(t, \omega)|_{C^\delta}\big)<\infty.\del{ \;\; \text{for every}\; \omega.} 
\end{equation}
Thus Assumption3 is fulfilled.

 \del{we prove that for $ \alpha \in (3/2, 2]$, $\delta \in ( 1-\frac{\alpha}{2}, \frac{\alpha-1}{2})$ and  $u_0\in H^{-(\delta+\frac12)}_2$, we have almost surely,
\begin{equation}\label{Eq-Assum-3-proo-both}
\sup_{N\in \mathbb{N}}\sup_{t\in (0, T]}\big(N^{(\frac{\alpha-1}{2} -\delta)-}|(1-P_N)(e^{-tA^{\alpha/2}}u_0(\omega)+ \mathcal{W}(\omega, t)|_{C^\delta}\big)<\infty. 
\end{equation}
In the first step, we prove, for $ \alpha \in (3/2, 2]$, $\delta \in ( 1-\frac{\alpha}{2}, \frac{\alpha-1}{2})$ and  $p_0>\frac{2(\alpha+1)}{\alpha-1-2\delta}$ fixed, that almost surely,
\begin{equation}\label{Eq-Assum-3-proo}
\sup_{N\in \mathbb{N}}\sup_{t\in (0, T]}\big(N^{[(\alpha-1)(\frac1{2}-\frac1{p_0}) -\delta]-} \del{t^\alpha}|(1-P_N) \mathcal{W}(\omega, t)|_{C^\delta}\big)<\infty. 
\end{equation}
In fact, using Lemma \ref{lem-est-stoch-main} and Lemma  \ref{appendix-A-4}, we deduce, for $ \alpha \in (1, 2]$, $0<\beta<\frac{\alpha-1}2$, $q\geq 2$ and  $p_0>\frac{2\alpha}{\alpha-1-2\beta}$ be fixed, that almost surely,
\begin{equation}\label{Eq-Assum-3-proo-beta}
\sup_{N\in \mathbb{N}}\sup_{t\in (0, T]}\big(N^{(\frac{\alpha-1}{2}-(\beta+\frac \alpha {p_0}))-} \del{t^\alpha}|(1-P_N)\mathcal{W}(\omega, t)|_{H^\beta_q}\big)<\infty. 
\end{equation}
Moreover, it is easy to see that Est.\eqref{Eq-Assum-3-proo-beta}, is also valid for the special case; $\delta < \frac{\alpha-1}{2}$, $ p_0>\frac{2(\alpha+1)}{\alpha-1-2\delta}$, $ \beta = \delta+\frac 1{p_0}$ and $q=p_0$. Thanks to the Imbedding $ H^{\delta+\frac1{p_0}}_{p_0} \hookrightarrow C^\delta$, we get 
Est. \eqref{Eq-Assum-3-proo}. In addition, as $ \delta>1-\frac\alpha2$, we have to choose $p_0$ such that $ 1-\frac\alpha2<\frac{\alpha-1}{2}-\frac1{p_0}$, which means  that $ \alpha>3/2$ and $ p_0>\max\{\frac{2(\alpha+1)}{\alpha-1-2\delta}, \frac1{\alpha-3/2}\}$. Now, for $p_0$ so large, the power of $N$ in \eqref{Eq-Assum-3-proo-beta} will be replaced by $ (\frac{\alpha-1}{2}-\delta)-$.

In the second step, we prove for $ \alpha \in (1, 2]$, $0\leq \delta< \frac{\alpha-1}{2})$ and $ u_0: \Omega \rightarrow H_2^{-\delta-\frac12}$, that there exists $ C_{\alpha, \delta}>0$, such that  
\begin{equation}\label{Eq-Assum-3-proo-u-0}
\sup_{N\in \mathbb{N}}\sup_{t\in (0, T]}\big(N^{(\frac{\alpha-1}{2} -\delta)-} \del{t^\alpha}|(1-P_N)e^{-tA^{\alpha/2}}u_0|_{C^\delta}\big)<\infty. 
\end{equation}
By application of lemmas \ref{Prop-sg-1} \& \ref{Lem-Est-1-PN}, we infer the existence of $ C_{\alpha, \delta}>0$, such that,  for almost all $ \omega$,
\begin{eqnarray}
|(1-P_N)e^{-tA^{\alpha/2}}u_0|_{C^\delta} &=& |e^{-tA^{\alpha/2}}(1-P_N)u_0|_{C^\delta} \nonumber\\
&\leq &  |e^{-tA^{\alpha/2}}|_{\mathcal{L}(H^{-\alpha/2}_2, C^\delta)} |1-P_N|_{\mathcal{L}(H^{-\alpha/2}_2, H^{(-\delta-\frac12)-}_2)} |u_0|_{ H^{(-\delta-\frac12)-}_2}\nonumber\\
&\leq & C_{\alpha, \delta} N^{(\frac{\alpha-1}{2}-\delta)-}|u_0|_{ H^{(-\delta-\frac12)-}_2}.
\end{eqnarray}}

%%%%%%%%%%%%%%%%%%%%%%%%%%%%%%%%%%%%%%%%%%%%%%%%

\vspace{0.35cm}

{\bf Assumption 4.} This assumption is satisfied thanks to Theorem \ref{Theorem-disc}.

\subsection{Proof of Lemma \ref{u-N-M-bounded}}\label{subsec-proof-lem}

Using Eq.\eqref{eq-u-N-M-m}, we rewrite $u_{N, M}^m$ as 
\begin{equation}\label{eq-u-N-M-2}
\left\{
\begin{array}{rl}
u_{N, M}^{0}& := P_N u_0,\\ 
u_{N,M}^{m}& = e^{-A^{\alpha/2}t_{m}} u_{N,M}^0 +  \int_0^{t_m} e^{-A^{\alpha/2}(t_{m}-s)}  \Delta t \sum_{k=0}^{m-1} \delta_{t_k}(s)P_N F(u_{N,M}^{k})ds + \mathcal{W}_N(t_m).    
\end{array}
\right.
\end{equation}
We introduce the following equation:

\begin{equation}\label{eq-Z-N-M}
\left\{
\begin{array}{rl}
Z_{N, M}(0)& := P_N u_0,\\
for \;\; & t\in (t_{m-1}, t_m]:\\
Z_{N,M}(t)& = e^{-A^{\alpha/2}t}  P_N u_0 +  \int_0^t e^{-A^{\alpha/2}(t-s)} \Delta t \sum_{k=0}^{m-1}\delta_{t_k}(s)P_N F(Z_{N,M}(s) + \mathcal{W}_N(t_m))ds.   
\end{array}
\right.
\end{equation}
We argue as in  the proof of Lemma \ref{lem-deter-diff-eq} using Lemma \ref{lem-Soundous}, we prove the existence of a stochastic process $ Z_{N, M}$ solution of Eq.\eqref{eq-Z-N-M} and satisfying $ \sup_{N, M}\sup_{t\in [0, T]}|y_{N, M}(t, \omega)|_{C^{\delta}}\leq C_{\alpha, \delta}(\omega)$. 
It is easy to see that $ Z_{N, M}(t_m) =  u_{N, M}^m-\mathcal{W}_N(t_m)$, where $u_{N, M}^m$ is solution of Eq.\eqref{eq-u-N-M-2}. Now, argue as in the proof of  Theorem \ref{Theorem-disc}, we infer that $u_{N, M}^m$ exists and fulfill Est.\eqref{eq-u-m-N-M-bounded}.\del{\begin{equation}
\sup_{m, N, M}|u^m_{N, M}(\omega)|\leq C_{\alpha, \delta}(\omega)
\end{equation}}

%%%%%%%%%%%%%%%%%%%%%%%%%%%delete%%%%%%%%%%%%%%%%%%%%%%%%%

\del{\begin{proof}
 We rewrite $u_{N, M}^m$ as 
\begin{equation}
\left\{
\begin{array}{rl}
u_{N, M}^{0}& := P_N u_0,\\ 
u_{N,M}^{m}& := e^{-A^{\alpha/2}t_{m}} u_{N,M}^0 + \Delta t \sum_{k=0}^{m-1}e^{-A^{\alpha/2}(t_{m}-t_{k})} P_N F(u_{N,M}^{k}) + \mathcal{W}_N(t_m).    
\end{array}
\right.
\end{equation}
We use a reccurent proof. By definition we have $ u_{N, M}^0:= P_Nu_0$, using Lemma \ref{Lem-Est-1-PN}, in particular \eqref{est-P_n-H-eta-C-delta} and the assumption that  $ u_0 \in H^\eta_2$, we  deduce that $ |u_{N, M}^0|_{C^\delta}\leq C_{\delta, \eta}$. As the proof for $ u_{N, M}^1$ is similar to the general case, we omit. So now, we assume that there exists $ C(\omega)$, such that for almost all $\omega$, we have $ |u_{N,M}^{k}(\omega)|_{C^\delta} \leq C(\omega) $ for all $ k=1,...,m-1 $. Then,
\begin{equation}\label{est-u-N-M-m-G-1}
\vert u_{N,M}^{m}\vert_{C^\delta} \leq  \vert e^{-A^{\alpha/2}t_{m}} P_N u_0 \vert_{C^\delta}+ \vert \Delta t \sum_{k=0}^{m-1}e^{-A^{\alpha/2}(t_{m}-t_{k})} P_N F(u_{N,M}^{k})\vert_{C^\delta} + \vert \mathcal{W}_N(t_m)\vert_{C^\delta}.
\end{equation}
Thanks to Lemma \ref{Prop-sg-1}, in particular, Est.\eqref{aide-2-1-Bis} and Lemma \ref{Lem-Est-1-PN}, in particular Est.\eqref{est-P_n-H-beta}, we infer that there exists a constant $ C_{\delta, \eta}$, such that  
\begin{equation}\label{est-1-u-N-M-m}
|e^{-A^{\alpha/2}t_{m}} P_N u_0|_{C^\delta} \leq |e^{-A^{\alpha/2}t_{m}}|_{\mathcal{L}(H^\eta_2, C^\delta)}|P_N|_{\mathcal{L}(H^\eta_2)} |u_0|_{H^\eta_2} \leq C_{\eta, \delta}|u_0|_{H^\eta_2}
\end{equation}
To estimate the second term in the RHS of Eq.\eqref{est-u-N-M-m-G-1}, we use Lemma \ref{Prop-sg-1}, in particular, Est.\eqref{aide-2-1} and Corollary \ref{Coro-linear-growth}, we get
\begin{eqnarray}
| \Delta t \sum_{k=0}^{m-1}e^{-A^{\alpha/2}(t_{m}-t_{k})}\!\!\!\!\!\!\!\!\!\! &{}&\!\!\!\! P_N F(u_{N,M}^{k})|_{C^\delta} \nonumber\\
& \leq &  \Delta t \sum_{k=0}^{m-1}|e^{-A^{\alpha/2}(t_{m}-t_{k})} P_N F(u_{N,M}^{k})|_{C^\delta} \nonumber\\ 
& \leq & \Delta t \sum_{k=0}^{m-1}|e^{-A^{\alpha/2}(t_{m}-t_{k})}|_{\mathcal{L}(H^{-\frac\alpha2}_2, C^\delta)}|P_N|_{\mathcal{L}(H^{-\frac\alpha2}_2)} |F(u_{N,M}^{k})|_{H^{-\frac\alpha2}_2} \nonumber\\
& \leq & C_{\alpha, \delta, R}\Delta t \sum_{k=0}^{m-1}(t_{m}-t_{k})^{-\frac{1}{2\alpha}(\alpha+2\delta+1)} (1+|u_{N,M}^{k}|_{C^\delta}).
\end{eqnarray}
using the representation of $\Delta t$ as $\int_{t_k}^{t_{k+1}} ds$, the fact that for $s \in [t_k, t_{k+1}]$, we have $ (t_m-t_k)\geq (t_m -s)$, we obtain
\begin{eqnarray}
| \Delta t \sum_{k=0}^{m-1}e^{-A^{\alpha/2}(t_{m}-t_{k})}\!\!\!\!\!\!\!\!\!\! &{}&\!\!\!\! P_N F(u_{N,M}^{k})|_{C^\delta} \nonumber\\
& \leq & C_{\alpha, \delta, R} \sum_{k=0}^{m-1}\int_{t_k}^{t_{k+1}}(t_{m}-s)^{-\frac{1}{2\alpha}(\alpha+2\delta+1)}ds (1+|u_{N,M}^{k}|_{C^\delta})\nonumber\\
& \leq & C_{\alpha, \delta, R}\big[\int_0^{t_m}(t_{m}-s)^{-\frac{1}{2\alpha}(\alpha+2\delta+1)}ds+  \sum_{k=0}^{m-1}\int_{t_k}^{t_{k+1}}(t_{m}-s)^{-\frac{1}{2\alpha}(\alpha+2\delta+1)}ds |u_{N,M}^{k}|_{C^\delta}\big].
\end{eqnarray}
As $ \delta<\frac{\alpha-1}{2}$, the integral in the RHS of the estimate above converge and we denote $ \varphi^{\alpha, \delta}_{k, m}:= C_{\alpha, \delta, R}\int_{t_k}^{t_{k+1}}(t_{m}-s)^{-\frac{1}{2\alpha}(\alpha+2\delta+1)}ds$. Moreover,
\begin{equation}\label{est-sum-phi-k-m}
 \sum_{k=0}^{m-1}\varphi^{\alpha, \delta}_{k, m} = C_{\alpha, \delta, R} \int_0^{t_m}(t_{m}-s)^{-\frac{1}{2\alpha}(\alpha+2\delta+1)}ds= \frac{2\alpha C_{\alpha, \delta, R}}{\alpha-2\delta-1}T^{1-\frac{1}{2\alpha}(\alpha+2\delta+1)}<\infty.
\end{equation}

\del{$ \varphi^{\alpha, \delta}_{k, m}:= C_{\alpha, \delta, R}\int_{t_k}^{t_{k+1}}(t_{m}-s)^{-\frac{1}{2\alpha}(\alpha+2\delta+1)}ds = \frac{2C_{\alpha, \delta, R}\alpha}{\alpha-2\delta-1}[(t_m-t_k)^{1-\frac{1}{2\alpha}(\alpha+2\delta+1)} - (t_m-t_{k+1})^{1-\frac{1}{2\alpha}(\alpha+2\delta+1)}]$. Moreover, $ \sum_{k=0}^{m-1}\varphi^{\alpha, \delta}_{k, m}<\infty$}

Thus
\begin{eqnarray}\label{est-2-u-N-M-m}
| \Delta t \sum_{k=0}^{m-1}e^{-A^{\alpha/2}(t_{m}-t_{k})} P_N F(u_{N,M}^{k})|_{C^\delta}
& \leq & C_{\alpha, \delta, R}+ \sum_{k=0}^{m-1}\varphi^{\alpha, \delta}_{k, m} |u_{N,M}^{k}|_{C^\delta}.
\end{eqnarray}
Now, we gathered estimates; \eqref{est-u-N-M-m-G-1}, \eqref{est-1-u-N-M-m}, \eqref{est-2-u-N-M-m} and Lemma \ref{lem-Soundous}, we end up with the following estimate:
\begin{eqnarray}\label{est-u-N-M-m-G-2}
|u_{N,M}^{m}|_{C^\delta} &\leq & C_{\alpha, \delta, R, |u_0|_{H^\eta_2}}(\omega)+ \sum_{k=0}^{m-1}\varphi^{\alpha, \delta}_{k, m} |u_{N,M}^{k}|_{C^\delta}.
\end{eqnarray}
We apply the discrete Gronwal lemma, see Lemma \ref{Gronwall-Lemma} and Est. \eqref{est-sum-phi-k-m}, we get 
\begin{eqnarray}\label{est-u-N-M-m-G-3}
|u_{N,M}^{m}|_{C^\delta} &\leq & C_{\alpha, \delta, R, T, |u_0|_{H^\eta_2}}(\omega)e^{\sum_{k=0}^{m-1}\varphi^{\alpha, \delta}_{k, m}}<\infty.
\end{eqnarray}
\end{proof}}

%%%%%%%%%%%%%%%%%%%%%%%%%%%%%%%delete%%%%%%%%%%%%%%%%%%%%%%%%%%

\subsection{Proof of Theorem \ref{main-result-fully}}\label{sec-proof-full-discrt}
\del{To prove Theorem \ref{main-result-fully},}

Using the triangular inequality and Theorem \ref{main-result-Galerkin-approx},  in particular Est.\eqref{main-est-Galerkin-approx}, we get,
\begin{eqnarray}
|u(t_m)- u_{N, M}^m|_{C^\delta} &\leq &  |u(t_m)- u_{N}(t_m)|_{C^\delta} + |u_N(t_m)- u_{N, M}^m|_{C^\delta}\nonumber\\ &\leq & C N^{-(\frac{\alpha -1}{2}-\delta)+}+ |u_N(t_m)- u_{N, M}^m|_{C^\delta}.
\end{eqnarray}
In the aim to estimate the term $|u_N(t_m)- u_{N, M}^m|_{C^\delta}$, we rewrite  $u_{N, M}^m$  as:
\begin{equation}
u_{N,M}^m := e^{-A^{\alpha/2}t_m}u_N^0 + \sum_{k=0}^{m-1}\int_{t_k}^{t_{k+1}} e^{-A^{\alpha/2}(t_m-t_k)}P_N F(u_{N, M}^{k}) ds + \mathcal{W}_N(t_m).    
\end{equation}
Then,  
\begin{eqnarray}\label{second-fully}
|u_N(t_m)- u_{N, M}^m|_{C^\delta} &\leq & 
\sum_{k=0}^{m-1}|\int_{t_k}^{t_{k+1}} [e^{-A^{\alpha/2}(t_m-s)}P_N F(u_{N}(s)) - e^{-A^{\alpha/2}(t_m-t_k)}P_N F(u_{N, M}^{k})]ds|_{C^\delta}\nonumber\\
&\leq & J_1+J_2+J_3,
\end{eqnarray}
where
\begin{eqnarray}
 J_1 &:= & \sum_{k=0}^{m-1}\int_{t_k}^{t_{k+1}} |e^{-A^{\alpha/2}(t_m-s)}P_N [F(u_{N}(s)) - F(u_{N}(t_{k})]|_{C^\delta}ds,
\end{eqnarray}
\begin{eqnarray}
 J_2 &:= & \sum_{k=0}^{m-1}\int_{t_k}^{t_{k+1}} |[e^{-A^{\alpha/2}(t_m-s)} - e^{-A^{\alpha/2}(t_m-t_k)}]P_NF(u_{N}(t_{k})|_{C^\delta}ds 
\end{eqnarray}
and 
\begin{eqnarray}
 J_3 &:= & \sum_{k=0}^{m-1}\int_{t_k}^{t_{k+1}} |e^{-A^{\alpha/2}(t_m-t_k)}P_N [F(u_{N}(t_{k})- F(u_{N, M}^{k})]|_{C^\delta}ds.
\end{eqnarray}
Now, we estimate the terms $J_1, J_2, J_3$. But, first of all, let us  remark that thanks to the uniform boundedness of $(u_N)_N$ with respect to $t$ and $N$ and of $u^m_{N, M}$ with respect to $m, N, M$, we can choose $R(\omega)= \max\{ \sup_N\sup_{t\in [0,T]}|u_N(t, \omega)|_{C^\delta}, \sup_{m, N, M}|u_{N,M}^m(\omega)|_{C^\delta}\}$, for $a.s. \omega \in \Omega$ and by Application of  Lemma \ref{main-nonlinear}, we infer the existence of a random variable $C_{F, \alpha, \delta}(\omega)$, such that 
\begin{equation}\label{unf-lipschits-1}
|F(u_N(s, \omega))-F(u_{N, M}^m(\omega))|_{H^{-\frac\alpha2}_2}\leq C_{F, \alpha, \delta}(\omega)(|u_N(s, \omega)- u_{N, M}^m(\omega)|_{C^\delta})
\end{equation}
and 
\begin{equation}\label{unf-lipschits-2}
|F(u_N(s, \omega))-F(u_{N}(t, \omega))|_{H^{-\frac\alpha2}_2}\leq C_{F, \alpha, \delta}(\omega)(|u_N(s, \omega)- u_{N}(t, \omega)|_{C^\delta}).
\end{equation}
To estimate $J_1$, we use Corollary \ref{Coro-1}, Lemma \ref{Lem-Est-1-PN}, in particular, Est.\eqref{est-P_n-H-beta} and  Est.\eqref{unf-lipschits-2}\del{Est.\eqref{main-inq-nolin}(recall here that due to the uniform boundedness of the sequences $u_N(t)$ and $u^m_{N, M}$, the constant $C_R$ is a universal random which depends only on $\alpha$ and $\delta$)}, we get
\begin{eqnarray}
 J_1 &\leq & \sum_{k=0}^{m-1}\int_{t_k}^{t_{k+1}} \Vert e^{-A^{\alpha/2}(t_m-s)} \Vert_{\mathcal{L}(H^{-\alpha/2}_2, C^\delta)} \Vert P_N\Vert_{\mathcal{L}(H^{-\frac{\alpha}{2}}_{2})}  |F(u_{N}(s)) - F(u_{N}(t_{k})|_{H^{-\frac\alpha2}_2}ds \nonumber\\  
& \leq & C_{F, \alpha, \delta}  \Big(\sum_{k=0}^{m-1}\int_{t_k}^{t_{k+1}}(t_m-s)^{-\frac{1}{2\alpha}(\alpha+2\delta+1)}| u_{N}(s) - u_{N}(t_{k})|_{C^\delta}ds\Big).
\end{eqnarray}
Thanks to the regularity of the Galerkin solution, see Theorem \ref{Theorem-disc},\del{ Proposition  \ref{Prop-reg-Galerkin-solu},} we infer that 
\begin{eqnarray}\label{eq-imd-j1}
 J_1 &\leq &  C_{F, \alpha,\delta}  \Big(\sum_{k=0}^{m-1}\int_{t_k}^{t_{k+1}}(t_m-s)^{-\frac{1}{2\alpha}(\alpha+2\delta+1)+}(s-t_k)^{(\frac{\alpha -1-2\delta}{2\alpha})-}ds\Big).
\end{eqnarray}
As $ (s-t_k)\leq \Delta t$ for all $s\in [t_k, t_{k+1}]$ and $\frac{1}{2\alpha}(\alpha+2\delta+1)<1$, we get
\begin{eqnarray}\label{J-1last}
 J_1 & \leq & C (\Delta t)^{(\frac{\alpha -1-2\delta}{2\alpha})-} \Big( \sum_{k=0}^{m-1}\int_{t_k}^{t_{k+1}}(t_m-s)^{-\frac{1}{2\alpha}(\alpha+2\delta+1)+}ds\Big) \nonumber\\ 
& \leq & C(\Delta t)^{(\frac{\alpha -1-2\delta}{2\alpha})-} \int_{0}^{t_{m}}(t_m-s)^{-\frac{1}{2\alpha}(\alpha+2\delta+1)+}ds \leq C_{\alpha,\delta}  T^{(1-\frac{1}{2\alpha}(\alpha+2\delta+1))-}(\Delta t)^{(\frac{\alpha -1-2\delta}{2\alpha})-}.\nonumber\\
& &
 \end{eqnarray}
To estimate $J_2$, we use \del{\eqref{aide-3-3}} Lemma \ref{lem-sg-regul},  Lemma \ref{Lem-Est-1-PN}, in particular, Est.\eqref{est-P_n-H-beta} and Est.\eqref{cond-lin-growth-univ}\del{Corollary \ref{Coro-linear-growth}( also recall here that due to the uniform boundedness of the sequences $u_N(t)$ and $u^m_{N, M}$, the constant $C_R$ is a universal random which depends only on $\alpha$ and $\delta$)}, then we end up with the following estiamte
\begin{eqnarray}
 J_2 & \leq & \sum_{k=0}^{m-1}\int_{t_k}^{t_{k+1}} \Vert e^{-A^{\alpha/2}(t_m-s)} - e^{-A^{\alpha/2}(t_m-t_k)}\Vert_{\mathcal{L}(H^{-\alpha/2}_2, C^\delta)} \Vert P_N\Vert_{\mathcal{L}(H^{-\frac{\alpha}{2}}_{2})} |F(u_{N}(t_{k})|_{H^{-\alpha/2}_2}ds\nonumber\\  
& \leq & C_{F, \alpha, \delta}\Big(\sum_{k=0}^{m-1}\int_{t_k}^{t_{k+1}} (s-t_k)^{\eta}(t_m-s)^{-\gamma} (1+ |u_{N}(t_{k})|_{C^\delta})ds\Big), \nonumber\\
 \end{eqnarray}
by taking $ \gamma=1-\epsilon $ and $ \eta= (\frac{\alpha-1-2\delta}{2\alpha})-\epsilon $ we deduce
\begin{eqnarray}\label{J-2last}
 J_2 & \leq & C_{F, \alpha,\delta} T^{\epsilon}(\Delta t)^{(\frac{\alpha -1-2\delta}{2\alpha})-}.
 \end{eqnarray}
Now, arguing as for the estimation of $J_1$, using  Corollary \ref{Coro-1}, Lemma \ref{Lem-Est-1-PN}, in particular, Est.\eqref{est-P_n-H-beta} and  Est.\eqref{unf-lipschits-1}\del{Est.\eqref{main-inq-nolin}(recall here that due to the uniform boundedness of the sequences $u_N(t)$ and $u^m_{N, M}$, the constant $C_R$ is a universal random which depends only on $\alpha$ and $\delta$)}, we get
\begin{eqnarray}\label{J-3last}
 J_3 &\leq & \sum_{k=0}^{m-1}\int_{t_k}^{t_{k+1}} \Vert e^{-A^{\alpha/2}(t_m-t_k)}\Vert_{\mathcal{L}(H^{-\alpha/2}_2, C^\delta)} \Vert P_N\Vert_{\mathcal{L}H^{-\frac{\alpha}{2}}_{2}}  |F(u_{N}(t_{k})- F(u_{N, M}^{k})]|_{H^{-\alpha/2}_2}ds\nonumber\\
 &\leq & C_{\alpha, \delta}\Big(\sum_{k=0}^{m-1}\int_{t_k}^{t_{k+1}} (t_m-t_k)^{-(\frac{\alpha +1+ 2\delta}{2\alpha})-}|u_{N}(t_{k})- u_{N, M}^{k}|_{C^\delta}ds\Big).  
\end{eqnarray}
Thus thanks to the estimates \eqref{second-fully}, \eqref{J-1last}, \eqref{J-2last} and \eqref{J-3last}, we get
\begin{eqnarray}
|u_N(t_m)- u_{N, M}^m|_{C^\delta} &\leq &
C_{\alpha, \delta, T}\Big(\sum_{k=0}^{m-1}(\int_{t_k}^{t_{k+1}} (t_m-t_k)^{-(\frac{\alpha +1+ 2\delta}{2\alpha})-}ds)|u_{N}(t_{k})- u_{N, M}^{k}|_{C^\delta} + (\Delta t)^{(\frac{\alpha -1-2\delta}{2\alpha})-}\Big),\nonumber\\
\end{eqnarray}
as $ t_m -s \leq t_m -t_k $, for all $ s \in [t_k,t_{k+1}] $ we get
\begin{equation}\label{fully-desc-bGrnwal}
|u_N(t_m)- u_{N, M}^m|_{C^\delta} \leq 
C_{\alpha, \delta, T}\Big(\sum_{k=0}^{m-1}(\int_{t_k}^{t_{k+1}} (t_m-s)^{-(\frac{\alpha +1+ 2\delta}{2\alpha})-}ds)|u_{N}(t_{k})- u_{N, M}^{k}|_{C^\delta} + (\Delta t)^{(\frac{\alpha -1-2\delta}{2\alpha})-}\Big)
\end{equation}
The application of the discretized version of Gronwall Lemma (see Lemma \ref{Gronwall-Lemma}) yields
\begin{eqnarray}
|u_N(t_m)- u_{N, M}^m|_{C^\delta} &\leq &
C_{\alpha, \delta, T} (\Delta t)^{(\frac{\alpha -1-2\delta}{2\alpha})-} \; exp \Big(\sum_{k=0}^{m-1}\int_{t_k}^{t_{k+1}} (t_m-s)^{-(\frac{\alpha +1+ 2\delta}{2\alpha})-} ds  \Big)  \nonumber\\
&\leq &
C_{\alpha, \delta, T}\del{ (\Delta t)^{\frac\xi\alpha}}(\Delta t)^{(\frac{\alpha -1-2\delta}{2\alpha})-} \; exp\Big( \int_{0}^{t_{m}} (t_m-s)^{-(\frac{\alpha +1+ 2\delta}{2\alpha})-} ds \Big).
\end{eqnarray}
Thanks to the fact that $ (\frac{\alpha +1+ 2\delta}{2\alpha}) < 1$, we obtain
\begin{equation}
|u_N(t_m)- u_{N, M}^m|_{C^\delta} \leq C_{\alpha, \delta, T}(\Delta t)^{(\frac{\alpha -1-2\delta}{2\alpha})-}.
\end{equation}
%%%%%%%%%%%%%%%%%%%%%%%%%%%%%%%%%%%%%%%%%%%%%%Appendix%%%%%%%%%%%%%%%%%%%%%%%%%%
\appendix

\section{Definitions and some Basic results.}\label{sec-Term-Basic-results}

We define function spaces on bounded domain  $D\subset \mathbb{R}$ and on $\mathbb{R}$. Recall that by a domain we mean an open set. The definitions and results above are still valid for domains   $D\subset \mathbb{R}^n$.\del{Let $\mathcal{D}(D):=C_{0}^{\infty}(D) $ be the set of infinitely differentiable functions with compact support on $ D $ and its dual space $ \mathcal{D}'(D) $.} We denote by $\mathcal{S}$ respectively $\mathcal{S}'$ the Schwartz respectively distribution spaces and by $ \mathcal{F}$ respectively $\mathcal{F}^{-1}$ the Fourier respectively the inverse Fourier transforms. Let $ 0< p,q\leq \infty, \;\; s\in \mathbb{R}   \del{1\leq p \leq\infty, \;\; 0<s\neq integer}$. For simplicity reasons, we somtimes restrict these parameters for the required cases. 

\noindent {\bf Lebesgue Space}. Let $D\subseteq \mathbb{R}$, $$ 
L^p:= L^p(D):=\{f \;\;\text{measurable}\;\;  s.t. \;\; |f|^p_{L^p}:= \int_D|f(x)|^pdx<\infty\}, \;\; 0< p< \infty.$$

\begin{equation*}
L^\infty:= L^\infty(D):=\{f \;\;\text{measurable}\;\;  s.t. \;\; |f|_{L^\infty}:= esssup_{D} |f(x)|<\infty\}.
\end{equation*}

\subsection{Function spaces on $\mathbb{R}$} (For simplicity we omit to mention $\mathbb{R}$ in notations.)

\vspace{0.25cm}
\noindent {\bf Sobolev spaces}. For $1\leq p\leq \infty$ and $m \in \mathbb{N}$,
$$ W^m_{p}:= \{f\in L^p, \;\; s.t. \;\;  |f|_{W^m_{p}}^{p}:= \sum_{k=0}^m|D^{k}f|_{L^p}^{p} <\infty \},$$
where $ D^{k}f $ represents the derivative of $ f $ of order $ k $ in the distributional sense.

\vspace{0.25cm}

\noindent {\bf Fractional Sobolev spaces}. For  $1\leq p< \infty \;\; \text{and} \;\; 0<s\neq integer,$ 
with $ \; [s], \{s\}$ are respectively the integer and the fractional parts of s.
$$  W^s_{p}:= \{f\in W^{[s]}_{p}, \;\; s.t. \;\; |f|_{W^{s}_{p}}^{p}:= |f|_{W^{[s]}_{p}}^{p} + \sum_{k=0}^{[s]} \int_{\mathbb{R}} \int_{\mathbb{R}}  \frac{|D^{k}f(x)-D^{k}f(y)|^p}{|x-y|^{1+\{s\}p}} dx \; dy <\infty \}.$$
%%%%%%%
\del{\vspace{0.25cm}
   For $ s < 0, \; 1<p<\infty$, we define  $W^{s}_{p}(D)$ as the dual of $W^{-s}_{q,0}(D)$, i.e.
\begin{equation*}
W^{s}_{p}(D):= (W^{-s}_{q,0}(D))^{'},
\end{equation*}
where $ W^{-s}_{q,0}(D) $ is the closure of $ C_{0}^{\infty}(D) $ in the space $ W^{-s}_{q}(D) $,   $ 1 < q < \infty $ is  the conjugate of $p$.
}
%%%%%%
\begin{remark}
Let us mention here that for $ p=2 $, the space $ W^s_{2} $ is a Hilbert space. We  denote it by $ H^{s}_{2} $.
\end{remark}

In order to introduce Besov spaces we need to define special systems of functions.
\begin{defn}\label{Def-Syst}\cite[Definition 1, P7]{Runst-Sickel}.
Let $ \phi=(\phi_{j})_{j=0}^{\infty}\subset  \mathcal{S}$ be a system such that 
\begin{enumerate}
\item for every $ x \in \mathbb{R} $, $ \sum_{j=0}^{\infty} \phi_j(x)=1 $,
\item there exist constants $ c_{1}, c_{2}, c_{3} > 0 $ with
\begin{equation*}
supp \phi_0 \subset \lbrace x \in \mathbb{R}: \vert x \vert \leq c_1 \rbrace,
\end{equation*}
and
\begin{equation*}
supp \phi_j \subset \lbrace x \in \mathbb{R}:  c_2 \; 2^{j-1} \vert x \vert \leq c_3 \; 2^{j+1} \rbrace, \; \text{for} \; j=1,2,...
\end{equation*}
\item for every nonnegative integer $ k $ there exists $ c_k > 0 $ s.t,
\begin{equation*}
\sup_{x \in \mathbb{R}} \; \sup_{j=0,1,...} \; 2^{jk} \vert D^{k} \phi_j(x) \vert \leq c_k.
\end{equation*}
\end{enumerate}
\end{defn}

\noindent $\text{{\bf Besov spaces}. For}\;\; s\in \mathbb{R} \;\;  \text{and} \;\; 0< p,q\leq \infty,$ see  also \cite[Convention 1. P11]{Runst-Sickel}.
 $$ B^s_{p q}:= \lbrace f\in \mathcal{S}', \;\; s.t. \;  |f|_{B^s_{p q}}^{q}:= | 2^{sj} \mathcal{F}^{-1}[\phi_j\mathcal{F}f](.)|_{l^q(L^p)} <\infty \rbrace,\;\; \text{where} \; (\phi_{j})_{j=0}^{\infty}\; \text{is given by  Def.\ref{Def-Syst}}.$$

\del{For $ 0> s \neq \text{integer}, \; 1<p<\infty, \; 0<q<\infty$, we define  $B^{s}_{pq}$ as the dual of $W^{-s}_{p'q'}$, where $p$ respectively $q$ is the conjugate of $p'$ respectively $q'$. In particular, we have  for $ 0> s \neq \text{integer}, \; 1<p<\infty$,  $W^{s}_{p}$ is the dual of $W^{-s}_{p'}$, where $p'$ is the conjugate of $p$.}

\del{$\text{{\bf Besov spaces}. For}\;\; s\in \mathbb{R} \;\;  \text{and} \;\; 0< p,q\leq \infty ,$
 $$ B^s_{p q}:= \lbrace f\in \mathcal{S}', \;\; s.t. \;  |f|_{B^s_{p q}}^{q}:= \sum_{j=0}^{\infty} 2^{sjq} \vert \mathcal{F}^{-1}(\phi_j\mathcal{F}f)(.)|_{L^p}^{q} <\infty \rbrace,$$}

\noindent {\bf Space of continuous functions.} \begin{equation*}
C:=\lbrace f \; \text{bounded and continuous, s.t.} \; \vert f \vert_{C}:= \sup_{x \in \mathbb{R}} \vert f(x) \vert < \infty \rbrace.
\end{equation*}

$$\text{{\bf H\"older spaces}. For}\;  \delta \in (0, 1), \;\;   
C^\delta:=\{ f\in C, s.t. \;\; |f|_{C^\delta}:=|f|_{C}+ \sup_{x, y \in \mathbb{R}, \; \\ x\neq y}\frac{|f(x)-f(y)|}{|x-y|^\delta}<\infty\}.$$

$$\text{{\bf Zygmund spaces}. For }  \; s \in (0,1), \;\; \mathcal{C}^s:= \{f\in C \; s.t. \; 
|f|_{\mathcal{C}^\delta}:=|f|_{C}+ \sup_{h \in \mathbb{R}, h\neq 0} |h|^{-s}|\Delta_h^2f|_{C}<\infty \}, \; $$
with $\Delta_h^2f(x):= \sum_{l=0}^2 (-1)^lC_l^2f(x+(2-l)h).$

\begin{lem}\label{lem-ess-general} 
We have:  
\begin{itemize}
\item The identities, see e.g. \cite[P.14]{Runst-Sickel} and \cite{Triebel-83}.
\begin{itemize}
\item  For $0<s\neq$ integer, $\;\; \; C^s= B^s_{\infty\; \infty}= \mathcal{C}^s$.
\item  For $0<s\neq$ integer and $ 1\leq p < \infty $, $\;\; \;  B^s_{p p}= W^s_p$.
\item  For $ 1\leq p \leq \infty $, $ W^0_p=L^p$.
\end{itemize}
\item The continuous embeddings: 

\begin{equation}\label{Imbd-2}
B^{s}_{p q_0} \hookrightarrow B^{s}_{p q_1},
\end{equation}
for $ s >0 $, $ 0<p\leq \infty $ and $ 0 < q_0 < q_1 \leq \infty $, see e.g. \cite[Prop 2.2.1, P.29]{Runst-Sickel}. 
\begin{equation}\label{Imbd-1}
B^{s_0}_{p_0 q_0} \hookrightarrow B^{s_1}_{p_1 q_1},
\end{equation}
provided $ s_1<s_0$, $ s_0-\frac{1}{p_0}> s_1-\frac{1}{p_1}$ and $ p_0 \leq p_1 $, see e.g. \cite[Remark 2, P.31]{Runst-Sickel}.
\item As a consequence of Embedding \eqref{Imbd-1}, we have for $ s_1<s_0$
\begin{equation}\label{Imbd-req-1}
B^{s_0}_{\infty \infty} \hookrightarrow B^{s_1}_{2 2},
\end{equation}
\begin{equation}\label{Imbd-req-2}
\mathcal{C}^{s_0} \del{= C^{s_0}} \hookrightarrow H_{2}^{s_1},
\end{equation}
and for $ s_0>\frac{2-\alpha}{2}$ and $ s_0\neq $integer,
\begin{equation}\label{Imbd-req-3}
\mathcal{C}^{s_0} = C^{s_0} \hookrightarrow H_{2}^{1-\frac\alpha2}.
\end{equation}
\item  \cite[Theorem 2.8.3, Ps. 145-146]{Triebel-83}, for $0<p,q\leq \infty, \;\;  s>\frac{1}{p}$, $ B^{s}_{p q}$ is a multiplication algebra. In particular, for $s>0$, $ \mathcal{C}^{s}$ is a multiplication algebra.
\end{itemize}
\end{lem}

\begin{theorem}\label{aide-theorem1}\cite[Theorem 2.8.2]{Triebel-83}.
Let $ s \in \mathbb{R} $, $ 0 < p , q \leq \infty $ and $ \delta > max (s , (\frac{1}{min(p,1)}-1)-s) $. Then for every $ f \in \mathcal{C}^{\delta} $ and every  $ g \in  B_{p,q}^{s}$ there exists $ C>0 $ such that
\begin{equation}
\vert f \; g \vert_{B_{p,q}^{s}} \leq C \; \vert f \vert_{\mathcal{C}^{\delta}} \; \vert g \vert_{B_{p,q}^{s}}.
\end{equation}
\end{theorem}

\subsection{Function spaces on domains}
\begin{defn}{\cite[Definition 2.4.1.2\del{, P. 74}]{Runst-Sickel}}
For $ s \in \mathbb{R} $ and $ 0<p,q\leq\infty $,
\begin{equation*}
B^s_{p q}(D):= \lbrace f\in \mathcal{D}'(D), \;\ \; \exists \; g \in B^s_{p q}\; with \; g|_{D}=f, \; \; s.t. \; \vert f \vert_{B^s_{p q}(D)}:=  \inf \vert g \vert_{B^s_{p q}} <\infty \rbrace.
\end{equation*}
\end{defn}

\begin{defn} \cite[Section 5.3.2]{Triebel-92} \& \cite[Section 1.10.3\del{, P. 73}]{Triebel-92} Let $ D $ be a bounded $ C^{\infty} $-domain, then for $ s>0$, we define $\mathcal{C}^s(D):= B^s_{\infty\infty}(D).$ In particular, for $ s \in (0,1) $
\begin{equation*}
\mathcal{C}^s(D):= \{f\in C(D) \; s.t. \; 
|f|_{\mathcal{C}^\delta(D)}:=|f|_{C(D)}+ \sup_{h \in \mathbb{R}, h\neq 0} |h|^{-s}|\Delta_h^2 f(.,D)|_{C(D)}<\infty \},
\end{equation*}
with
\begin{equation*}
\Delta_h^2f(x,D):=\Big\{
\begin{array}{lr}
\Delta_h^2f(x), \;\;\text{if}\;\;\;\; x+jh \in D \; \text{for}\; j=0,1,2,\\
0, \; \; \; \; \text{otherwise}.
\end{array}
\end{equation*}
\end{defn}
\del{Let us mention here that the above results are still valid for bounded domains, see in particular, \cite[Section 3.3.2]{Triebel-83}.In particular,}
\begin{remark}\label{Rem-New}
The above results are still valid for bounded domains, see, e.g. \cite[Section 3.3.2]{Triebel-83} and \cite[Section 5.4]{Triebel-92}.
\end{remark}
Moreover, we have

\begin{coro}\label{coro-aide-theorem1}
Let $ 0 < s \neq \; integer$ and $\delta>s$ (in our study $\delta\in (0, 1)$), then for $ f \in \mathcal{C}^{\delta}(D) $ and every  $ g \in  H_{2}^{s}(D) $ there exists a positive constant $ C $ such that
\begin{equation}
\vert f \; g \vert_{H_{2}^{s}} \leq C \; \vert f \vert_{\mathcal{C}^{\delta}} \; \vert g \vert_{H_{2}^{s}}.
\end{equation}
\end{coro}
\begin{proof}
The ressult is obtained by application of Theorem \ref{aide-theorem1} and  \cite[Proposition 2.1.2  P.14]{Runst-Sickel}. \del{from which  we have $ B_{2,2}^{s}= H^{s}_2$ and  $ C^{\delta}= \mathcal{C}^{\delta}$}
\end{proof}

\del{\begin{theorem}\label{aide-theorem1}\cite[Section 3.3.2]{Triebel-83}.
Let $ s \in \mathbb{R} $, $ 0 < p , q \leq \infty $ and $ \delta > max (s , (\frac{1}{min(p,1)}-1)-s) $. Then for every $ f \in \mathcal{C}^{\delta}(D) $ and every  $ g \in  B_{p,q}^{s}(D) $ there exists a positive constant $ C $ such that
\begin{equation}
\vert f \; g \vert_{B_{p,q}^{s}} \leq C \; \vert f \vert_{\mathcal{C}^{\delta}} \; \vert g \vert_{B_{p,q}^{s}}.
\end{equation}
\end{theorem}
}

\begin{theorem}\label{main-sob-embedding}\cite[Theorem 8.2]{Nezza}.
Let $ D \subseteq \mathbb{R}  $ be an extension domain for $ W^{r}_{p} $ with no external cups and let $ p \in [1, \infty) $, $ r \in (0,1) $ s.t.  $ \;r > \frac{1}{p} $. Then $\; \exists \;\;  C_{D, p, r} > 0 $ s.t.
\begin{equation}
\vert f \vert_{C^{r-\frac{1}{p}}} \leq C_{D, p,r} \left( \vert f \vert_{L^{p}}^{p} + \int_{D} \int_{D} \frac{\vert f(x)-f(y)\vert^{p}}{\vert x-y \vert^{1+rp}} dx \; dy \right)^{\frac{1}{p}}, \; \; \text{for any} \; f \in L^{p}(D).
\end{equation}
\end{theorem}

\begin{lem}\label{Lemma-DeBre} \cite[Lemma 2.11]{BrzezniakDebbi1} For each $\alpha>\frac32$ there exsits a constant $C_\alpha>0$ such that for all $t>0$ and for any bounded and strongly-measurable function $v: (0, t) \rightarrow L^1(0, 1)$ the following inequality holds 
\begin{equation}
\int_0^t| e^{-A^{\alpha/2} (t-s)}\frac{\partial v}{\partial x}(s)|_{L^2}ds\leq C_\alpha t^{1-\frac3{2\alpha}}\sup_{s\leq t}|v(s)|_{L^1}.
\end{equation}
\end{lem}
%%%%%%%%%%%%%%%%%%%%%%%%%%%%%%%%%%%
\del{How to include this result in remark \ref{Rem-New} or in Lemma  \ref{lem-ess-general},
\begin{prop}\label{sob-embedding}
Let $ \Omega $ be a bounded $ C^{\infty} $-domain in $ \mathbb{R} $ (see for instance Definition 1 p. 73 from section 2.4.1 of \cite{Runst-Sickel}). Let $ s, \epsilon > 0 $, $ 0 < p ,q \leq \infty $ and suppose $ 0 < p_{0} < p_{1} \leq \infty $, $ 0 < q_{0} < q_{1} \leq \infty $. Then it holds
\begin{equation}
B^{s}_ {p , q_{0}}(\Omega) \hookrightarrow B^{s}_ {p , q_{1}}(\Omega), \; \; \; \; B^{s}_ {p_{1} , q}(\Omega) \hookrightarrow B^{s}_ {p_{0} , q}(\Omega) \; and \; \; \; B^{s + \epsilon}_ {p , \infty}(\Omega) \hookrightarrow B^{s}_ {p , q}(\Omega),
\end{equation}
where $ \hookrightarrow $ denotes continuous embedding.(You have already used this notation!!!)
\end{prop}

More precisely,  we work in the space $H:=L^2(0, 1)$. It is well known that $H:=L^2(0, 1)$ is a separable Hilbert space. Here we take the basis of the eigenvectors of the Laplacian endowed by the Dirichlet boundary conditions;  $e_k(\cdot):=\sin (2k\pi \cdot)$, hence $ \lambda_k:= (k\pi)^2$. We end this part of  the semi group properties  by the following trivial lemmas;}
%%%%%%%%%%%%%%%%%%%%%%%%%%%%%%%%%%%%%
\begin{lem}\label{lem-e-k-holder}
Let $ \delta\in (0, 1]$. There exists a constant $ c_\delta>0$ (independent of $k$) such that
\begin{equation}\label{est-k-holder}
|e_k|_{C^\delta} \leq c_\delta k^\delta.
\end{equation}
\end{lem}
\begin{lem}\label{lem-elementary-1}
 $ \forall \gamma > 0, \exists \;\; C_{\gamma} > 0, s.t \; x^{ \gamma}e^{-x} \leq C_{ \gamma}$.
 \end{lem}

\begin{lem}\label{lem-elementary-2}
 $ \forall \eta \in (0,1), \exists \;\;  C_{\eta} > 0, s.t \; x^{-\eta}(1-e^{-x}) \leq C_{\eta}$.
 \end{lem}

\begin{lem}\label{appendix-A-1}\cite[Lemma 9]{Jentzen-thesis}.
Let $ \eta \in (0,1) $, then $\;\; \int_{0}^{1}\int_{0}^{1} \vert x-y \vert^{-\eta} dx dy \leq \frac{3}{1-\eta} $.
\end{lem}

\begin{lem}\label{appendix-A-4}\cite[Lemma 2.1]{Kloeden1}. Let $ \tau > 0$,  $ (C_p)_{p \geq 1}  \subset [0,\infty )$ and let $ (Z_n)_{n \in \mathbb{N}} $ be a sequence of random variables such that 
\begin{equation}
 (\mathbb{E}\vert Z_n \vert^{p} )^{\frac{1}{p}} \leq C_p \; n^{-\tau} ,
\end{equation}
for all $ p \geq 1 $ and all $ n \in \mathbb{N} $. Then
\begin{equation}
\mathbb{P} \left( \sup_{n \in \mathbb{N}} ( n^{\tau -\epsilon} \vert Z_n \vert ) < \infty  \right) = 1,
\end{equation}
for all $ \epsilon \in (0,\tau) $. 
\end{lem} 

\begin{theorem}\label{append-lem-gauss}\cite[Lemma 10]{Jentzen-thesis}. 
Let $ Y: \Omega \rightarrow \mathbb{R}$ be a $ \mathcal{F}/\mathcal{B}(\mathbb{R})-$measurable mapping that is centered and normal distributed. Then for every $p\in \mathbb{N}$, 
\begin{equation}
\mathbb{E}|Y|^p\leq p!(\mathbb{E}|Y|^2)^{\frac p2}.
\end{equation}
\end{theorem}

\begin{lem}\label{appendix-A-2}\cite[Lemma 12.]{Jentzen-thesis}
Let $ W: [0,T]\times \Omega \rightarrow \mathbb{R} $ be a standard Brownian motion. Then for every $ r \in [0,1]$, $ \lambda \in (0,\infty)$ and $ t_{1},t_{2} \in [0,T] $,

\begin{equation}
\mathbb{E}\left( \vert \int_{0}^{t_{2}} e^{-\lambda(t_{2}-s)} dW(s) - \int_{0}^{t_{1}} e^{-\lambda(t_{1}-s)} dW(s) \vert^{2} \right) \leq \lambda^{r-1} \vert t_{2}-t_{1} \vert^{r} , 
\end{equation}
\end{lem}

\begin{lem}\label{Gronwall-Lemma} -Discrete Gronwall Lemma \cite{Holte}-.
Let $ (x_{n})_{n \in \mathbb{N}} $ and $(y_{n})_{n \in \mathbb{N}} $ be positive sequences and $ C $ a positive constant. If for any $ n \geq 0 $
$$ x_{n} \leq C + \sum_{k=0}^{n-1} x_{k} y_{k}, $$ 
then
$$ x_{n} \leq C exp(\sum_{k=0}^{n-1} y_{k}). $$
\end{lem}


\begin{thebibliography}{09}


\del{\bibitem{Adams-Hedberg-99} Adams R. I. and  Hedberg L. I. \textit{Function Spaces and Potential Theory}, A Series of Comprehension Studies of Mathematics, V 314, Springer (1999).}

\bibitem{Alabert-Gyongy} Alabert A. and Gy$ \ddot{o} $ngy I. \textit{On numerical approximation of stochastic Burgers' equation}, From stochastic calculus to mathematical finance, Springer Berlin 1-15 (2006).

\bibitem{Alibaud-no-unquness-burgers}  Alibaud N. and Andreianov B. \textit{Non-uniqueness of weak solutions for the fractal Burgers equation.}
 Ann. Inst. H. Poincar\'e Anal. Non Lin\'eaire 27  no. 4, 997-1016  (2010).

\bibitem{Blomker-Jentzen} Bl$ \ddot{o} $mker D and Jentzen A. \textit{Galerkin approximations for the stochastic Burgers equation}. SIAM J. Numer. Anal. 51, 694-715 (2013). 


\bibitem{Bilerandal} Biler P., Funaki, T. and  Woyczynski W. A.
\textit{Fractal {B}urgers' equations.} J. {D}ifferential {E}quations
148, 9-46 (1998).


\bibitem{BrzezniakDebbiGoldy} Brze{\'z}niak Z., Debbi L. and Goldys B. \textit{Ergodic properties of fractional stochastic Burgers equations}. Global and stochastic analysis, Vol. 1 n 2, 149-174 (2011).

\bibitem{BrzezniakDebbi1} Brze{\'z}niak Z. and Debbi L. \textit{On Stochastic Burgers Equation Driven by a Fractional Power of the Laplacian and space-time white noise. } Stochastic Differential Equation: Theory and Applications, A volume in Honor of Professor Boris L. Rozovskii. Edited by P. H. Baxendale and S. V. Lototsky, 135-167 (2007).

\bibitem{Caffarelli-Vasseur2010}  Caffarelli L. A. and  Vasseur A. \textit{Drift diffusion equations with fractional diffusion and the
quasi-geostrophic equation}.  Ann. of Math. 2,  171  no. 3, pp.
1903-1930 (2010).

\bibitem{Cardon-Weber-Millet}  Cardon-Weber C. and Millet A.
\textit{A Support Theorem for a Generalized Burgers SPDE.} Potential Analysis 15, 361-408. (2001).



\bibitem{cordoba-cordoba-Fontelos-05} C\'ordoba A. C\'ordoba D. and  Fontelos M. A. \textit{Formation of singularities for a transport
equation with nonlocal velocity.} Ann. of Math. (2) 162, no. 3, 1377-1389 (2005).


\del{\bibitem{Millet-Chueshov-Hydranamycs-2NS-10}  Chueshov I. and  Millet A. \textit{Stochastic 2D hydrodynamical type systems: well posedness and large deviations.} Appl. Math. Optim. 61  no. 3, 379-420  (2010).}

\del{\bibitem{Conus-Weak-14} Conus D., Jentzen A. and Kurniawan R.\textit{ Weak convergence rates of spectral Galerkin approximations for SPDEs with nonlinear diffusion coefficients.} Ann. Appl. Probab.


\bibitem{Cox--Neervan-Arxiv-17}   Cox S.,  Hutzenthaler M., Jentzen A.,  van Neerven J. and  Welti T. \textit{Convergence in H ̈older norms with applications to Monte Carlo methods in infinite dimensions.} Arxiv https://arxiv.org/pdf/1605.00856.pdf  (2017).


\bibitem{DaPrato-Debussche-Cahn-96}  Da Prato, G. and Debussche, A. \textit{Stochastic Cahn-Hilliard equation}. Nonlinear Anal. 26, no. 2, 241-263  (1996).} 

\bibitem{DapratoDebusscheTemam94} Daprato D. G, Debussche A. and Temam R. \textit{Stochastic Burgers'equation.} NoDEA 1. 389-402 (1994).

\bibitem{DapratoZabczyk96} Da Prato G. and Zabczyk J. \textit{Ergodicity for Infinite Dimensional Systems.} Combridge university press (1996). 

\del{\bibitem{Daprato} Da Prato G. and Zabczyk J. \textit{Stochastic equations in infinite dimensions}, Springer, Combridge university press. ISBN 0-521-38529-6. (1992).}

\bibitem{Debbi-scalar-active} Debbi L. \textit{Fractional stochastic active scalar equations generalizing the multi-d-Quasi-Geostrophic and the 2D-Navier-stokes equations.} Submitted to Annals of Probability in 03-09-2012, Short version: arXiv: 1208.2932v2 [math-AP] (2012).

\bibitem{Debbi-scalar-active-R-d}  Debbi L. \textit{Fractional stochastic  active scalar quations Generalizing the Multi-D-Quasi-Geostrophic \&
2D-Navier-Stokes Equations-The general case-.} Preprint.

\bibitem{Debbi-FNS}  Debbi L.  \textit{Well-posedness of the multidimensional fractional stochastic Navier-Stokes equations on the Torus and on bounded domains.} Journal of Mathematical Fluid Mechanics, March 2016, Volume 18, Issue 1, pp 25-69.


\del{\bibitem{DebbiThesis} Debbi L. \textit{Deterministic and Stochastic Partial Differential Equations Driven by Fractional Operators.} Thesis, University Henri Poincar\'ee, Nancy1, (2006).}

\bibitem{DebbiL2Solution} Debbi L. \textit{On the $L^2-$solution of Fractional  Stochastic Partial Differential Equations Driven by Fractional Operators.} arXiv: 1102.4715v1.

\del{\bibitem{Debbi-Arab} Debbi L. and Arab Z. \textit{Numerical approximation for some fractional stochastic partial differetial equatuions}. Preprint.}

\bibitem{DebbiDozzi1} Debbi L. and Dozzi M. \textit{On the solutions of nonlinear stochastic fractional partial differential equations in one spatial dimension.} Stochastic Processes and their Applications, Vol. 115, N 11, pp. 1764-1781  (2005).

\bibitem{DebbiDozzi2}  Debbi L. and  Dozzi M. \textit{On a space discretization scheme for the fractional stochastic heat equations.} arXiv:1102.4689v1  (2011).

\bibitem{Nezza} Di Nezza E.,  Palatucci G. and  Valdinoci E. \textit{ Hitchhiker's guide to the fractional Sobolev spaces.} Bull. Sci. Math. 136, no. 5, 521-573  (2012).


\bibitem{Duan-Lv-conservation-law} Duan J., Gao H., Lv G. and Wu J. L. \textit{On a stochastic nonlocal conservation law in a bounded domain.} Bull. Sci. math. (2016), http://dx.doi.org/10.1016/j.bulsci.2016.03.003.


\bibitem{Duan-Lv-Burgers} Duan J. and  Lv G. \textit{ Martingale and Weak Solutions for a Stochastic Nonlocal Burgers Equation on Bounded Intervals.} http://arxiv.org/abs/1410.7691v1 



\del{\bibitem{Flandoli-Gatarek-95} Flandoli F. and Gatarek D. \textit{
Martingale and stationary solutions for stochastic Navier-Stokes equations.} Probab. Theory Related Fields 102 no. 3, 367-391 (1995).}

\bibitem{Flandoli-Schmalfub-99} Flandoli F. and Scmalfuss B. \textit{Weak Solutions and Attractors for Three-Dimensional Navier-Stokes Equations with Nonregular Force}. Journal of Dynamics and Differential Equations, Vol. 11, No. 2,  (1999).


\bibitem{Gyongy-Gerencser}  Gerencs\`er M. and Gy\"ongy I.
\textit{ Finite Difference Schemes for Stochastic Partial Differential Equations in Sobolev Spaces}. Applied Mathematics \& Optimization, Volume 72, Issue 1,  77-100 (2015).


\bibitem{Gugg--Niggemann} Gugg C., Kielh\"ofer H. and Niggemann M. \textit{On the approximation of the stochastic Burgers equation.} Commun. Math. Phys, 230, 181-199  (2002).

\del{\bibitem{Gyongy-1} Gyongy I.   \textit{Lattice Approximations for Stochastic Quasi-Linear Paraboloic Partial differential Equations Driven by Space-Time White Noise I.} Potential Analysis 9, 1-25 (1998).}


\bibitem{Gyongy-98-B-type}  Gy\"ongy I. \textit{Existence and uniqueness results for semi-linear stochastic partial differential equations.} Stochastic Process. Appl. 73 (1998), 271-299. 

\bibitem{Gyongy-Krylov-10} Gy\"ongy I.  and Krylov N. \textit{
Accelerated Finite Difference Schemes for Linear Stochastic Partial Differential Equations in the Whole Space.} SIAM J. Math. Anal., 42(5), 2275–2296  (2010).




\bibitem{Gyongy-Millet-09} Gy\"ongy I. and Millet A.  \textit{Rate of Convergence of space Time Approximations for Stochastic Evolution  Equations.} 
Potential Analysis 30, 29-64 (2009).


\bibitem{Gyongy-Millet-07} Gy\"ongy I. and Millet A.  \textit{Rate of Convergence of Implicit Approximations for Stochastic Evolution  Equations.} 
Stochastic Differential Equation: Theory and Applications, A volume in Honor of Professor Boris L. Rozovskii. Edited by P. H. Baxendale and S. V. Lototsky, 281-310 (2007).

 
 \bibitem{Gyongy-Nualart-Burgers-99} Gy\"ongy I. and Nualart D.   \textit{On the Stochastic Burgers' Equations in the real line.} Ann. Probab., 27(2), 782-802  (1999).
 
 
 \bibitem{Gyongy-Nualart} Gy\"ongy I. and Nualart D.   \textit{Implicite schemes for Stochastic Paraboloic Partial differential Equations Driven by Space-Time White Noise.} Potential Analysis 7, 725-757  (1997).


\bibitem{Hairer-Matetski-16}  Hairer M and Matetski K. \textit{Optimal rate of convergence for stochastic Burgers-type equations.} Stoch. PDE: Anal. Comp. V 4, 402-437  (2016).

\bibitem{Hairer-AP-N-SDE}  Hairer M, Hutzenthaler M. and Jentzen A. \textit{Loss of regularity for Kolmogorov equations.} Ann. Probab. V. 43, No. 2, 468-527  (2015).

\bibitem{Hairer-Weber-13}  Hairer M and Weber H. \textit{Rough Burgers-like equations with multiplicative noise.} Probab. Theory Relat. fields 155, 71-126  (2013).


\bibitem{Hairer-Voss} Hairer M and Voss J. \textit{Approximations to the stochastic Burgers equation.} JNonlinear Sci. 2, 897-920  (2011).




\bibitem{Holte} Holte J. M., \textit{Discrete Gronwall lemma and applications}. MAA north central section meeting at UND. Octobre (2009).


\bibitem{Jentzen-Pertub-Arx} Hutzenthaler M. and Jentzen A. \textit{ On a perturbation theory and on strong convergence rates for stochastic ordinary and partial differential equations with non-globally monotone coefficients.} https://arxiv.org/abs/1401.0295  (2014).


\bibitem{Jentzen-thesis} Jentzen A. \textit{Taylor Expansions for Stochastic Partial Differential Equations}, Dissertation, Johann Wolfgang Goethe-University, Frankfurt am Main, Germany (2009).

\bibitem{Jentzen-Kloeden-Winkel} Jentzen A., Kloeden P. and Winkel G. \textit{Efficient simulation of nonlinear parabolic SPDEs with additive noise}, The Annals of Applied Probability. Vol.21, No. 3, 908-950 (2011).

\bibitem{Jentzen-Kloeden} Jentzen A. and Kloeden P.  \textit{Overcoming the order barrier in the numerical approximation of stochastic partial differential equations with additive space-time noise}. Proceedings of the Royal Society, London Ser.A Math. Phys. Eng. Sci. 465  649-667 MR2471778.

\bibitem{Kiselev-Nazarov-Fractal-Burgers-08} Kiselev A., Nazarov F. and Shterenberg R. \textit{Blow up and regularity for fractal Burgers equation.} Dyn.
Partial Differ. Equ. 5 no. 3, 211-240 (2008).

\bibitem{Kloeden1} Kloeden P. E. and Neuenkirch A. \textit{The pathwise convergence of approximation schemes for stochastic differential equations}, LMS J. Comput. Math., pp. 235-253, 10 (2007). 


\bibitem{Kruse-Optimal-14} Kruse R.  \textit{Optimal error estimates of Galerkin finite element methods for stochastic partial differential equations with multiplicative noise}.
IMA Journal of Numerical Analysis, Volume 34, Issue 1, 217–251 (2014).


\bibitem{SugKak} Kukatani T. and Sugimoto N.  \textit{Generalized {B}urgers Equations for Nonlinear Viscoelastic Waves.} Wave Motion 7, 447-458 (1985).




\bibitem{Pazy-83} Pazy A. \textit{Semigroups of linear operators and applications to partial
differential equations.} Applied Mathematical Sciences 44.
Springer-Verlag, New York (1983).


\bibitem{PeszatZabczykBook07} Peszat S and Zabczyk J. \textit{Stochastic partial
differential equations with L\'evy noise.} Cambridge University Press. (2007).

\bibitem{Printems} Printems. J. \textit{On the discretization in time of parabolic SPDE}, Mathematical Modelling and Numerical Analysis, (2001).



\bibitem{Runst-Sickel} Runst T. and Sickel W. \textit{Sobolev spaces of fractional order, Nemytskijoperators, and nonlinear partial differential equations}. vol. 3, de Gruyter Series in nonlinear Analysis and Applications, ISBN 3-11-015113-8 (1996). 

\bibitem{Sug} Sugimoto N.\textit{Generalized Burgers equations and Fractional Calculus.}
Nonlinear Wave Motion.(A. Jeffery, Ed) 162-179 (1991).


\bibitem{Sug-Frac-cal-89} Sugimoto N. \textit{ "Generalized'' Burgers equations and fractional calculus.}
 Nonlinear wave motion,  Pitman Monogr. Surveys Pure Appl. Math., 43, Longman Sci. Tech., Harlow, 162-179  (1989).

\bibitem{Woyczynski-Num-approx-fract-05} Stanescu D., Kim D. and  Woyczynski W. A. \textit{Numerical study of interacting particles approximation for integro-differential equations.} J. Comput. Phys.206 no. 2, 706-726 (2005).


\bibitem{Taylor-PDE-I} Taylor M. E. \textit{Partial differential equations I. Basic Theory.} Applied Mathematical Sciences V 115.  Springer 1996.

\del{\bibitem{Taylor-PDE-II} Taylor M. E. \textit{Partial differential equations II. Qualitative studies of linear equations.}
Applied Mathematical Sciences V 116.  Springer 1996.}

\bibitem{Taylor-PDE-III} Taylor M. E. \textit{Partial differential equations III. Nonlinear equations.} Applied Mathematical Sciences V 117.  Springer 1997.

\del{\bibitem{Temam-NS-Main-79} Temam R. \textit{Navier-Stokes equations. Theory and numerical analysis.} Revised edition. With an appendix by F. Thomasset. Studies in Mathematics and its Applications 2. North-Holland Publishing Co. Amsterdam-New York  (1979).}

\bibitem{Temam-NSE} Temam R. \textit{Navier-Stokes equations and nonlinear functional analysis}. InL CBMS-NSF Regional Conference Series in Applied Mathematics, 2nd edn, vol. 66. Society for Industrial and Applied Mathematics (SIAM), Philadelphia (1995).

\bibitem{Triebel-83} Triebel H. \textit{Theory of function spaces}. vol. 78, Monographs in Mathematics,(Birkh$ \ddot{a} $user Verlag, Basel), (1983).

\bibitem{Triebel-92} Triebel H. \textit{Theory of function spaces II}. vol. 84, Monographs in Mathematics,(Birkh$ \ddot{a} $user Verlag, Basel), (1992).


\bibitem{TrumanWu-06} Truman A. and Wu J.L.\textit{Fractal Burgers' equation driven by L\'evy
noise.}  Stochastic partial differential equations and
applications. VII,  295-310, Lect. Notes Pure Appl. Math. 245
Chapman \& Hall/CRC Boca Raton FL 2006.

\bibitem{Neerven-Evolution-Eq-08} Van Neerven J., M. C. Veraar and L. Weis
\textit{Stochastic evolution equations in UMD Banach spaces.} J. Funct. Anal. 255, no. 4, 940-993  (2008).

\bibitem{Xie-frac-Burg} Wu J. L. and Xie B. \textit{On a Burgers type nonlinear equation perturbed by a pure jump Levy noise in $\mathbb{R}^d.$} Bull. Sci. math. 136. 481-506  (2012).

\bibitem{Yosida} Yosida K. \textit{Functional Analysis}, Springer-Verlag Berlin Heidelberg New York (1971).

\end{thebibliography}
\end{document}